\documentclass[reqno,english]{smfart}

\usepackage[latin1]{inputenc}
\usepackage[francais, english]{babel}
\usepackage{amsmath,amsfonts,amssymb,amsthm,fancyhdr}
\usepackage{bm}
\usepackage{graphicx}
\usepackage{color}
\usepackage{enumitem}
\usepackage{natbib}

\theoremstyle{plain}

\newtheorem{theorem}{Theorem}[section]
\newtheorem{corollary}[theorem]{Corollary}
\newtheorem{definition}[theorem]{Definition}
\newtheorem{example}[theorem]{Example}
\newtheorem{lemma}[theorem]{Lemma}

\newtheorem{proposition}[theorem]{Proposition}
\newtheorem{remark}[theorem]{Remark}

\numberwithin{equation}{section}

\newcommand{\re}{{\mathbb R}}
\newcommand{\cn}{{\mathbb C}}

\newcommand{\supp}{\operatorname{supp}}

\newcommand{\im}{\operatorname{i}}
\newcommand{\Tr}{\operatorname{Tr}}
\newcommand{\tr}{\operatorname{tr}}
\newcommand{\rk}{\operatorname{rank}}
\newcommand{\Det}{\operatorname{Det}}

\newcommand{\Herm}{\operatorname{Herm}}

\newcommand{\Rplus}{\mathbb{R}_{\geqslant 0}}

\newcommand{\set}[1]{\left\{#1\right\}}

\renewcommand{\labelenumi}{\rm{(\roman{enumi})}}
\renewcommand{\theenumi}{\rm{(\roman{enumi})}}

\begin{document}

\frontmatter

\author[C.~Cuchiero, M.~Keller-Ressel, E.~Mayerhofer and J.~Teichmann]{Christa Cuchiero, Martin Keller-Ressel, Eberhard Mayerhofer, and \\ Josef Teichmann}
\address{ETH Z\"urich, Departement Mathematik, R\"amistrasse 101, 8092 Z\"urich, Switzerland\\
TU Berlin - Fakult\"at II, Institut f\"ur Mathematik, MA-705, Strasse des 17.~Juni 136, 10623 Berlin, Germany\\
Deutsche Bundesbank, Research Center, Wilhelm-Epstein-Str. 14, 60431 Frankfurt am Main, Germany}
\email{christa.cuchiero@math.ethz.ch, mkeller@math.tu-berlin.de\\ eberhard.mayerhofer@bundesbank.de, josef.teichmann@math.ethz.ch}

\title[Affine processes on symmetric cones]{Affine processes on symmetric cones}
\alttitle{Processus affines dans des c\^ones sym\'etriques}

\begin{abstract}
We consider affine Markov processes taking values in
convex cones. In particular, we characterize all affine processes taking values in an irreducible symmetric cone in
terms of certain L\'evy-Khintchine triplets. This is the complete classification of affine processes on these conic
state spaces, thus extending the theory of Wishart processes on positive semidefinite matrices, as put forward by Bru (1991).
\end{abstract}

\subjclass{Primary: 60J25; Secondary: 15B48}
\keywords{affine processes, symmetric cones, non-central Wishart distribution, Wishart Processes}
\altkeywords{processus affines, c\^ones sym\'etriques, distribution de Wishart non centr\'ee, processus de Wishart}

\thanks{The first and the fourth author gratefully acknowledge the financial support by the ETH Foundation}

\maketitle
\tableofcontents

\mainmatter

\section{Introduction}
In recent years the study of affine Markov processes has gained
increasing interest both in the theory of stochastic processes and
their applications. We continue and generalize with the present work
our research on positive matrix-valued affine processes
(see~\citet{cfmt}). Much of this research has been
motivated by applications in mathematical finance, where affine
processes serve as realistic models for stochastic correlation of
multivariate asset models as well as for economic risk factors with
a non-trivial dependence structure. For an account of relevant
applications of matrix-valued affine processes we refer
to~\citet{cfmt}.

A natural generalization of positive semidefinite matrices are the so-called symmetric cones
(see the standard reference~\citet{Faraut1994}). This framework covers many important examples, such as the cone $\re_+^n$,
the cone of Hermitian matrices and the Lorentz cone $\Lambda_n$, defined by
$\Lambda_n:=\{ x \in \re^n \, |\, x_1^2 -\sum_{i=2}^nx_i^2 \geq 0, x_1 \geq 0\} $.
Affine diffusion processes on this kind of state spaces were first considered by~\citet{grasselli} in the context of affine term structure models.
In this article, we take up the setting of symmetric cones and provide a full characterization of affine processes thereon.

Other state spaces, such as sets whose boundary is described by a
quadratic polynomial, have been considered by~\citet{spreij1}. It
turns out that the condition of a quadratic boundary structure
implies that the state space is either parabolic (see
also~\citet[Section 12]{dfs}) or isomorph to the symmetric Lorentz
cone. Let us remark that the boundary of other symmetric cones is in
general described by polynomials of higher degree.

Only a few articles have considered affine processes
on completely general state spaces, such as~\citet{kst1}
and~\citet{cuchteich}, where regularity, that is the
time-differentiability of their Fourier-Laplace transform, and path
properties of affine processes are considered. For an analysis of
affine processes (under the regularity condition) on relatively
general state spaces we refer to the thesis of~\citet{veerman}.

In the present article the results of~\citet{cfmt} are extended or reformulated or simplified as follows:
\begin{itemize}
\item Regularity, the Feller property and necessary admissibility
conditions are proved for affine processes on proper closed convex cones.
\item Sufficient admissibility conditions, in other words the full
characterization of affine processes, are derived
for symmetric cones.
\item It is shown for the first time that there exist affine diffusion processes
on cones which are neither polyhedral nor
symmetric.
\item Non-central Wishart distributions on symmetric cones are
analyzed, described and embedded into affine processes, extending
considerably the knowledge on those distributions.
\item We manage to simplify the theory of affine processes by taking
the paradigm seriously that every argument involving Kolmogorov equations 
should be replaced by an argument involving generalized
Riccati equations.
\end{itemize}
This last point implies in particular that essential parts are now proved differently: Not only do we use the special
structure of Euclidean Jordan algebras, but we also approach the existence issue
in a new, more intrinsic and elementary way.\\
Indeed, the existence proof is based on the result that the solutions of the generalized Riccati differential equations of a pure diffusion process
with a particular drift, which we call \emph{Bru process}, can be recognized as cumulant generating functions
of the non-central Wishart distribution. In this context we also derive
the explicit form of its density function (whenever it exists) on general symmetric cones, which has not been provided so far.
Furthermore, for a particular class of affine diffusion processes (which correspond  to the class of Wishart processes in the case of positive semidefinite matrices) we also establish the precise conditions under which the Markov kernels admit a density.

The remainder of the article is organized as follows:
Section~\ref{sec:mainresult} contains important definitions and a summary of the main results, whose proofs are
postponed to the subsequent sections. In Section~\ref{sec:cone} we
focus on affine processes on general convex cones, while in
Section~\ref{sec:symcone} the corresponding results are refined in
the setting of symmetric cones. A construction of affine processes
on symmetric cones as well as a derivation of the non-central
Wishart densities is done in Section~\ref{sec:existence}. In
Section~\ref{sec:nonboundary} we finally establish precise
conditions under which affine processes remain almost surely in the
interior of a symmetric cone.
 For the reader's convenience important notions of Euclidean Jordan algebras are summarized in Appendix~\ref{appendix:EJA}.

\section{Definition and Main Results}\label{sec:mainresult}

Let $V$ be a finite-dimensional real vector space with inner product $\langle \cdot, \cdot \rangle$,
containing a closed convex cone $K$, and the closed dual cone
\[
K^{\ast} = \{u \in V\, |\, \langle x, u \rangle \geq 0 \textrm{ for all } x \in K\}.
\]
We assume $K$ to be \emph{proper}, i.e., $K \cap (-K) = \set{0}$, and \emph{generating}, i.e., $K$ contains a basis, which is
equivalent to $V= K - K$ (see, e.g., \citet[Lemma 3.2]{aliprantis_07}). These assumptions imply in particular that $K^{\ast}$ is also generating and proper (see, e.g.,~\citet[Proposition I.1.4]{Faraut1994}).
We denote the open dual cone of $K$ by
\[
\mathring{K}^{\ast} = \{u \in V \, |\, \langle x, u \rangle > 0 \textrm{ for all } x \in K\},
\]
and by
\[
\partial K^{\ast} = \{u \in V \, |\, \langle x, u \rangle=0 \textrm{ for some }x \in K\}
\]
the boundary of $K^{\ast}$. Note that $\mathring{K}^\ast$ is non-empty (see~\citet[I.1.4]{Faraut1994}).
Like any cone, $K^{\ast}$ induces a partial and strict order relation on $V$: For $u,v\in V$
we write $u \preceq v$ if and only if $v-u  \in K^{\ast}$ and $u \prec v$ if and only if $v-u  \in \mathring{K}^{\ast}$.
Finally, symmetric matrices and positive semidefinite matrices over $V$ are denoted by
$S(V)$ and $S_+(V)$, respectively, while $\mathcal{L}(V)$ corresponds to the space of linear maps on $V$.

We want to study a class of time-homogeneous Markov processes which are \emph{stochastically continuous},
take values in the cone $K$, and have the so-called \emph{affine property}. Since we shall not assume the processes to be conservative,
we adjoin to the state space $K$ a point $\Delta \notin K$, called cemetery state, and set $K_{\Delta}=K\cup\{\Delta\}$. Let now $X$ be a Markov process on $K$. We denote by
 $(p_t(x,\cdot))_{t\geq 0, x \in K}$  the transition kernels of $X$. These are extended to $K_{\Delta}$ by setting
\[
 p_t(x,\{\Delta\})=1-p_t(x,K), \quad p_t(\Delta, \{\Delta\})=1,
\]
for all $ t \in \re_+$ and $x \in K$, with the convention $f(\Delta)=0$ for any function $f$ on $K$.

\begin{definition}[Cone-valued affine process]\label{def:affineprocessK}
A time-homogeneous Markov process $X$ relative to some filtration
$(\mathcal{F}_t)$ with state space $K$ (augmented by $\Delta$) and
transition transition kernels $(p_t(x,d\xi))_{t
\geq 0,x \in K}$ is called \emph{affine} if
\begin{enumerate}
\item it is stochastically continuous, that is, $\lim_{s\to t}
p_s(x,\cdot)=p_t(x,\cdot)$ weakly on $K$ for every $t \geq 0$ and $x\in K$, and
\item\label{def:affineprocess2K} its Laplace transform has exponential-affine
dependence on the initial state. This means that there exist functions
$\phi:\re_+ \times K^{\ast} \to \re$ and $\psi:\re_+ \times K^{\ast} \to V $
such that
\begin{align}\label{eq:affineprocessK}
\int_{K}e^{-\langle u, \xi
\rangle}p_t(x,d\xi)=e^{-\phi(t,u)-\langle \psi(t,u),x\rangle},
\end{align}
for all $x\in K$ and $(t,u) \in \re_+ \times K^{\ast}$.
\end{enumerate}
\end{definition}

\begin{remark}\label{rem:eqvdef}
In the papers by~\citet{kst1} and by~\citet{cuchteich} affine
processes on a general state space $D$ are defined
by requiring the exponential-affine form of the Fourier-Laplace
transform: This means that there exist functions $\Phi:\re_+ \times
\mathcal{U} \to \cn$ and $\Psi:\re_+ \times \mathcal{U} \to V +\im
V$ such that
\begin{align}\label{eq:affineprocess}
\int_{D}e^{\langle u, \xi
\rangle}p_t(x,d\xi)=\Phi(t,u)e^{\langle \Psi(t,u),x\rangle},
\end{align}
for all $x\in D$ and $(t,u) \in \re_+ \times \mathcal{U}$.
Here $\mathcal{U}$ is defined by
\begin{align*}
\mathcal{U}=\left\{u \in V +\im V\,\big|\, e^{\langle u,x\rangle} \textrm{ is a bounded function on $D$}\right\}.
\end{align*}

Since the set $\mathcal{U}$ is given by $\mathcal{U}=-K^{\ast}+\im V$ in the case of a conic state space, the definition of an affine process can be slightly modified by requiring the affine
property only on $-K^{\ast}$. Thus, instead of the Fourier-Laplace transform,
we here only consider the Laplace transform of $X$, implying that
\[
\int_{K}e^{-\langle u, \xi\rangle}p_t(x,d\xi)=\Phi(t,-u)e^{\langle\Psi(t,-u),x\rangle}=e^{-\phi(t,u)-\langle \psi(t,u),x\rangle}
\]
is real-valued and cannot become $0$ for $u \in K^{\ast}$.

Let us also note that in the present cone setting this is equivalent to the affine property used in~\citet{cuchteich}.
Indeed, if $X$ is an affine process with state space $D=K$ in the sense of~\eqref{eq:affineprocess},
then it is clearly also an affine process
in the sense of Definition~\ref{def:affineprocessK}, since the only difference is the restriction of $\mathcal{U}$ to $-K^{\ast}$.
Note that for $u \in K^{\ast}$ we have
\[
\Phi(t,-u)=e^{-\phi(t,u)} \quad \textrm{and} \quad \Psi(t,-u)=-\psi(t,u).
\]

The other direction can be shown by following the proof of~\citet[Lemma 2.5]{kst}, which implies that property~\eqref{eq:affineprocessK}
can be extended to $\mathcal{U}$, where it takes the form~\eqref{eq:affineprocess}.
\end{remark}

\begin{remark} It is standard to realize a Markov process on
the canonical path space. By~\citet{cuchteich} (and the above
equivalence of the various definitions) $X$ has a version with c\`adl\`ag paths. Hence on
proper convex cones we can consider affine processes on the filtered
probability space $(\Omega, \mathcal{F}, (\mathcal{F}_t),
\mathbb{P}_x)$. Here $\Omega=\mathbb{D}(K_{\Delta})$ denotes the
space of c\`adl\`ag paths $\omega:\re_+\to K_{\Delta}$ with
$\omega(t)=\Delta$ for $t\geq s$, whenever $\omega(s-)=\Delta$ or
$\omega(s)=\Delta$, and $\mathbb{P}_x$ the law of $X$ given $X_0=x$.
Moreover, $\mathcal{F}, \mathcal{F}_t$ are given by
\begin{align*}
\mathcal{F}:=\bigcap_{x \in K_{\Delta}}\mathcal{F}^x, \quad \mathcal{F}_t:=\bigcap_{x \in K_{\Delta}}\mathcal{F}_t^x,
\end{align*}
where $(\mathcal{F}^{x}_t)$ is the usual augmentation of the natural filtration $\sigma(X_s,\, s\leq t)$ with respect to $\mathbb{P}_x$.
\end{remark}

The following theorem summarizes now  the main results on affine processes
in the setting of generating proper closed convex cones $K$:

\begin{theorem}\label{th: main theorem}
Let $X$ be an affine process on $K$. Then $X$ is a Feller process, the functions $\phi$ and $\psi$ given in
\eqref{eq:affineprocessK} are differentiable with respect to time
and satisfy the generalized Riccati equations for $u \in K^{\ast}$,
that is,
\begin{subequations}\label{eq:Riccati}
\begin{align}
\frac{\partial \phi(t,u)}{\partial t}&=F(\psi(t,u)), &\quad &\phi(0,u)=0,\label{eq:RiccatiF}\\
\frac{\partial \psi(t,u)}{\partial t}&=R(\psi(t,u)), &\quad &\psi(0,u)=u \in K^{\ast}, \label{eq:RiccatiR}
\end{align}
\end{subequations}
where $F(u)=\partial_t\phi(t,u)|_{t=0}$ and $R(u)=\partial_t \psi(t,u)|_{t=0}$.
Moreover, relative to any truncation function\footnote{A truncation function is continuous, bounded in norm by $1$ and equals the identity in a neighborhood of the origin.}  $\chi$, there exists a parameter set $(Q,b,B,c,\gamma,m,\mu)$
such that the functions $F$ and $R$ are of the form
\begin{subequations}\label{eq:FRLevyK}
\begin{align}
F(u)&=\langle b,u \rangle + c -\int_{K} \left(e^{-\langle u, \xi\rangle}-1\right) m(d\xi),\label{eq:FLevyK}\\
R(u)&=-\frac{1}{2}Q(u,u)+B^{\top}(u)+\gamma-\int_{K} \left(e^{-\langle u, \xi\rangle}-1+ \langle \chi(\xi),u \rangle \right) \mu(d\xi),\label{eq:RLevyK}
\end{align}
\end{subequations}
where
\begin{enumerate}
\item\label{eq:const_drift1} $b \in K$,
\item\label{eq:const_killing1} $c \in \re_+$,
\item\label{eq:const_measure1} $m$ is a
Borel measure on $K$ satisfying $m(\{0\}) = 0$ and
\[
\int_K \left(\|\xi\| \wedge 1\right)m(d\xi)< \infty,
\]
\item\label{eq:lin_diffusion1} $Q: V \times V \rightarrow V$ is a symmetric bilinear function
such that for all $v \in V$,  $Q(v,v) \in K^{\ast}$ and $\langle x, Q(u,v) \rangle=0$, whenever $\langle u, x\rangle=0$ for $u \in K^{\ast}$ and $x \in K$,
\item\label{eq:lin_killing1} $\gamma \in K^{\ast}$,
\item\label{eq:lin_measure1} $\mu$ is a $K^{\ast}$-valued $\sigma$-finite Borel measure on $K$ satisfying $\mu(\{0\}) = 0$,
$
\int_K\left(\|\xi\|^2 \wedge 1 \right)\langle x,\mu(d\xi)\rangle < \infty
$ for all $x \in K$,
and
\[
\int_K\langle \chi(\xi), u \rangle \langle x, \mu(d\xi)\rangle < \infty \textrm{ for all $u \in K^{\ast}$ and  $x \in K$ with $\langle u, x\rangle=0$},
\]
\item\label{eq:lin_drift1} $B^{\top}: V \rightarrow V$ is a linear map, satisfying
\[
\langle x, B^{\top}(u)\rangle  - \int_K \langle \chi(\xi), u \rangle \langle x, \mu(d\xi)\rangle \geq 0 \textrm{ for all $u \in K^{\ast}$ and $x \in K$ with $\langle u, x\rangle=0$}.
\]
\end{enumerate}
Conversely, let $(Q=0, b, B, c, \gamma, m, \mu)$ be a parameter set satisfying the above conditions. Then there exists a unique affine (pure jump) process
on $K$ such that~\eqref{eq:affineprocessK} holds for all $(t,u) \in \re_+\times
K^{\ast}$, where $\phi(t,u)$ and $\psi(t,u)$ are given by~\eqref{eq:RiccatiF} and~\eqref{eq:RiccatiR}.
\end{theorem}

\begin{proof}
The Feller property is proved in Proposition~\ref{th:Feller}, while the differentiability of $\phi$ and $\psi$ follows from
Proposition~\ref{th:regularityK}. The second part of the assertion is a consequence of Proposition~\ref{prop:LKform} and Proposition~\ref{prop:necessaryadm}. The existence of pure affine jump processes for a given parameter set is established in Proposition~\ref{th: existence Markov2}.
\end{proof}

\subsection{{Symmetric Cones}}
In order to make the conditions on the parameters as formulated in
Theorem \ref{th: main theorem} more explicit and to prove existence
of affine processes with a diffusion part, we now assume $K$ to be a
\emph{symmetric} cone. This setting imposes an additional algebraic structure on $V$ and implies a natural multiplication operation $\circ:
V \times V \to V$, which endows $V$ with the structure of a
so-called \emph{Euclidean Jordan Algebra}. The cone $K$ is then
exactly the cone of squares in this algebra, that is, $K = \{x \circ
x: x \in V\}$.

We start by explaining the fundamental definitions
from the standard reference on symmetric cones and Euclidean Jordan algebras,~\citet{Faraut1994}.
In order to give some intuition, we illustrate them by using the $r \times r$ real symmetric matrices $S_r$.

\begin{definition}[Symmetric cone]\label{def:symcone}
A convex cone $K$ in an Euclidean space $(V,\langle \cdot, \cdot \rangle)$ of dimension $n$ is called \emph{symmetric} if it is
\begin{enumerate}
\item \emph{homogeneous}, which means that the automorphism group
\begin{align}\label{eq:aut}
G(K)=\{g \in GL(V)\,|\,gK=K\}
\end{align}
acts transitively on $K$. That is, for all $x,y \in \mathring{K}$ there exists
an invertible linear map $g: V \to V$ that leaves $K$ invariant and maps $x$ to $y$,
\item \emph{self-dual}, that is, $K^\ast = K$.
\end{enumerate}
A symmetric cone $K$ is said to be \emph{irreducible} if there are no non-trivial subspaces $V_1$, $V_2$ and
symmetric cones $K_1 \subset V_1$, $K_2 \subset V_2$ such that $V$
is the direct sum of $V_1$ and $V_2$ and $K=K_1+K_2$.
\end{definition}

\begin{example}
For illustrative purposes, let us consider the vector space $S_r$ of symmetric $r \times r$-matrices, which is of
dimension $n=\frac{r(r+1)}{2}$. A scalar product on this space is
given by $\langle x, y \rangle =\tr(xy)$, where $\tr$ denotes the
usual matrix trace. The set of positive semidefinite
matrices is a symmetric cone in this vector space, and will be
denoted by $S_r^+$. Moreover, we write $S_r^{++}$ for the open cone
of positive definite matrices. Clearly, $S_r^+$ is self-dual with
respect to $\langle x, y \rangle =\tr(xy)$. Furthermore, its
automorphism group is given by
\[
 G(S_r^+)=\left\{G \in GL(S_r)\,|\, Gx=gxg^{\top},\, g \in GL(\re^r)\right\}.
\]
Let $\sqrt{z}$ denote the unique square root of the a positive
semidefinite matrix $z$. Then by setting
$g=\sqrt{y}\sqrt{x^{-1}}$, we have $G(x)=y$. Hence $S_r^+$ is
homogeneous.
\end{example}

As already mentioned, symmetric cones are directly related to Euclidean Jordan algebras. These, in turn, are defined as follows:

\begin{definition}[Euclidean Jordan algebra]\label{def:EJA}
A real Euclidean space $(V,\langle \cdot, \cdot \rangle)$ with a bilinear product $\circ: V \times V \to V: (x,y) \mapsto x \circ y$ and identity element $e$ is called an \emph{Euclidean Jordan algebra} if
\begin{enumerate}
 \item $V$ is a Jordan algebra with product $\circ$, that is, for all $x,y \in V$
\[
(a)\quad x\circ y = y \circ x, \qquad(b)\quad x^{2}\circ (x\circ y) = x \circ (x^{2} \circ y),
\]
\item and the Jordan product is compatible with the scalar product, in the sense that
\[
\langle x\circ y, z\rangle = \langle y, x \circ z \rangle.
\]
\end{enumerate}
An Euclidean Jordan algebra is said to be \emph{simple} if it does not contain any non-trivial ideal.
\end{definition}

\begin{remark}
Note that the Jordan product is commutative by (a), but in general not associative.
Thus (b) is a genuine axiom. We have used $x^2$ to denote the Jordan product $x \circ x$. This should not cause confusion, even when we use the same notation to denote powers of scalars and matrices. By induction it is seen that $V$ is a power associative algebra, that is, $x^m \circ x^n =x^n \circ x^m=x^{m+n}$, for all $m,\,n \geq 1$.
\end{remark}

\begin{example}
 By defining the following product on the vector space of $r \times r$ real symmetric matrices
\[
 x \circ y = \frac{1}{2}( xy+yx),
\]
it is easily verified that $S_r$ is a Jordan algebra.
Here $xy$ denotes the usual matrix multiplication. It is well known that the trace is invariant under cyclic permutations, that is,
\[
\langle x, yz\rangle=\langle y,zx\rangle=\langle z,xy\rangle.
\]
Together with  $\langle x, y \rangle = \tr(xy) = \tr(yx) = \tr(x \circ y)$ we obtain (ii).
\end{example}

For an element $x \in V$ we introduce the \emph{left-product operator}, denoted by $L$ and defined by
\begin{align}\label{eq:leftprod}
L(x) y = x \circ y.
\end{align}
Moreover, $P$ denotes the so-called \emph{quadratic representation} of $V$, given by
\begin{align}\label{Eq:quadratic_product}
P(x) = 2L(x)^2 - L(x^2).
\end{align}
For both operators we have $L=L^{\top}$ and $P=P^{\top}$. In the
case of $S_r$, the quadratic representation is given by $P(x)y=xyx$.

The one-to-one correspondence between Euclidean Jordan algebras and symmetric cones is established
in~\citet[Theorem III.3.1, III.4.4 and III.4.5]{Faraut1994} and can be rephrased as follows:

\begin{theorem}\label{th:symconeEJA}
Let $K$ be a symmetric cone in $V$.
Then there exists a Jordan product $\circ$ on $V$ such that $(V,\circ)$ is an Euclidean Jordan algebra, and
\[
K = \{x^2: x \in V\}.
\]
The symmetric cone is irreducible if and only if the associated Euclidean Jordan algebra is simple.
\end{theorem}

\begin{example}
In the case of $S_r^+$, the above theorem can easily be verified, since
\[
 S_r^+=\{x \circ x =x^2 \,|\, x \in S_r\}.
\]
Note that the Jordan product $x \circ x=x^2$ equals the matrix
product $xx=x^2$ in this case. It can be easily seen
from their eigenvalue decomposition that the set of squares of
symmetric matrices is equal to the set of positive semidefinite
matrices.
\end{example}

We use some further facts from the theory of Jordan algebras in the
sequel. For those we refer to~\citet{Faraut1994} and also to
Appendix~\ref{appendix:EJA}, where we give an overview of the most
important results and illustrate them by means of
real-valued symmetric matrices.

We are now prepared to formulate the refined conditions on the parameters given in Theorem~\ref{th: main theorem}
in the context of irreducible symmetric cones.

\begin{theorem}\label{th:neccond}
Let $X$ be an affine process on an irreducible symmetric cone $K$ with parameters
$(Q,b,B,c,\gamma,m,\mu)$ as derived in Theorem~\ref{th: main theorem}. Then
there exists  $\alpha \in K$ such that
\[
 Q(u,u)=4P(u)\alpha,
\]
where $P(u)$ denotes the quadratic representation of the Euclidean Jordan algebra
$V$, defined in~\eqref{Eq:quadratic_product}, and where $ \alpha=Q(e,e)/4$.

Moreover, if $\dim V >2$, then the linear jump measure $\mu$ additionally satisfies
\begin{align}\label{eq:finitejumpmu}
 \int (\|\xi\| \wedge 1) \langle x, \mu(d\xi)\rangle < \infty, \quad \textrm{for all } x \in K.
\end{align}

Furthermore, the following drift condition holds:
\begin{equation}\label{eq:drift}
 b\succeq d(r-1) \alpha.
\end{equation}
Here $r$ denotes the rank (see~\eqref{Eq:rank}) and $d$ the Peirce invariant (see~\eqref{eq:Peirce}) of $V$.
\end{theorem}
We point out that the new parameter conditions of
Theorem \ref{th:neccond} are independent of the choice of truncation
function $\chi$. The following corollary shows that
the drift condition \eqref{eq:drift} can be strengthened to give a
condition for boundary non-attainment.

\begin{corollary}\label{cor:noboundary}
Let $X$ be a conservative affine process. If $\dim V >2$ and
\begin{align}\label{eq:nonboudary}
 b\succeq (d(r-1)+2) \alpha,
\end{align}
then $\mathbb{P}_x[X \in \mathring{K}]=1$ for each $x \in
\mathring{K}$, that is if $X$ is started at $x \in \mathring{K}$ it
remains in $\mathring{K}$ almost surely.
\end{corollary}

\begin{proof}
The results are a consequence of Proposition~\ref{prop:alpha},
Proposition~\ref{prop:jumpfinite}, Proposition~\ref{Prop:drift} and
(for the corollary) Proposition~\ref{prop:boundary}.
\end{proof}

In the following definition we summarize the above introduced parameter restrictions for affine processes
on irreducible symmetric cones.

\begin{definition}[Admissible parameter set]
An \emph{admissible parameter set}\\ $(\alpha, b,B, c,\gamma, m, \mu)$ (associated with a truncation function $\chi$) for an affine process on an
irreducible symmetric cone $K$ consists of
\begin{itemize}
\item a linear diffusion coefficient
\begin{equation}\label{eq: alpha}
\alpha \in K,
\end{equation}
\item a constant drift term satisfying
\begin{equation}\label{eq: b}
b \succeq d(r-1)\alpha,
\end{equation}
\item a constant killing rate term
\begin{equation}\label{eq: c}
c \in \re_+,
\end{equation}
\item a linear killing rate coefficient
\begin{equation}\label{eq: gamma}
\gamma \in K,
\end{equation}
\item a constant jump term: a
Borel measure $m$ on $K$ satisfying
\begin{equation}\label{eq: m}
m(\{0\}) = 0 \textrm{ and } \int_K \left(\|\xi\| \wedge 1\right)m(d\xi)< \infty,
\end{equation}
\item a linear jump coefficient: a $K$-valued $\sigma$-finite
Borel measure $\mu$ on $K$ with $\mu(\{0\}) = 0$
such that the kernel
\begin{equation}\label{eq: mudef}
 M(x,d\xi):=\langle x, \mu(d\xi)\rangle
\end{equation}
satisfies
\[
\int_K(\|\xi\|^2 \wedge 1) M(x, d\xi) < \infty, \quad \textrm{for all } x \in K,
\]
and
\begin{equation}\label{eq: mu}
\int_{K} \langle \chi(\xi), u\rangle M(x,d\xi) < \infty\quad\textrm{for all $x,u\in K$ with $\langle x,u\rangle=0$,}
\end{equation}
\item a linear drift coefficient:  a linear map $B^{\top}: V \rightarrow V$ such that
\begin{equation}\label{eq: betaij}
\langle x,B^{\top}(u)\rangle -
\int_{K}\left\langle\chi(\xi),u\right\rangle M(x,d \xi) \geq 0\quad\textrm{for all $x,u\in K$ with $\langle x,
u\rangle=0$.}
\end{equation}
\end{itemize}
\end{definition}

\begin{remark}
From equation \eqref{eq:affineprocessK} and Theorem~\ref{th: main
theorem} we see that
\[\int_K e^{-\langle u,\xi \rangle}p_t(x,d\xi) \approx \exp \left(-F(u) - \langle R(u),x \rangle\right)\]
for small $t \ge 0$. Hence the coefficients of $F$ influence the
transition probability of $X$ in a `constant' way, while the
coefficients of $R$ enter `linearly' with respect to the starting
value $x$. This explains the terminology of `constant' and `linear'
parameters as used above.
\end{remark}

\begin{remark}\label{rem:jumps}
If $\dim V>2$, then~\eqref{eq: mu} can be replaced by~\eqref{eq:finitejumpmu}. Indeed, we can introduce a new drift $\widetilde{B}$ by setting
\[
\widetilde B(u):= B(u)-\int_{K}\left\langle\chi(\xi),u\right\rangle \mu(d\xi),
\]
which in view of ~\eqref{eq: betaij} satisfies
\[
\langle x, \widetilde B(u)\rangle\geq 0\quad\textrm{for all $x,u\in K$ with $\langle x,
u\rangle=0$.}
\]
Accordingly, the function $R$ in Theorem~\ref{th: main theorem} could be altered by omitting $\chi$ and replacing
$B(u)$ by $\widetilde B(u)$.
\end{remark}

Using the above definition we can reformulate and improve Theorem~\ref{th: main theorem} for affine processes on irreducible
symmetric cones.

\begin{theorem}\label{th:main theorem sym cone}
Let $X$ be an affine process on an irreducible symmetric cone $K$.
Then $X$ is a Feller process, the functions $\phi$
and $\psi$ defined in \eqref{eq:affineprocessK} are differentiable
with respect to time and satisfy the generalized Riccati
equations~\eqref{eq:Riccati}. Moreover, there exists an admissible
parameter set $(\alpha,b,B,c,\gamma,m,\mu)$ associated with some
truncation function $\chi$ such that the functions $F$ and $R$ are
of the form
\begin{align}
F(u)&=\langle b,u \rangle + c -\int_{K} \left(e^{-\langle\xi, u\rangle}-1\right) m(d\xi),\label{eq:FsymK}\\
R(u)&=-2 P(u)\alpha+B^{\top}(u)+\gamma-\int_{K} \left(e^{-\langle\xi, u\rangle}-1+ \langle \chi(\xi),u \rangle \right) \mu(d\xi).\label{eq:RsymK}
\end{align}
Conversely, let $(\alpha, b, B, c, \gamma, m, \mu)$ be an admissible parameter set.
Then there exists a unique affine process
on $K$ such that~\eqref{eq:affineprocessK} holds for all $(t,u) \in \re_+\times
K$, where $\phi(t,u)$ and $\psi(t,u)$ satisfy the generalized Riccati equations~\eqref{eq:RiccatiF} and~\eqref{eq:RiccatiR}.
\end{theorem}

\begin{proof}
The first assertion is a reformulation of Theorem~\ref{th: main theorem} using the results of Theorem~\ref{th:neccond}.
The second statement concerning the existence of affine processes for a given admissible parameter set is subject of
Proposition~\ref{th:existence}.
\end{proof}

\subsection{Discussion of the Admissibility Conditions}

In order to give some intuition on the above introduced admissibility conditions for affine processes on symmetric cones,
we discuss and highlight some properties of the admissible
parameter set $(\alpha, b,B, c, \gamma, m, \mu)$.
In particular, we shall compare them with the well-known admissibility conditions
for the canonical state space $\mathbb{R}^m_+\times \mathbb{R}^{n-m}$ and the cone $\mathbb{R}^m_+$ (see~\citet[Definition 2.6]{dfs}).
Note that the latter is a reducible symmetric cone, whose associated Euclidean Jordan algebra $\mathbb{R}^m$ is of rank $1$.
We also exemplify the admissibility conditions by means of the cone of $r \times r$ positive semidefinite matrices
(compare also~\citet{cfmt}).

\subsubsection{Diffusion}
The diffusive behavior of an affine process on a general proper
convex cone is fully determined by the linear diffusion coefficient. This is in constrast to the mixed state space
$\mathbb{R}^n\times \mathbb{R}^m_+$ ($n>0$), on which affine
processes with a non-zero constant diffusion part exist.

Condition~\ref{eq:lin_diffusion1} of Theorem~\ref{th: main theorem}, that is,
\[
 \langle x, Q(u,v) \rangle=0 \textrm { for all } v \in V \textrm{ and } x \in K,\, u \in K^{\ast} \textrm{ with }  \langle u, x\rangle=0,
\]
is the dual formulation of the parallel diffusion behavior along the boundary, which is necessary to guarantee that the process remains in $K$.
In the case of irreducible symmetric cones this translates to
\begin{align}\label{eq:diffAP}
\langle u, A(x) u\rangle = \langle x, Q(u,u) \rangle=4 \langle x, P(u)\alpha\rangle, \quad u \in V,
\end{align}
where $\alpha$ is the linear diffusion coefficient. This property is in line with the admissibility conditions on the reducible symmetric cone $\mathbb{R}^m_+$.
In this case the diffusion part $A(x)$ is of the form $A(x)=\sum_{i=1}^m 4\alpha_i e^i x_i$.
Here, $\alpha_i \in \re_+$ and $e^i$ denotes the $m\times m $ matrix, where $(e^i)_{kl}=\delta_{ik}\delta_{il}$.
As the quadratic representation of $\re^m$ is given by
\[
 P(x)y=(x_1^2 y_1, \ldots, x_m^2y_m)^{\top}, \quad x,y \in \re^m,
\]
relation~\eqref{eq:diffAP} thus also holds on the reducible symmetric cone $\mathbb{R}^m_+$.
In the case of positive semidefinite matrices, the above simplifies to
$\langle u, A(x) u\rangle = 4 \langle x, u\alpha u\rangle$.

We remark that~\eqref{eq:diffAP} has also been stated in~\citet{grasselli}, but it has been justified using different arguments than ours.

\subsubsection{Drift} The drift condition~\eqref{eq: b}
can be explained by the fact that the boundary of a symmetric cone is in general curved and kinked, which implies
this order relation between the diffusion coefficient $\alpha$ and the drift coefficient $b$. We derive this condition by using the positive maximum principle
for the generator $\mathcal{A}$, defined in~\eqref{eq: generator} (see Lemma~\ref{lemprop: characteristicsdet}).

Note that, for the rank $1$ Jordan algebra $\re$ (or, equivalently for the symmetric cone $\re_+$) the drift condition simply reduces to the non-negativity of $b$.

In the case of positive semidefinite $r \times r$ matrices the Peirce invariant equals $1$, whence $b \succeq (r-1)\alpha$.
The stronger condition~\eqref{eq:nonboudary} implying that the process remains in the interior of the cone for all times,
reduces to the well-known Feller condition on $\re_+$. Indeed, the process given by
\[
 dX_t=b dt+ 2\sqrt{ \alpha X_t} dW_t
\]
is always positive if $b \geq 2\alpha$.

One possible specification of the linear drift $B^{\top}$, e.g., in the case of Wishart processes taking values
in the cone of positive semidefinite matrices, is to consider linear maps $M: V \to V$ which belong
to the Lie algebra $\mathfrak{g}(K)$ of the automorphism group $G(K)$ as defined in~\eqref{eq:aut}. These linear maps are characterized by the relation
\[
2P(M(u))u=MP(u)+P(u)M
\]
(see~\citet[Proposition III.5.2]{Faraut1994}). Moreover, since $M \in  \mathfrak{g}(K)$ means
\[
 e^{Mt}(K)=K, \quad \textrm{ for all $t \in \re$, }
\]
condition~\eqref{eq: betaij} reads as
\[
 \langle x, M(u)\rangle =0 \textrm{ for all $x,u \in K$ with $\langle x,u \rangle =0$. }
\]
In the case of $S_r^+$, all elements in $\mathfrak{g}(S_r^+)$ can be represented by
\begin{align}\label{eq: wishart B}
M(u)= Hu+uH^{\top}
\end{align}
for some $r \times r$ matrix $H$.

\subsubsection{Killing Rate} \label{sec:killing}
A necessary condition for an affine process on any convex proper cone
to be conservative is $c=0$ and $\gamma=0$. In the case of symmetric cones,
it can be proved as in~\citet{mayerhofer} that $X$ is conservative if
and only if $c=0$ and $\psi(t,0)\equiv 0$ is the only $K$-valued
local solution of~\eqref{eq:RiccatiR} for $u=0$. The latter
condition clearly requires that $\gamma=0$.

A sufficient condition for $X$ to be conservative is $c=0$,
$\gamma=0$ and
\[
\int_{{K}\cap\{\|\xi\|\geq
1\}}\|\xi\|M(x,d\xi)< \infty, \quad \textrm{for all } x \in K.
\]
Indeed, it can be shown similarly as in~\citet[Section 9]{dfs} that the latter property implies Lipschitz
continuity of $R(u)$ on $K$.

\subsubsection{Jump Behavior} For general convex proper cones, Condition (iii) of Theorem~\ref{th: main theorem}
means that jumps described by $m$ should be of finite variation.
Similarly, Condition (vi) asserts finite variation for the inward
pointing directions of the linear jump part. However, due to the
geometry of irreducible symmetric cones, such a behavior is no
longer possible in dimensions $\dim V> 2$ and all jumps are in fact
of finite total variation, as asserted in~\eqref{eq:finitejumpmu}
(see also Proposition~\ref{prop:jumpfinite} and
Remark~\ref{rem:jumps}). This restriction has been described by \cite{mayerhoferpers} for positive
semidefinite matrices. However, in the case of $\mathbb{R}_+$ and
the two-dimensional Lorentz cone, the linear jump part can have
infinite total variation (for an explicit example, see
\citet{mayerhoferpers}).

Let us also remark that for $r>1$ and $d>0$, affine diffusion
processes $X$ cannot be approximated (in law) by pure jump
processes, since this would imply that $X$ is infinitely divisible. 
Indeed, in view of Proposition~\ref{prop: C CS semiflow} pure jump processes are
infinitely divisible, and it is well known that this property is
conserved under convergence in law. The marginal laws of an affine diffusion process however follow a (generalized) Wishart distribution, 
which is known to be \emph{not} infinitely divisible (see e.g.~\citet{levy}).
For a characterization of
infinite divisibility in the positive semidefinite case, see
\cite[Theorem 2.9]{cfmt}.

However, such an approximation is possible for the canonical state space, since the rank
of the Euclidean Jordan algebra $\re^n$ is $1$. This is explicitly exploited in the existence proof
for affine processes on $\re^m_+ \times \re^{n-m}$ (see~\citet[Section 7]{dfs}).

\subsection{Affine Diffusion Processes on Non-Symmetric Cones}

In this section we show that there actually exist
affine diffusion processes on non-symmetric cones.
Both of the following examples are squared Bessel-type processes; the first one is defined on a (non-homogeneous) polyhedral cone and
the second one on the \emph{dual Vinberg cone}, which is homogeneous, but non-symmetric (see~\citet{vinberg}).

\begin{example}\label{pol ex}
We define the polyhedral cone
\[
K=\left\{t_1a_1+t_2a_2+t_3a_3+t_4a_4\;:\; t_1,t_2,t_3,t_4\geq 0 \right\},
\]
where
\[
a_1:=\left(\begin{array}{lll} 0\\0\\1\end{array} \right),\quad a_2:=\left(\begin{array}{lll} 1\\0\\1\end{array} \right),\quad a_3:=\left(\begin{array}{lll} 1\\1\\1\end{array} \right),\quad a_4:=\left(\begin{array}{lll} 0\\1\\1\end{array}\right) ,
\]
which is not homogeneous by~\citet[Section 2]{Ishi_gradient}. Expressed in standard coordinates, we have
\[
K=\left\{\left(\begin{array}{lll} x_1\\x_2\\x_3\end{array}\right):\quad x_1\geq 0,\;x_2\geq 0,\;x_3\geq x_1,\; x_3\geq x_2 \right\}.
\]

Let $B=(B_1,\dots,B_4)$ be a four-dimensional standard Brownian motion. We consider the surjective quadratic map
\[
q: \mathbb R^4\rightarrow K,\quad q(y):=\sum_{i=1}^4 y_i ^2 a_i.
\]
Then the process $X:=q(y+B)$, $y\in \mathbb R^4$, gives rise to an affine process. Indeed, a straight-forward
calculation yields that $X$ is an It\^o-process satisfying a stochastic differential equation of the form
\[
 dX_t=b(X_t)dt+\sigma(X_t)dW_t, \quad X(0)=q(y)
\]
where $W$ is a $3$-dimensional Brownian motion, $a(x):=\frac{1}{2}(\sigma\sigma^\top)(x)$ and $b(x)$ are affine functions in $x$ given by
\begin{align*}
a(X)&:=\frac{1}{2}([dX_i,dX_j])_{1\leq i,j\leq 3}=2\left(\begin{array}{lll}X_1&X_1+X_2-X_3& X_1\\X_1+X_2-X_3& X_2&X_2\\X_1&X_2&X_3\end{array}\right),\\
b(X)&=\left( \begin{array}{lll}2 &2&4\end{array}\right)^\top.
\end{align*}

Note that this example is covered by the general theory of affine diffusion processes on polyhedral cones, (see~\citet{spreij1}).
\end{example}

The next example provides an affine process whose state space is the \emph{dual Vinberg cone}.

\begin{example}
Let $K^{\ast}$ be the \emph{Vinberg cone}. This is a homogeneous cone in a
$5$--dimensional Euclidean space, which can be represented as
\begin{align*}
K^{\ast}&=\left\{\left(\begin{array}{lll}
a& b_1& b_2\\
b_1 & c_1 & 0\\
b_2 & 0 & c_2
\end{array}\right)\, \Bigg|\,  a \geq 0,\, ac_1-b_1^2\geq 0,\, ac_2-b_2^2 \geq 0 \right\}.
\end{align*}
Its dual cone $K$ the so-called \emph{dual Vinberg cone} is given by
\begin{align*}
K&=\left\{x=\left(\begin{array}{lll}
a& b_1& b_2\\
b_1 & c_1 & 0\\
b_2 & 0 & c_2
\end{array}\right)\, \Bigg|\, x \textrm{ is positive semidefinite }  \right\}.
\end{align*}
According to~\citet[Exercise I.10b]{Faraut1994}, every element $x \in K$ can be written as a sum
\[
 x=x^0+x^1+x^2
\]
with $x_i \in \Lambda^i, \, i \in \{0,1,2\}$, where the sets $\Lambda^i$ are defined by
\begin{align*}
 \Lambda^0&=\left\{\left(\begin{array}{lll}
a_0& 0& 0\\
0 & 0 & 0\\
0 & 0 & 0
\end{array}\right)\, \Bigg|\, a_0\geq0  \right\},\\
\Lambda^1&=\left\{\left(\begin{array}{lll}
a_1& b_1& 0\\
b_1 & c_1 & 0\\
0 & 0 & 0
\end{array}\right)\, \Bigg|\, a_1\geq0, \, a_1c_1=b_1^2  \right\},\\
\Lambda^2&=\left\{\left(\begin{array}{lll}
a_2& 0& b_2\\
0 & 0 & 0\\
b_2 & 0 & c_2
\end{array}\right)\, \Bigg|\, a_2\geq0, \, a_2c_2=b_2^2  \right\}.
\end{align*}
Notice that the map $\Lambda^0 \times \Lambda^1_{c_1 >0} \times \Lambda^2_{c_2 >0} \to K$ is invertible as long
as $c_1$ and $c_2$ are strictly positive.
We now consider three independent affine processes $X^i,\,i\in\{0,1,2\}$, taking values in the sets
$\Lambda^i$. To this end let us denote by $\Pi^0$ the projection on the $(1,1)$ component and
by $\Pi^i, \, i \in \{1,2\}$, the projection on the matrix obtained
by deleting the $(4-i)^{\textrm{th}}$ row and column. The processes $X^i,\,i\in\{0,1,2\}$, are then specified as
\begin{align*}
 d\Pi^0(X^0_{t})&=b dt+2\sqrt{\Pi^0(X^0_{t})}dB_t, \quad \Pi^0(X^0_{t})=x^0_{11}\geq 0, \quad  b\geq 0,\\
 d\Pi^i(X^i_{t})&=\left(\begin{array}{ll}
1& 0\\
0 & 0
\end{array}\right)dt+\sqrt{\Pi^i(X^i_{t})}dW^i_t\left(\begin{array}{ll}
1& 0\\
0 & 0
\end{array}\right)+\left(\begin{array}{ll}
1& 0\\
0 & 0
\end{array}\right)(dW_t^i)^{\top}\sqrt{\Pi^i(X^i_{t})}, \\
\Pi^i(X^i_{0})&=x^i= z^i(z^{i})^{\top}, \quad z^i \in \mathbb{R}^2.
\end{align*}
Here, $B$ is a one-dimensional Brownian motion and $W^{i}, i \in \{1,2\}$,
a $2 \times 2 $ matrix of Brownian motions, all mutually independent.
The remaining entries of the processes $X^i$ are supposed to be $0$.
Note in particular that $\Pi^0(X^0)$ takes values in $\mathbb{R}_+$ and $\Pi^i(X^i),\,i\in\{1,2\}$, values in
$\partial S_2^+$. 
The latter property can for example be seen by noticing
that the law of $\Pi^i(X^i_t),\, i \in \{1,2\}$, is equal to
\[
Y^i_t:=\begin{pmatrix}\left (Z^i_t+z^i_1\right)^2 & Z^i_t z_{2}^i+ z_1^i z_2^i \\
Z^i_t z_{2}^i+ z_1^i z_2^i & (z_{2}^i)^2 \end{pmatrix}=\left(\begin{array}{c} z^i_1+Z^i_t\\ z_2^i\end{array}\right)\left(\begin{array}{cc} z^i_1+Z^i_t & z_2^i \end{array}\right),
\]
where $Z^i,\, i \in \{1,2\}$, are standard one-dimensional Brownian motions.
Note that if $z^i=(z_1^i, 0)^{\top}$, then $\Pi^0(X^i),\, i \in \{1,2\}$, are one-dimensional squared Bessel processes and all other entries of $X^i,\, i \in \{1,2\}$, are $0$.
Let us now define a process $X$ by
\[
 X=X^0+X^1+X^2.
\]
Then $X$ takes values in $K$ and is an affine process.
Indeed, the functions $\phi$ and $\psi$ can be explicitly calculated and are given by
\begin{align*}
\phi(t,u)&=\left(\frac{b}{2}+1\right)\ln(1+2t \Pi^0(u)),\\
\psi(t,u)&=\begin{pmatrix}
\psi^0(t, \Pi^0(u)) & \psi^1_{12}(t,\Pi^1(u)) & \psi^1_{12}(t,\Pi^2(u))\\
\psi^1_{12}(t,\Pi^1(u)) & \psi^1_{22}(t,\Pi^1(u))&0\\
\psi^1_{12}(t,\Pi^2(u)) & 0&\psi^1_{22}(t,\Pi^2(u))
\end{pmatrix}, \quad u \in K^{\ast},
\end{align*}
where
\begin{align*}
\psi^0(t,a)&=\frac{a}{1+2ta}, \quad a \geq 0\\
\psi^1(t,v)&=\left(v^{-1}+2t\begin{pmatrix} 1&0\\
                           0&0
                          \end{pmatrix}\right)^{-1}, \quad v \in S_2^{++}.
\end{align*}
In particular, $\psi^1_{11}(t,v)=\psi^0(t,v_{11})$. The Markov property can be deduced from the semi-flow property of the
functions $\phi$ and $\psi$ (see also~\eqref{eq:flowprop} below). Notice also that for a non-degenerate starting value, the process $X$ stays in a 3-dimensional submanifold of $K$,
since, for $c_1^{\ast}, c_2^{\ast} >0$, $\Lambda^0 \times \Lambda^1_{c_1=c_1^{\ast}} \times \Lambda^2_{c_2=c_2^{\ast}} \to K$ defines a foliation by
a 3-dimensional submanifold of $K$.
\end{example}

\section{Affine Processes on General Cones}\label{sec:cone}
As above we assume that $K$ is a proper closed
convex cone, which is generating.

\subsection{Feller Property and Regularity}

In this section we shall prove that the semigroup $(P_t)_{t \geq 0}$ induced by an affine process with state space $K$, i.e.,
\[
 P_tf(x):=\int_K f(\xi) p_t(x,d \xi)
\]
is a Feller semigroup acting on the Banach space of
$C_0(K)$-functions. In order to show this property, we shall mainly
rely on Lemma~\ref{lem:psiint} below. In addition, this result also
enables us to give an alternative proof of the differentiability of
the functions $\phi$ and $\psi$ with respect to time, a property called \emph{regularity}, without
referring to the theorems obtained in~\citet{kst1}
and~\citet{cuchteich} for general state spaces. Indeed, regularity
for affine processes on cone state spaces can be obtained by arguing
as in~\citet{kst}, who obtained the corresponding statements on the
canonical state space $\re^m_+\times\re^{n-m}$ (see~\citet[Theorem
4.3]{kst}).

Let us start with the following properties of $\phi$ and $\psi$, which are immediate consequences
of Definition~\ref{def:affineprocessK}.

\begin{proposition}\label{prop:PhipsipropertiesK}
Let $X$ be an affine process on $K$. Then the functions $\phi$ and $\psi$ satisfy the following properties:
\begin{enumerate}
\item \label{item:phipsirange}$\phi$ maps $\re_+ \times K^\ast$ into $\re_+$ and $\psi$ maps $\re_+\times K^\ast$ into $K^\ast$.
\item \label{item:semiflow} $\phi$ and $\psi$ satisfy the \emph{semi-flow} property, that is,
for any $s,t \geq 0$ and $u \in K^\ast$ we have
\begin{subequations}\label{eq:flowprop}
\begin{align}
\phi(t+s,u) &= \phi(t,u) + \phi(s,\psi(t,u)), &\quad &\phi(0,u) = 0,\label{eq:flowpropphi}\\
\psi(t+s,u) &= \psi(s,\psi(t,u)), &\quad &\psi(0,u) = u \label{eq:flowproppsi}.
\end{align}
\end{subequations}
\item \label{item:continuity} $\phi$ and $\psi$ are jointly continuous on $\re_+ \times K^\ast$. Furthermore, $u \mapsto \phi(t,u)$ and $u \mapsto \psi(t,u)$ are real-analytic on $\mathring{K}^\ast$.
\item \label{item:orderpreserving} For any $t \geq 0$ and $u,v \in K^\ast$ with $u \preceq v$ the order relations
\[
\phi(t,u) \leq \phi(t,v) \quad \textrm{and} \quad \psi(t,u) \preceq \psi(t,v)
\]
hold true.
\end{enumerate}
\end{proposition}

\begin{proof}The left hand side of~\eqref{eq:affineprocessK} is clearly bounded by $1$ for all $x \in K$.
Inserting first $x = 0$ shows that $\phi(t,u)$ can only take values in $\re_+$. For arbitrary $x \in K$ the right hand side remains bounded only if $\psi(t,u) \in K^\ast$, which shows~\ref{item:phipsirange}.

Assertion~\ref{item:semiflow} follows from the Chapman-Kolmogorov equation, that is,
\begin{align*}
e^{-\phi(t+s,u)-\langle\psi(t+s,u),x\rangle}&= \int_K e^{-\langle u, \xi\rangle}p_{t+s}(x,d\xi) \\
&=\int_{K} p_s(x,d\xi)\int_{K}e^{-\langle u,\widetilde{\xi}\rangle}p_{t}(\xi,d\widetilde{\xi})\\
&=e^{-\phi(t,u)}\int_{K}e^{-\langle\psi(t,u),\xi\rangle}p_s(x,d\xi)\\
&=e^{-\phi(t,u)-\phi(s,\psi(t,u))-\langle\psi(s,\psi(t,u)),x\rangle}.
\end{align*}
Taking logarithms and using the fact that $K$ is generating, yields~\ref{item:semiflow}.

For~\ref{item:continuity}, note that stochastic continuity of $p_t(x,d\xi)$ implies
joint continuity of $\int_K e^{-\langle u, \xi \rangle}p_t(x,d\xi)$ for all $x \in K$ and
hence also of $\phi(t,u)$ and $\psi(t,u)$.
The real-analyticity in of $\phi$ and $\psi$ in $\mathring{K}^\ast$ follows from well-known properties of the Laplace transform.

Concerning~\ref{item:orderpreserving}, let $u \preceq v$,
which is equivalent to $\langle u, x \rangle \leq \langle v, x \rangle $ for all $x \in K$.
Hence, for all $t \geq 0$ and $x \in K$, we have
\[
e^{-\phi(t,u)-\langle \psi(t,u),x\rangle}=\int_{K}e^{-\langle u,\xi\rangle}p_t(x,{d}\xi) \geq \int_{K}e^{-\langle v,\xi\rangle}p_t(x,{d}\xi)=e^{-\phi(t,v)-\langle
\psi(t,v),x\rangle},
\]
which yields~\ref{item:orderpreserving}.
\end{proof}

The following lemma states that the function $\psi$ associated to an
affine process remains in the interior of the dual cone if it starts in the interior, i.e., if $\psi(0,u) \in
\mathring{K}^{\ast}$.

\begin{lemma}\label{lem:psiint}
Let $\psi:\mathbb R_+\times K^{\ast}\rightarrow V$ be any map
satisfying $\psi(0,u)=u$ and the properties~\ref{item:phipsirange}--\ref{item:orderpreserving} of
Proposition~\ref{prop:PhipsipropertiesK} (regarding the function $\psi$). Then $\psi(t,u) \in
\mathring{K}^{\ast}$ for all $(t,u) \in \re_+ \times \mathring{K}^{\ast}$.
\end{lemma}

\begin{proof}
We adapt the proofs of~\citet[Proposition 1.10]{keller} and~\citet[Lemma 3.3]{cfmt} to our
setting. Assume by contradiction that there exists some $(t,u) \in \re_+ \times \mathring{K}^\ast$
such that $\psi(t,u) \in \partial K^\ast$.  We show that in this case also $\psi(\frac{t}{2},u) \in \partial K^\ast$. First note that
\begin{align}\label{eq:recursion}
\psi\left(\frac{t}{2},v\right) \preceq \psi\left(\frac{t}{2},\psi\left(\frac{t}{2},u\right)\right) = \psi(t,u)
\end{align}
for all $v \in \Theta := \{v \in K^\ast: v \preceq \psi(\frac{t}{2},u)\}$ by
Proposition~\ref{prop:PhipsipropertiesK}~\ref{item:semiflow} and~\ref{item:orderpreserving}.
Take now some $0 \neq x \in K$ such that $\langle x,\psi(t,u)\rangle = 0$. By \eqref{eq:recursion}
also $\langle x,\psi(\frac{t}{2},v)\rangle = 0$ for all $v \in \Theta$.
If $\psi(\frac{t}{2},u) \in \mathring{K}^\ast$, then $\Theta$ is a set with non-empty interior.
By real-analyticity of $\psi$, it then follows that $\langle x,\psi(\frac{t}{2},w)\rangle = 0$
and hence $\psi(\frac{t}{2},w) \in \partial K^\ast$ for all $w \in \mathring{K}^\ast$, which is a contradiction.
We conclude that $\psi(\frac{t}{2},u) \in \partial K^\ast$.

Repeating these arguments yields, for each $n \in \mathbb{N}$, the existence of an element
$x_n \neq 0 \in K $, for which
\[
\left \langle x_n,\psi\left(\frac{t}{2^n},u\right)\right\rangle = 0.
\]
Without loss of generality we may assume that $\|x_n\| = 1$ for each $n$. Since the unit sphere is compact in finite dimensions,
there exists a subsequence $n_k$ such that $x_{n_k} \to x^\ast \neq 0$, as $k \to \infty$.
From the continuity of the function $t \mapsto \psi(t,u)$ and the scalar product we deduce that
\[
0 = \lim_{k \to \infty} \left\langle x_{n_k}, \psi\left(\frac{t}{2^{n_k}},u\right)\right\rangle = \langle x^\ast, \psi(0,u)\rangle = \langle x^\ast ,u\rangle > 0,
\]
which is the desired contradiction.
\end{proof}

It is now a direct consequence of this lemma that any affine process $X$ on $K$ is a Feller process.

\begin{proposition}\label{th:Feller}
Let $X$ be an affine process on $K$. Then $X$ is a Feller process.
\end{proposition}

\begin{proof}
The assertion can be proved by applying the same arguments as in~\citet[Proposition 3.4]{cfmt}.
\end{proof}

Let us now recall the concept of regularity.

\begin{definition}[Regularity]\label{def:regularityK}
An affine process $X$ on $K$ is called \emph{regular} if for all $u \in K^{\ast}$ the derivatives
\begin{align}\label{eq:FRdefK}
F(u) = \frac{\partial \phi(t,u)}{\partial t}\Bigg |_{t=0}, \qquad
R(u) = \frac{\partial \psi(t,u)}{\partial t}\Bigg |_{t=0}
\end{align}
exist and are continuous in $u$.
\end{definition}

\begin{proposition}\label{th:regularityK}
Let $X$ be an affine process on $K$. Then $X$ is regular and the functions $\phi$ and $\psi$ satisfy the ordinary differential equations~\eqref{eq:Riccati}.
\end{proposition}

\begin{proof}
A proof of the above theorem can be obtained by following the lines
of~\citet[Proof of Theorem 4.3]{kst}. The equations~\eqref{eq:Riccati} follow
immediately by differentiating the semi-flow equations~\eqref{eq:flowprop}.
\end{proof}

\begin{remark}
The differential equations~\eqref{eq:Riccati} are called \emph{generalized Riccati equations}.
This terminology should become clear after Proposition~\ref{prop:LKform} below.
\end{remark}

\subsection{Necessary Parameter Conditions and Quasi-monotonicity}

In this section, we focus on the specific form of the functions $F$ and
$R$, defined in~\eqref{eq:FRdefK}. As already proved in~\citet{kst1} and~\citet{cuchteich} for the case of general state spaces,
$F$ and $R$ have parameterizations of L\'evy-Khintchine type. We here show this result
in the particular case of cone-valued affine processes and relate the form of $F$ and $R$ to the notion of quasi-monotonicity.

For the proof of the main results of this section we first state a convergence result for Fourier-Laplace transforms
which can be proved exactly as in~\citet[Lemma 4.5]{cfmt}

\begin{lemma}\label{lem: FourierLaplace}
Let $(\nu_n)_{n\in \mathbb{N}}$ be a sequence of measures on $V$ with
\[
L_n(u)=\int_{V}e^{-\langle u,\xi\rangle}\nu_n(d\xi)<\infty \quad \textrm{and} \quad \lim_{n \to \infty} L_n(u) =
L(u), \textrm{ for all } u\in \mathring{K}^{\ast}\cup \{0\},
\]
pointwise, for some finite function $L$ on $\mathring{K}^{\ast}\cup \{0\}$, continuous at $u = 0$. Then $\nu_n$ converges weakly to
some finite measure $\nu$ on $V$ and the Fourier-Laplace
transform converges for $ u \in \mathring{K}^{\ast} \cup \{0\} $ and $ v \in
V$ to the Fourier-Laplace transform of $ \nu $, that is,
\[
\lim_{n \to \infty} \int_{V} e^{-\langle u + \im v,\xi\rangle}\nu_n(d\xi) =
\int_{V} e^{-\langle u + \im v,\xi\rangle}\nu(d\xi).
\]
In particular, $\nu(V)=\lim_{n \to \infty}\nu_n(V)$ and
\[
L(u) = \int_{V} e^{-\langle u ,\xi\rangle}\nu(d\xi),
\]
for all $ u \in \mathring{K}^{\ast} \cup \{ 0 \} $.
\end{lemma}

\subsubsection{L\'evy-Khintchine form of $F$ and $R$}\label{sec:LKform}

In the following, $\chi: V\rightarrow V$ denotes some bounded continuous
truncation function with $\chi(\xi)=\xi$ in a neighborhood of $0$.

\begin{proposition}\label{prop:LKform}
Let $X$ be an affine process on $K$. Then the functions $F$ and $R$ as defined in~\eqref{eq:FRdefK} are of form~\eqref{eq:FLevyK} and~\eqref{eq:RLevyK}, that is,
\begin{align*}
F(u)&=\langle b,u \rangle + c -\int_{K} \left(e^{-\langle u, \xi\rangle}-1\right) m(d\xi),\\
R(u)&=-\frac{1}{2}Q(u,u)+B^{\top}(u)+\gamma-\int_{K} \left(e^{-\langle u, \xi\rangle}-1+ \langle \chi(\xi),u \rangle \right) \mu(d\xi),
\end{align*}
where
\begin{enumerate}
\item\label{eq:const_drift} $b \in K$,
\item\label{eq:const_killing} $c \in \re_+$,
\item\label{eq:const_measure} $m$ is a
Borel measure on $K$ satisfying $m(\{0\}) = 0$ and
\[
\int_K \left(\|\xi\| \wedge 1\right)m(d\xi)< \infty.
\]
\item\label{eq:lin_diffusion} $Q: V \times V \rightarrow V$ is a symmetric bilinear function with $Q(v,v) \in K^{\ast}$ for all $v \in V$,
\item\label{eq:lin_drift} $B^{\top}: V \rightarrow V$ is a linear map,
\item\label{eq:lin_killing} $\gamma \in K^{\ast}$,
\item\label{eq:lin_measure} $\mu$ is a $K^{\ast}$-valued $\sigma$-finite Borel measure on $K$ satisfying $\mu(\{0\}) = 0$ and
\[
\int_K\left(\|\xi\|^2 \wedge 1 \right)\langle x,\mu(d\xi)\rangle < \infty \quad \textrm{for all } x \in K.
\]
\end{enumerate}
\end{proposition}

\begin{proof}
In order to derive the particular form of $F$ and $R$
with the above parameter restrictions, we follow the approach of~\citet[Theorem 2.6]{keller} (compare also~\citet[Proposition 4.9]{cfmt}).
Note that the $t$-derivative of $P_te^{-\langle u,x\rangle}$ at $t=0$ exists for
all $x  \in K$ and $u \in \mathring{K}^{\ast}$, since
\begin{equation}\label{Asharpeux}
\begin{aligned}
\lim_{t \downarrow 0}\frac{P_te^{-\langle u,x\rangle}-e^{-\langle u,x\rangle}}{t}&=\lim_{t \downarrow 0}\frac{e^{-\phi(t,u)-\langle\psi(t,u),x\rangle}-e^{-\langle u,x\rangle}}{t}\\
&=(-F(u)-\langle R(u),x\rangle)e^{-\langle u, x\rangle}
\end{aligned}
\end{equation}
is well-defined by Proposition~\ref{th:regularityK}. Moreover, we
can also write
\begin{align*}
-F(u)-\langle R(u),x\rangle&=\lim_{t \downarrow 0}\frac{P_te^{-\langle u,x\rangle}-e^{-\langle u,x\rangle}}{t
e^{-\langle u,x\rangle}} \\
&=\lim_{t\downarrow
0}\frac{1}{t}\left(\int_{K}e^{-\langle u,
\xi-x\rangle}p_t(x,d\xi)-1\right)\\
&=\lim_{t\downarrow 0}\left(\frac{1}{t}\int_{K-x}\left(e^{-\langle
u,\xi\rangle}-1\right)p_t(x,d\xi+x) +\frac{p_t(x,K)-1}{t}\right).
\end{align*}
By the above equalities and the fact that $p_t(x,K)\leq 1$, we
then obtain for $u=0$
\begin{align*}
0 \geq \lim_{t\downarrow 0}\frac{p_t(x,K)-1}{t}=-F(0)-\langle R(0),x\rangle.
\end{align*}
Setting $F(0)=c$ and $R(0)=\gamma$ yields $c \in \re^+$ and $\gamma \in K^{\ast}$, hence~\ref{eq:const_killing} and~\ref{eq:lin_killing}. We
thus obtain
\begin{align}\label{eq: FRinfdiv}
-(F(u)-c) - \langle R(u)-\gamma,x\rangle=\lim_{t\downarrow 0}
\frac{1}{t}\int_{K-x}\left(e^{-\langle u,\xi\rangle}-1\right)p_t(x,d\xi+x).
\end{align}

For every fixed $t>0$, the right hand side of~\eqref{eq: FRinfdiv}
is the logarithm of the Laplace transform of a compound Poisson
distribution supported on $K-\re_+x$ with intensity
$p_t(x,K)/t$ and compounding distribution
$p_t(x,d\xi+x)/p_t(x,K)$. Concerning the support, note that the
compounding distribution is concentrated on $K-x$, which implies
that the compound Poisson distribution has support on the convex
cone $K-\re_+ x$. By Lemma~\ref{lem: FourierLaplace}, the
pointwise convergence of~\eqref{eq: FRinfdiv} for $t\rightarrow 0$
to some function being continuous at $0$ implies weak convergence
of the compound Poisson distributions to some infinitely divisible
probability distribution $\nu(x,dy)$ supported on $K-\re_+x$.
Indeed, this follows from the fact that any compound Poisson
distribution is infinitely divisible and the class of infinitely
divisible distributions is closed under weak convergence
(see~\citet[Lemma 7.8]{sato}). Again, by Lemma~\ref{lem: FourierLaplace}, the Laplace transform of $\nu(x,dy)$ is then given as exponential of
the left hand side of~\eqref{eq: FRinfdiv}.

In particular, for $x=0$, $\nu(0,dy)$ is an infinitely divisible
distribution with support on the cone $K$. By the
L\'evy--Khintchine formula on proper cones (see, e.g.,~\citet[Theorem
3.21]{skorohod}), its Laplace transform is therefore of the form
\begin{align*}
\exp\left(-\langle b,u\rangle+\int_{K}(e^{-\langle
u, \xi\rangle}-1)m(d\xi)\right),
\end{align*}
where $b \in K$ and $m$ is a
Borel measure supported on $K$ with $m(\{0\})=0$
such that
\[
\int_{K}\left(\|\xi\|\wedge 1\right)m(d\xi) <
\infty,
\]
yielding~\ref{eq:const_measure}. Therefore,
\begin{align*}
F(u)=\langle b, u\rangle + c -\int_{K}(e^{-\langle u, \xi\rangle}-1)m(d\xi).
\end{align*}

We next obtain the particular form of $R$. Observe that for
each $x \in K$ and $k\in\mathbb{N}$,
\[
\exp\left( -(F(u)-c)/k - \langle R(u)-\gamma,x\rangle \right)
\]
is the Laplace transform of the infinitely divisible distribution
$\nu(kx,dy)^{\ast\frac{1}{k}}$, where $\ast\frac{1}{k}$ denotes the
$\frac{1}{k}$ convolution power. For $k\to\infty$, these Laplace
transforms obviously converge to $\exp(-\langle
R(u)-\gamma,x\rangle)$ pointwise in $u$. Using again the same
arguments as before (an application of Lemma~\ref{lem:
FourierLaplace} to equation~\eqref{eq: FRinfdiv}),
we can deduce that $\nu(kx,dy)^{\ast\frac{1}{k}}$ converges weakly
to some infinitely divisible distribution $L(x,dy)$ on $K-\re_+x$
with Laplace transform $\exp(-\langle R(u)-\gamma,x\rangle)$ for
$u\in K^{\ast}$.

By the L\'evy-Khintchine formula on $V$ (see~\citet[Theorem
8.1]{sato}), the characteristic function of $L(x,dy)$ has the form
\begin{align}\label{eq:characteristicfkt}
\widehat{L}(x,u)&=\exp\Bigg(\frac{1}{2}\langle u, A(x) u \rangle+\langle
B(x),u\rangle \notag\\
&\quad+\int_{V} \left(e^{\left\langle u,
\xi\right\rangle}-1-\left\langle \chi\left(\xi\right), u
\right\rangle\right)M\left(x, d \xi\right)\Bigg),
\end{align}
for $u\in\im V$, where, for every $x \in D$, $A(x) \in S_+(V)$ is a symmetric positive semidefinite linear operator on
$V$, $B(x) \in V$, $M(x,\cdot)$ a Borel measure on $V$
satisfying $M(x,\{0\})$
\[
\int_{V}(\|\xi\|^2 \wedge 1) M(x,d\xi) < \infty,
\]
and $\chi$ some appropriate truncation function. Furthermore,
by~\citet[Theorem 8.7]{sato},
\begin{align}\label{eq: Levymeasure limit}
\int_{V}f(\xi)\frac{1}{t}p_t(x,d\xi+x)\stackrel{t
\to 0}{\longrightarrow}
\int_{V}f(\xi)m(d\xi)+\int_{V}f(\xi)M(x,d\xi)
\end{align}
holds true for all $f: V \to \re$ which are bounded,
continuous and vanishing on a neighborhood of $0$. We thus conclude that
$M(x,d\xi)$ has support in $K-x$. Therefore, the characteristic
function $\widehat{L}(x,u)$ admits an analytic extension to
$K^{\ast}+\im V$, which then has to coincide with the Laplace
transform for $u\in K^{\ast}$. Hence, for all $x\in K$,
\begin{multline}\label{eq: levy khintchine R}
-\left\langle R(u)-\gamma,x\right\rangle =\frac{1}{2}\langle u, A(x)
u \rangle-\langle B(x),u\rangle\\+\int_{V}
\left(e^{-\left\langle u, \xi\right\rangle}-1+\left\langle
\chi\left(\xi\right), u \right\rangle\right)M\left(x, d
\xi\right),\quad u\in K^{\ast}.
\end{multline}
As the left side of~\eqref{eq: levy khintchine R} is linear in the
components of $x$ and as $K$ is generating, it follows that $x \mapsto A(x)$, $x\mapsto B(x)$
as well as $x \mapsto \int_E(\|\xi\|^2 \wedge 1) M(x,d\xi)$ for
every $E \in \mathcal{B}(V)$ are restrictions of linear maps
on $V$. In particular, Condition~\ref{eq:lin_drift} follows immediately.
Moreover, $\langle u, A(x) v\rangle$ can be written as
\begin{align}\label{eq:AQ}
\langle u, A(x) v\rangle=\langle x, Q(u,v) \rangle,
\end{align}
where $Q: V \times V \to V$
is a symmetric bilinear function satisfying $Q(v,v) \in K^{\ast}$ for all $v \in V$, since $A(x)$
is a positive semidefinite operator. This therefore yields~\ref{eq:lin_diffusion}.
Similarly, we have for all $E \in \mathcal{B}(V)$
\[
\int_E(\|\xi\|^2 \wedge 1) M(x,d\xi)=\int_E(\|\xi\|^2 \wedge 1)\langle x ,\mu(d\xi)\rangle,
\]
where $\mu$ is a $K^{\ast}$-valued $\sigma$-finite
Borel measure on $V$, satisfying $\mu(\{0\}) = 0$ and
\[
\int_V\left(\|\xi\|^2 \wedge 1\right)\langle x, \mu(d\xi)\rangle < \infty, \quad \textrm{for all } x \in K.
\]
Hence it only remains to prove that $\supp(\mu)\subseteq K$.
In \eqref{eq: Levymeasure limit} take now $x=\frac{1}{n}y$ for some $y \in K$ with $\|y\|=1$ and
nonnegative functions $f=f_n \in C_b(V)$ with
$f_n=0$ on $K-\frac{1}{n}y$.
Then, for each $n$, the left side of~\eqref{eq:
Levymeasure limit} is zero, since
$p_t(\frac{1}{n}y,\cdot)$ is concentrated on
$K-\frac{1}{n}y$. As $\supp(m)\subseteq K$, the first
integral on the right vanishes as well. Hence
\begin{align*}
0&=\int_{V}f_n(\xi)M\left(\frac{1}{n}y,d\xi\right)=\int_{V}f_n(\xi)\left\langle \frac{1}{n}y, \mu(d\xi)\right\rangle\\
\end{align*}
for any nonnegative function $f_n \in C_b(V)$ with $f_n=0$ on $K-\frac{1}{n}y$ implies that $\supp (\mu)\subseteq K-\frac{1}{n}y$ for
each $n$. Thus we can conclude that $\supp (\mu)\subseteq K$, which proves~\ref{eq:lin_measure}. Due to the definition of $Q$ and $\mu$
together with~\eqref{eq: levy khintchine R}, $R(u)$ is clearly of form~\eqref{eq:RLevyK}.
\end{proof}

\subsubsection{Parameter Restrictions}\label{sec:parameterrestriction}

In the following we continue the analysis of the function $R$ and derive further restrictions
on the involved parameters $Q$, $B^{\top}$ and $\mu$.

\begin{proposition}\label{prop:necessaryadm}
Let $X$ be an affine process on $K$ with $R$ of form~\eqref{eq:RLevyK}
for some $Q$, $B^{\top}$, $\gamma$ and $\mu$ satisfying the conditions of Proposition~\ref{prop:LKform} (iv)-(vii).
Then, for any $u \in K^{\ast}$ and $x \in K$ with $\langle u, x\rangle=0$, we have
\begin{enumerate}
\item\label{eq:parallel_diffusion}$\langle x, Q(u,v) \rangle=0$ for all $v \in V$,
\item\label{eq:jump_int} $\int_K\langle \chi(\xi), u \rangle \langle x, \mu(d\xi)\rangle < \infty$,
\item\label{eq:inwardpoint_drift} $\langle x, B^{\top}(u)\rangle  - \int_K \langle \chi(\xi), u \rangle \langle x, \mu(d\xi)\rangle \geq 0$.
\end{enumerate}
\end{proposition}

\begin{proof}
Let $u \in K^{\ast}$ and $x \in K$ with $\langle u, x\rangle=0$ be fixed. Define the linear map
$U: V \to \re,\, v \mapsto \langle u, v \rangle$.
As established in the proof of Proposition~\ref{prop:LKform}, $-\langle R(u)-\gamma, x \rangle$ is the Laplace transform of an infinitely divisible distribution
$L(x,dy)$ supported on $K-\re_+x$.  Similar to~\eqref{eq: levy khintchine R}, we denote the L\'evy triplet of $L(x,dy)$ by $(A(x), B(x), M(x,d\xi))$. Let now $Y_x$ be a random variable with distribution $L(x,dy)$. Then the distribution of $U(Y_x)=\langle u, Y_x\rangle$, which we denote by $L_u(x,dy)$, is again infinitely divisible and supported on $\re_+$. From~\citet[Proposition 11.10]{sato} we then infer that the L\'evy triplet $(a_u(x), b_u(x), \nu_u(x, d\xi))$ of  $L_u(x,dy)$  with respect to some truncation function $\widetilde{\chi}$ on $\re$ is given by
\begin{align*}
a_u(x)&=\langle u, A(x)u\rangle, \\
b_u(x)&=\langle B(x),u \rangle +\int_K \left(\widetilde{\chi}(\langle u, \xi\rangle)-\langle \chi(\xi), u\rangle\right)U_{\ast}M(x, d\xi),\\
\nu_u(x,d\xi)&=U_{\ast}M(x, d\xi).
\end{align*}
By the L\'evy Khintchine formula on $\mathbb{R}_+$, we conclude that
$a_u(x)=0$, $b_u(x)\geq 0$ and $\int_K (\|\xi\| \wedge 1)U_{\ast}M(x,d\xi)< \infty$. The last condition already implies~\ref{eq:jump_int} and allows to choose $\widetilde{\chi}=0$.
Moreover, $b_u(x) \geq 0$ yields~\ref{eq:inwardpoint_drift}. From
\[
0=a_u(x)=\langle u, A(x)u\rangle=\left\langle \sqrt{A(x)}u, \sqrt{A(x)}u\right\rangle
\]
it follows that $\langle v,A(x)u \rangle =0$ for all $v \in V$. Hence relation~\eqref{eq:AQ} implies~\ref{eq:parallel_diffusion}.
\end{proof}

\subsubsection{Quasi-monotonicity}

Quasi-monotonicity plays a crucial role in comparison theorems for ordinary differential equations and thus appears naturally in the setting of affine processes.
As we shall see in Section~\ref{sec:Riccati},
it is needed to establish global existence and uniqueness for
the ordinary differential equations defined in~\eqref{eq:RiccatiF}
and~\eqref{eq:RiccatiR}.
In the following we prove that the function $R$, as given in~\eqref{eq:RLevyK}, is quasi-monotone increasing if
the conditions of Proposition~\ref{prop:necessaryadm}~\ref{eq:parallel_diffusion}-\ref{eq:inwardpoint_drift} are satisfied.

\begin{definition}[Quasi-monotonicity]\label{def:quasimono}
Let $U$ be a subset of $V$. A function $f: U \to V$ is called \emph{quasi-monotone increasing} (with respect to $K^\ast$ and the induced order $\preceq$) if, for all $u,v \in U$ and $x \in K$ satisfying $u \preceq v$ and $\langle u, x\rangle =\langle v,x \rangle$,
\[
\langle f(u),x \rangle \leq \langle f(v), x \rangle.
\]
Accordingly, we call $f$ quasi-constant if both $f$ and $-f$ are quasi-monotone increasing.
\end{definition}

\begin{remark}
Note that in a one-dimensional vector space \emph{any} function is quasi-monotone. It is only in dimension greater than one that the notion of quasi-monotonicity becomes meaningful.
\end{remark}

\begin{proposition}\label{prop:quasimono}
Let $R$ be of form~\eqref{eq:RLevyK} for some $Q$, $B^{\top}$, $\gamma$ and $\mu$
satisfying the conditions of Proposition~\ref{prop:LKform} (iv)-(vii)
and Proposition~\ref{prop:necessaryadm} (i)-(iii).
Then $R$ is quasi-monotone increasing on $K^{\ast}$.
\end{proposition}

\begin{proof}
Let $\delta>0$, and define
\begin{equation}\label{eq:quasimono}
\begin{split}
R^{\delta}(u)&=-\frac{1}{2}Q(u,u)+B^{\top}(u)+\gamma-\int_{\{\|\xi\| \geq \delta\} \cap K} \left(e^{-\langle u, \xi\rangle}-1+ \langle \chi(\xi),u \rangle \right) \mu(d\xi)\\
&= -\frac{1}{2}Q(u,u)+ \gamma + B^{\top}(u)- \int_{\{\|\xi\| \geq \delta\} \cap K}\langle \chi(\xi),u \rangle \mu(d\xi)\\
&\quad -\int_{\{\|\xi\| \geq \delta\} \cap K} \left(e^{-\langle u, \xi\rangle}-1 \right) \mu(d\xi).
\end{split}
\end{equation}
Take now some $u,v \in K^{\ast}$ and $x \in K$  such that $u \preceq v$ and $\langle u, x\rangle =\langle v,x \rangle$.
Due to Condition~\ref{eq:parallel_diffusion} of Proposition~\ref{prop:necessaryadm}, we then have $\langle Q(v-u,w),x \rangle=0$ for all $w \in V$. As $Q$ is a bilinear function,
this is equivalent to $\langle Q(v,w),x \rangle=\langle Q(u,w),x \rangle$ for all $w \in V$. Inserting $w=u$ and $w=v$, we obtain
by the symmetry of $Q$
\[
\langle Q(v,v),x \rangle=\langle Q(u,u),x \rangle,
\]
whence the map $u\mapsto -\frac{1}{2}Q(u,u)+\gamma$ is quasi-constant.
Condition~\ref{eq:inwardpoint_drift} of Proposition~\ref{prop:necessaryadm} directly yields that
\[
u\mapsto  B^{\top}(u)- \int_{\{\|\xi\| \geq \delta\} \cap K}\langle \chi(\xi),u \rangle  \mu(d\xi)
\]
is a quasi-monotone increasing linear map on $K^{\ast}$. Finally, the
quasi-monotonicity of
\[
u\mapsto \int_{\{\|\xi\| \geq \delta \}\cap K} \left(1-e^{-\langle u, \xi\rangle} \right) \mu(d\xi)
\]
is a consequence of the monotonicity of the exponential map and
$\supp(\mu)\subseteq K$.
By dominated convergence, we have $\lim_{\delta \to 0}
R^\delta(u)= R(u)$ pointwise for each $u\in K^{\ast}$. Hence the
quasi-monotonicity carries over to $R$. Indeed, we
have for all $\delta >0$, $\langle R^\delta(v)-R^\delta(u),x\rangle\geq
0$. Thus
\[
\langle R^\delta(v)-R^\delta(u),x\rangle\rightarrow \langle R(v)-R(u),x\rangle\geq 0
\]
as $\delta \to 0$, which proves that $R$ is quasi-monotone
increasing.
\end{proof}

\subsection{The Generalized Riccati Equations}\label{sec:Riccati}

In order to prove the existence of affine processes for a given parameter set which satisfies the conditions of Proposition~\ref{prop:LKform} and
Proposition~\ref{prop:necessaryadm}, we shall heavily rely on the following existence and uniqueness
result for the generalized Riccati equations~\eqref{eq:RiccatiF} and~\eqref{eq:RiccatiR}, where
$F$ and $R$ are given by~\eqref{eq:FLevyK} and~\eqref{eq:RLevyK}.
Indeed, in the case of general proper convex cones, this allows us to prove
existence of affine pure jump processes (see Section~\ref{sec:existencejump}).
In the particular case of affine processes on symmetric cones, which we study in Section~\ref{sec:symcone},
we obtain, using Proposition~\ref{prop_ricc_sol} below, existence of affine processes
for any given parameter set (see Section~\ref{sec:existencesymcone}).

For the analysis of the generalized Riccati
equations~\eqref{eq:RiccatiF} and~\eqref{eq:RiccatiR} we shall use
the concept of quasi-monotonicity, as introduced above, several
times. Indeed, the proof of Proposition~\ref{prop_ricc_sol} below is
based to a large extent on the methods applied in~\citet[Proposition
5.3]{cfmt} which rely on the following comparison result for
ordinary differential equations (see~\citet{Volkmann1973}).

\begin{theorem}\label{th: Volkmann}
Let $U\subset V$ be an open set. Let $f\colon [0,T)\times U\to
V$ be a continuous locally Lipschitz map such that $f(t,\cdot)$ is
quasi-monotone increasing on $U$ for all $t\in [0,T)$. Let
$0<t_0\leq T$ and $g,h: [0,t_0)\to U$ be differentiable maps
such that $g(0)\preceq h(0)$ and
\[
\partial_t g(t)-f(t,g(t))\preceq \partial_t h(t)-f(t,h(t)), \quad \quad 0\leq
t<t_0.
\]
Then we have $g(t)\preceq h(t)$ for all $t\in [0,t_0)$.
\end{theorem}

The following estimate is needed to establish the existence of a global solution of~\eqref{eq:RiccatiR}.
Let us remark that a slightly stronger statement is proved in~\citet[Lemma 5.2]{cfmt} for the cone of positive semidefinite matrices, which however uses
the self-duality of the cone explicitly.

\begin{lemma}\label{lemKestR}
Let $R$ be of form~\eqref{eq:RLevyK} for some $Q$, $B^{\top}$, $\gamma$ and $\mu$
satisfying the conditions of Proposition~\ref{prop:LKform} (iv)-(vii). Then
\begin{align*}
R(u)&\preceq B^\top(u)+ \gamma+ \mu(K\cap\{\|\xi\| > 1\}).
\end{align*}
\end{lemma}

\begin{proof}
We may assume without loss of generality that the truncation
function $\chi$ takes the
form $\chi(\xi)=1_{\{\|\xi\| \leq 1\}}\xi$. Then, for all $u\in K^{\ast}$, we have
\begin{equation*}
\begin{split}
R(u)&=-\frac{1}{2}Q(u,u)+B^\top(u)+\gamma-\int_{K\cap\{\|\xi\| \leq
1\}}\underbrace{\left(e^{-\langle
u,\xi\rangle}-1+\langle \xi,u\rangle\right)}_{\geq 0}\mu(d\xi)\\
&\quad-\int_{K\cap\{\|\xi\| >
1\}}\left(e^{-\langle
u,\xi\rangle}-1\right)\mu(d\xi)\\
&\preceq -\frac{1}{2}Q(u,u)+B^\top(u)+\gamma+\mu(K\cap\{\|\xi\| > 1\})\\
&\preceq B^\top(u)+\gamma+\mu(K\cap\{\|\xi\| > 1\}),
\end{split}
\end{equation*}
where we use
$
-\int_{K\cap\{\|\xi\| > 1\}} \left(e^{-\langle
u,\xi\rangle}-1\right)\mu(d\xi)\preceq
\int_{K\cap\{\|\xi\| > 1\}} \mu(d\xi).
$
\end{proof}

Here is our main existence and uniqueness result for the generalized
Riccati differential equations~\eqref{eq:RiccatiF}--\eqref{eq:RiccatiR}.

\begin{proposition}\label{prop_ricc_sol}
Let $F$ and $R$ be of form~\eqref{eq:FLevyK} and~\eqref{eq:RLevyK} such that the conditions of
Proposition~\ref{prop:LKform} and~\ref{prop:necessaryadm} are satisfied. Then, for every $u\in \mathring{K^{\ast}}$,
there exists a unique global $\re_+\times
\mathring{K^{\ast}}$-valued solution $(\phi,\psi)$ of~\eqref{eq:RiccatiF}--\eqref{eq:RiccatiR}.
Moreover, $\phi(t,u)$ and $\psi(t,u)$ are real-analytic in $(t,u)\in \re_+\times \mathring{K^{\ast}}$.
\end{proposition}

\begin{proof}
We only have to show that, for every $u\in \mathring{K^{\ast}}$, there exists a
unique global $\mathring{K^{\ast}}$-valued solution $\psi$ of \eqref{eq:RiccatiR},
since $\phi$ is then uniquely determined by integrating
\eqref{eq:RiccatiF}.

Let $u\in \mathring{K^{\ast}}$. Since $R$ is real-analytic on $\mathring{K^{\ast}}$ (see, e.g.,~\citet[Lemma A.2]{dfs}), standard
ODE results (see, e.g.,~\citet[Theorem 10.4.5]{dieu_69}) yield that there
exists a unique local $\mathring{K^{\ast}}$-valued solution $\psi(t,u)$ of
\eqref{eq:RiccatiR} for $t\in [0,t_{\infty}(u))$, where
\[
t_{\infty}(u)=\lim_{k\to\infty} \inf\{ t\geq 0\mid \text{$\|\psi(t,u)\|\ge k$ or
$\psi(t,u)\in\partial K^{\ast}$}\}\leq \infty.
\]
It thus remains to show that $t_{\infty}(u)=\infty$. Real-analyticity of $\psi(t,u)$ and $\phi(t,u)$
in $(t,u)\in \re_+\times \mathring{K}^{\ast}$ then follows from~\citet[Theorem 10.8.2]{dieu_69}.

Since $R$ may not be Lipschitz continuous at $\partial K^{\ast}$, we first have to regularize it.
We thus define
\begin{align*}
\widetilde{R}(u)&=-\frac{1}{2}Q(u,u)+B^\top(u)+\gamma
-\int_{K\cap \{\|\xi\| \leq
1\}}\left(e^{-\langle
u,\xi\rangle}-1+\langle\xi,u\rangle\right)\mu(d\xi).
\end{align*}
Then $\widetilde{R}$ is real-analytic
on $V$. Hence, for all $u\in V$, there exists a unique local
$V$-valued solution $\widetilde{\psi}$ of
\[
\frac{\partial\widetilde{\psi}(t,u)}{\partial
t}=\widetilde{R}(\widetilde{\psi}(t,u)), \quad
\widetilde{\psi}(0,u)=u  ,
\]
for all $t\in [0,\widetilde{t}_{\infty}(u))$ with maximal lifetime
\[
\widetilde{t}_{\infty}(u)=\lim_{k\to\infty}\inf \{ t\geq 0\mid \text{$\|\widetilde{\psi}(t,u)\|\geq k$}\}\leq \infty.
\]
Consider now the normal cone of $K^{\ast}$ at $u \in \partial K^{\ast}$, consisting of inward pointing normal vectors, that
is,
\[
 N_{K^{\ast}}(u)=\{x \in K\, |\, \langle u, x \rangle =0\}, \quad u \neq 0,
\]
and $ N_{K^{\ast}}(0)=K$ (see, e.g.,~\citet[Example III.5.2.6]{hiriart}).
The conditions of Proposition~\ref{prop:necessaryadm} thus imply that
\[
 \langle \widetilde{R}(u), x \rangle \geq 0,
\]
for all $x \in N_{K^{\ast}}(u)$. Since $\widetilde{R}$ is clearly Lipschitz continuous, it follows from~\citet[Theorem III.10.XVI]{walter}
that $\widetilde{\psi}(t,u) \in K^{\ast}$ for all $t < \widetilde{t}_{\infty}(u)$ and $u \in K^{\ast}$.

Let us now define $y$ satisfying
\begin{align}\label{eq:linODE}
  \frac{\partial y(t,u)}{\partial t}= B^\top(y(t,u))+\gamma, \quad y(0,u)=u.
\end{align}
Then we have by Lemma~\ref{lemKestR} for $t < \widetilde{t}_{\infty}(u)$
\[
0=\frac{\partial\widetilde{\psi}(t,u)}{\partial
t}-\widetilde{R}(\widetilde{\psi}(t,u)) = \frac{\partial
y(t,u)}{\partial t}-B^\top(y(t,u))-\gamma \preceq \frac{\partial
y(t,u)}{\partial t}-\widetilde R(y(t,u)).
\]
Volkmann's comparison Theorem~\ref{th: Volkmann} thus implies for all $x \in K$
\[
 \langle \widetilde{\psi}(t,u),x\rangle \leq \langle y(t,u), x\rangle , \quad t\in [0,\widetilde{t}_{\infty}(u)).
\]
As $ \widetilde{\psi}(t,u) $ lies in $K^{\ast}$ up to its lifetime, the left hand side is nonnegative for all $x \in K$. By the very definition of
the dual cone, we therefore have
\[
 \widetilde{\psi}(t,u) \preceq y(t,u) , \quad t\in [0,\widetilde{t}_{\infty}(u)).
\]
Moreover, the affine ODE~\eqref{eq:linODE} admits a global solution. Since in finite dimensions any proper closed convex cone (in particular $K^\ast$) is normal,
that is, there exists a constant $\gamma_{K^\ast}$ such that
\[
0\preceq x\preceq y\Rightarrow \|x\|\leq \gamma_{K^\ast}\|y\|,
\]
we have
\[
\| \widetilde{\psi}(t,u) \|\leq \gamma_{K^\ast}\|y(t,u)\|<\infty.
\]
Hence we conclude that $\widetilde{t}_{\infty}(u)=\infty$ for all $u\in K^{\ast}$.

Moreover, by Proposition~\ref{prop:quasimono}, $\widetilde{R}$ is quasi-monotone increasing on
$K^{\ast}$. Hence another application of Theorem~\ref{th: Volkmann} yields
\[
0\preceq\widetilde{\psi}(t,u)\preceq \widetilde{\psi}(t,v),\quad t \geq 0, \quad\text{for all $0\preceq u\preceq v$.}
\]
Therefore and since $\widetilde{\psi}(t,u)$ is
also real-analytic in $u$, Lemma~\ref{lem:psiint} implies that
$\widetilde{\psi}(t,u) \in \mathring{K^{\ast}}$ for all $(t,u) \in \re_+ \times
\mathring{K^{\ast}}$.

We now carry this over to $\psi(t,u)$ and assume without loss of
generality, as in the proof of Lemma~\ref{lemKestR}, that the
truncation function $\chi$
takes the form $\chi(\xi)=1_{\{\|\xi\| \leq 1\}}\xi$. Then
\[
R(u)-\widetilde{R}(u)=-\int_{K\cap \{\|\xi\| > 1\}}\left(e^{-\langle
u,\xi\rangle}-1\right)\mu(d\xi)\succeq 0,\quad u\in K^{\ast}.
\]
Hence, for $u\in \mathring{K^{\ast}}$ and $t<t_{\infty}(u)$, we have
\[
0=\frac{\partial\widetilde{\psi}(t,u)}{\partial
t}-\widetilde{R}(\widetilde{\psi}(t,u)) = \frac{\partial
\psi(t,u)}{\partial t}-R(\psi(t,u)) \preceq \frac{\partial
\psi(t,u)}{\partial t}-\widetilde R(\psi(t,u)).
\]
Theorem~\ref{th: Volkmann} thus implies
\[
\psi(t,u)\succeq \widetilde{\psi}(t,u) \in \mathring{K^{\ast}},\quad t\in
[0,t_{\infty}(u)).
\]
Hence $t_{\infty}(u)=\lim_{k\to\infty}\inf \{ t\geq 0\mid \textrm{$\|\psi(t,u)\|\ge k$}\}$.
Using again Lemma~\ref{lemKestR} and the comparison argument with an affine ODE of the form
\[
  \frac{\partial y(t,u)}{\partial t}= B^\top(y(t,u))+\gamma+\mu(K \cap \{\|\xi\| >1\}, \quad y(0,u)=u,
\]
we conclude that
$t_{\infty}(u)=\infty$, as desired.
\end{proof}

\subsection{Construction of Pure Jump Processes}\label{sec:existencejump}

For affine processes on generating convex proper cones without diffusion component, that is, $Q=0$, the existence question can
be handled entirely as in the case of affine processes on the canonical state space $\re_+^{m}\times \mathbb R^{n-m}$.
By following the lines of~\citet[Section 7]{dfs} and~\citet[Section 5.3]{cfmt}, we here
prove existence of affine pure jump processes for a given parameter set, which satisfies the conditions of
Proposition~\ref{prop:LKform} and Proposition~\ref{prop:necessaryadm} with the additional assumption $Q=0$.

We call a function $f: K^{\ast} \to \mathbb R$ of subordinator
L\'evy-Khintchine form on $K$ if
\begin{align*}
f(u)=\langle b,u\rangle-\int_{K}(e^{-\langle u,
\xi\rangle}-1)m(d\xi),
\end{align*}
where $b \in K$ and $m$ is a Borel measure supported on
$K$ such that
\[
\int_{K}\left(\|\xi\|\wedge 1\right)m(d\xi) <
\infty.
\]
Recall that a distribution on $K$ is infinitely divisible if and
only if its Laplace transform takes the form $e^{-f(u)}$, where $f$
is of the above form. This means -- similarly as in the case of
$\re_+$ -- that L\'evy processes on proper cones
can only be of finite variation.

As in~\citet{dfs}, let us introduce the sets
\begin{align*} \mathcal
C&:=\{f+c\,\,|\,f:\,K^{\ast} \to \mathbb R\textit{ is of
L\'evy-Khintchine form on}\;\;K \,, c\in \re_+\},\\
\mathcal C_{S}&:=\{\psi\,| \, u\mapsto\langle \psi(u),
x\rangle\in\mathcal C\quad\textit{for all } x\in K\;\}.
\end{align*}

The following assertion can be obtained easily by
mimicking the proofs of~\citet[Proposition 7.2 and Lemma 7.5]{dfs}.

\begin{lemma}\label{lemma: three dfs statements}
We have,
\begin{enumerate}
\item $\mathcal C$, $\mathcal C_{S}$ are convex cones in
$C(K^{\ast})$.
\item $\phi\in \mathcal C$, $\psi\in\mathcal C_{S}$ imply $\phi(\psi)\in\mathcal C$.
\item $\psi,\psi_1\in\mathcal C_{S}$ imply $\psi_1(\psi)\in\mathcal C_{S}$.
\item If $\phi_k\in\mathcal C$ converges pointwise to a continuous function $\phi$ on $\mathring{K}^{\ast}$, then $\phi\in\mathcal C$
and $\phi$ has a continuous extension to $K^{\ast}$. A similar statement holds for
sequences in $\mathcal C_{S}$.
\item \label{stat: 6}  Let $R$ be of form~\eqref{eq:RLevyK} and let $R^{\delta}$ be defined as in~\eqref{eq:quasimono}
such that the involved parameters satisfy the conditions of Proposition~\ref{prop:necessaryadm}.
Then $R^{\delta}$ converges to $R$ locally uniformly as
$\delta \rightarrow 0$.
\end{enumerate}
\end{lemma}

\begin{proposition}\label{prop: C CS semiflow}
Let $F$ and $R$ be of form~\eqref{eq:FLevyK} and~\eqref{eq:RLevyK} such that the involved parameters satisfy the conditions of
Proposition~\ref{prop:LKform} and~\ref{prop:necessaryadm}. Then, for all $t\geq 0$, the solutions $(\phi(t,\cdot),\psi(t,\cdot))$ of
\eqref{eq:RiccatiF} and~\eqref{eq:RiccatiR} lie in $(\mathcal C,\mathcal C_S)$.
\end{proposition}

\begin{proof}
Suppose first that
\begin{equation}\label{finite var}
\int_{K } \left(\|\xi\| \wedge 1\right) \langle x, \mu(d\xi)\rangle < \infty, \quad \textrm{for all } x \in K.
\end{equation}
Then equation~\eqref{eq:RiccatiR} is equivalent to the integral equation
\begin{equation}\label{eq: var of const}
\psi(t,u)=e^{\widetilde B^{\top} t}(u)+\int_0^te^{\widetilde B^{\top}
(t-s)}(\widetilde R(\psi(s,u))ds,
\end{equation}
where $R(u)=\widetilde R(u)+\widetilde B^{\top}(u)$ and
$\widetilde B^{\top} \in \mathcal L(V)$ is given by
\[
\widetilde B^{\top}(u):=B^{\top}(u)-\int_{K}\langle
\chi(\xi),u\rangle \mu(d\xi).
\]
Here, $e^{\widetilde{B}^{\top} t}(u)$ is the notation for the semigroup
induced by $\partial_t y(t,u)=\widetilde B^{\top}(y(t,u))$,
$y(0,u)=u$. Hence the variation of constants formula yields~\eqref{eq: var of const}.

Due to Proposition~\ref{prop:necessaryadm}~\ref{eq:inwardpoint_drift},
$\widetilde B^{\top}$ is a linear drift which is ``inward pointing'' at
the boundary of $K^{\ast}$. This in turn is equivalent to the fact that $e^{\widetilde{B}^{\top} t}$ maps
$K^{\ast}$ into $K^{\ast}$. Therefore $e^{\widetilde{B}^{\top}t} \in \mathcal{C}_S$
and since $\widetilde R(u)$ is given by
\[
\widetilde R(u)=\gamma-\int_{K}(e^{-\langle u,\xi\rangle}-1)\mu(d\xi)
\]
with $\mu$ satisfying~\eqref{finite var}, we also have $\widetilde R \in\mathcal C_S$.

By a classical fixed point argument, the solution $\psi(t,u)$ is the pointwise limit of the sequence $(\psi^{(k)}(t,u))_{k \in \mathbb{N}}$, for $(t,u) \in \re_+ \times \mathring{K}^{\ast}$, obtained by Picard iteration
\begin{align*}
\psi^{(0)}(t,u)&:=u,\\
\psi^{(k+1)}(t,u)&:=e^{\widetilde B^{\top} t}(u)+\int_0^te^{\widetilde B^{\top}
(t-s)}(\widetilde R(\psi^{(k)}(s,u))ds,
\end{align*}
and due to Lemma~\ref{lemma: three dfs statements}~(i) and (iii), $\psi^{(k)}(t,\cdot)$
lies in $\mathcal C_S$ for all $k \in \mathbb{N}$. In view of Lemma~\ref{lemma: three dfs statements}~(iv),
the limit $\psi(t,\cdot)$ thus lies in $\mathcal C_S$ as well and
there exists a unique continuous extension of $\psi$ on $\re_+ \times K^{\ast}$.
Since $F \in \mathcal{C}$, we have by Lemma~\ref{lemma: three dfs statements}~(ii)
\[
\phi(t,\cdot)=\int_0^t F(\psi(s,\cdot))ds\in \mathcal C.
\]
By applying Lemma~\ref{lemma: three dfs statements}~\ref{stat: 6}, the general case is then reduced to the former, since $\mu
1_{\{\|\xi\|\geq \delta\}}$ clearly satisfies~\eqref{finite var}.
\end{proof}

We are now prepared to prove existence of affine processes on generating convex proper cones
under the additional assumption $Q=0$:

\begin{proposition}\label{th: existence Markov2}
Suppose that the parameters $(Q=0, b, B, c, \gamma, m, \mu)$ satisfy the conditions of Proposition~\ref{prop:LKform} and Proposition~\ref{prop:necessaryadm}. Then there exists a unique affine process on $K$ such that~\eqref{eq:affineprocessK} holds for all $(t,u) \in \re_+ \times K^{\ast}$, where
$\phi(t,u)$ and $\psi(t,u)$ are solutions of~\eqref{eq:RiccatiF} and~\eqref{eq:RiccatiR} with $F$ and $R$ given by~\eqref{eq:FLevyK} and~\eqref{eq:RLevyK}.
\end{proposition}

\begin{proof}
By Proposition~\ref{prop: C CS semiflow},
$(\phi(t,\cdot),\psi(t,\cdot))$ lie in $(\mathcal C,\mathcal C_S)$.
Hence, for all $t\in \re_+$ and $x\in K$, there exists an infinitely divisible sub-stochastic
measure on $K$ with Laplace-transform $e^{-\phi(t,u)-\langle
\psi(t,u),x\rangle}$. Moreover, the Chapman-Kolmogorov equation holds in view
of the flow property of $\phi$ and $\psi$, which implies the assertion.
\end{proof}

\section{Parameter Conditions on Symmetric Cones}\label{sec:symcone}

We will now considerably strengthen our assumptions on the conic state space $K$ and assume that $K$ is a
symmetric cone. This allows us to refine the conditions found in
Proposition~\ref{prop:LKform} and
Proposition~\ref{prop:necessaryadm} such that we finally obtain
conditions which guarantee existence of affine processes on
symmetric cones. The focus lies in particular on the bilinear form
$Q$ corresponding to the linear diffusion part, on the linear jump
coefficient $\mu$ and on the constant drift part $b$. We here build
on the results obtained in the setting of positive semidefinite
matrices (see~\citet{cfmt},~\citet{mayerhoferpers}), while utilizing
to a larger extent the algebraic structure of the underlying
Euclidean Jordan algebra.

Throughout this section we always suppose that $X$ is an affine process on some irreducible symmetric cone $K$ and $V$ denotes the associated simple Euclidean Jordan algebra of dimension $n$ and rank $r$, equipped with the natural scalar product
\[
\langle \cdot, \cdot \rangle : V \times V \to \re, \quad \langle x, y\rangle := \tr(x \circ y).
\]
For the notion of the rank $r$ and the trace, denoted by $\tr$, we refer to Appendix~\ref{sec:det}.
We shall also use the Peirce invariant $d$ corresponding to the dimension of $V_{ij},\, i< j$, as defined in~\eqref{Peirce decomposition 2}. In our case of a simple Euclidean Jordan algebra, we then have $n=r+\frac{d}{2}r(r-1)$.
For the precise definition of these notions we refer to Appendix~\ref{appendix:EJA}.

\subsection{The Diffusion Coefficient in Euclidean Jordan Algebras}

The next proposition establishes a direct relation between the bilinear form $Q$
satisfying the condition of Proposition~\ref{prop:necessaryadm} and the quadratic representation of $V$.
For its proof we use the Peirce and spectral decomposition of an Euclidean Jordan algebra, as introduced in Appendix~\ref{sec:Peirce}.

\begin{proposition}\label{prop:alpha}
Let $V$ be a simple Euclidean Jordan algebra of rank $r$ and let $ Q: V \times V \rightarrow V$ be a symmetric
bilinear function with $Q(v,v) \in K$ for all $v \in V$. Then Condition (i)
of Proposition~\ref{prop:necessaryadm}, that is,
\begin{align}\label{eq:Condprop}
 \langle x, Q(u,v) \rangle=0 \textrm { for all } v \in V \textrm{ and } u, x \in K \textrm{ with }  \langle u, x\rangle=0,
\end{align}
 is satisfied if and only if
\[
Q(u,u) =4 P(u)\alpha.
\]
Here, $P(u)$ is the quadratic representation of the Jordan algebra $V$, defined in~\eqref{Eq:quadratic_product}
and $\alpha \in K$ is determined by $4\alpha = Q(e,e)$.
\end{proposition}

\begin{proof}
We first assume that~\eqref{eq:Condprop} is satisfied.
Let $u \in V$ be fixed. Then there exists a Jordan frame $p_1, \dotsc, p_r$ (see Appendix~\ref{sec:Peirce} for the precise definition) such that
its spectral decomposition is given by $u = \sum_{i=1}^r \lambda_i p_i$.

As $p_1 + \dotsm + p_r = e$, we can write $Q(e,e)$ as
\begin{align}\label{Eq:alpha_peirce}
Q(e,e) = Q(p_1,p_1) + Q(p_2,p_2)  + \dotsm + Q(p_r,p_r) + \sum_{i < j} 2Q(p_i,p_j).
\end{align}
We now show that~\eqref{Eq:alpha_peirce} is precisely the Peirce decomposition of $Q(e,e)$
 with respect to the Jordan frame $p_1, \dotsc, p_r$.
More precisely, we show that $Q(p_j,p_j) \in V_{jj}$ and $Q(p_i,p_j) \in V_{ij}$ for each $i,j \in \set{1, \dotsc, r}$.

Let $i \neq j$, then clearly $\langle p_i, p_j\rangle=0$. From~\eqref{eq:Condprop} we deduce that
\begin{equation}\label{Eq:Q_ortho}
 \langle p_i, Q(p_j,p_j)\rangle = 0.
\end{equation}
But $Q(p_j,p_j) \in K$ such that we can conclude using Lemma~\ref{Lem:orthogonality} that $Q(p_j,p_j) \circ p_i = 0$. Keeping $j$ fixed, we can subtract these equalities from
\[
Q(p_j,p_j) \circ e = Q(p_j,p_j),
\]
running through all $i \neq j$, and we arrive at $Q(p_j,p_j) \circ p_j = Q(p_j,p_j)$. This shows that $Q(p_j,p_j) \in V(p_j,1) = V_{jj}$ for all $j \in \set{1, \dotsc, r}$.

Let now $i,j,k$ be arbitrary in $\{1, \dotsc, r\}$, but all distinct. Using again~\eqref{eq:Condprop}, we see that $\langle p_i,Q(p_k + p_j,p_k+p_j)\rangle = 0$, and from Lemma~\ref{Lem:orthogonality} it follows that $Q(p_k + p_j,p_k+p_j) \circ p_i = 0$. Thus
\[
Q(p_k,p_j) \circ p_i = \frac{1}{2} \left(Q(p_k+p_j,p_k+p_j) \circ p_i - Q(p_k,p_k) \circ p_i - Q(p_j,p_j) \circ p_i \right) = 0
\]
for any distinct $i,j,k \in \set{1, \dotsc, r}$.
Keeping now $k$ and $j$ fixed, we can subtract the equalities $Q(p_k,p_j) \circ p_i = 0$ from the equality $Q(p_k,p_j) \circ e = Q(p_k,p_j)$, running through all $i$ distinct from both $j$ and $k$, and obtain $Q(p_k,p_j) \circ (p_k + p_j) = Q(p_k,p_j)$. For symmetry reasons we must have $Q(p_k,p_j) \circ p_k = Q(p_k,p_j) \circ p_j$ and we thus conclude that
\[
Q(p_k,p_j) \circ p_k = Q(p_k,p_j) \circ p_j = \frac{1}{2}Q(p_k,p_j).
\]
Equivalently, $Q(p_k,p_j) \in V(p_k,1/2) \cap V(p_j,1/2) = V_{kj}$. Hence we have shown that~\eqref{Eq:alpha_peirce} is the Peirce decomposition of $Q(e,e)$ with respect to the Jordan frame $p_1, \dotsc, p_r$.

Define $4\alpha : = Q(e,e)$. As the projection onto $V_{ii}$ is given by the quadratic representation
$P(p_i)$ and the projection onto $V_{ij}$ by
$4 L(p_i) L(p_j)$, we can write $Q(p_j,p_j) = 4P(p_j)\alpha$ and $2Q(p_i,p_j) = 16 L(p_i)L(p_j)\alpha$.
Therefore,
\begin{align*}
Q(u,u) &= \lambda^2_1 Q(p_1,p_1) + \dots + \lambda^2_r Q(p_r,p_r) + \sum_{i < j} 2\lambda_i \lambda_j Q(p_i,p_j) \\
 &= 4 \left(\lambda_1^2 P(p_1)\alpha + \dotsm + \lambda_r^2 P(p_r)\alpha + \sum_{i < j} \lambda_i \lambda_j 4 L(p_i) L(p_j) \alpha\right)\\
 &=4 P\left(\sum_{i=1}^r \lambda_i p_i\right) \alpha \\
 &= 4P(u) \alpha,
\end{align*}
and we have shown the first implication.

Concerning the other direction, let $Q$ be given by $Q(u,u)=4P(u)\alpha$ for some $\alpha \in K$. Using polarization, we then get
\[
Q(u,v)=2\left(P(u+v)-P(u)-P(v)\right)\alpha=4P(u,v)\alpha.
\]
Take now some $x,u \in K$ such that $\langle x, u\rangle =0$. By Lemma~\ref{Lem:orthogonality} (ii), we have $u \circ x=0$ and consequently
\[
 \langle x,u^2\rangle =\langle x \circ u,u \rangle =0,
\]
which in turn implies  $u^2 \circ x=0$. The definition of the quadratic representation thus yields
\[
 \langle P(u)\alpha,x \rangle =\langle \alpha, P(u) x \rangle=\langle \alpha, 2 u \circ( u \circ x)-u^2 \circ x \rangle =0.
\]
Since $L(x)$ and $L(u)$ commute, which is a consequence of~\citet[Proposition II.1.1 (i)]{Faraut1994},
we similarly get $\langle P(u+v)\alpha,x\rangle=\langle P(v)\alpha,x \rangle$. This proves the assertion.
\end{proof}

\subsection{Linear Jump Behavior in Euclidean Jordan Algebras}

In this section, we show that the linear jump coefficient $\mu$ satisfying Condition~(vii) of Proposition~\ref{prop:LKform} and Condition~(ii)
of Proposition~\ref{prop:necessaryadm} necessarily integrates $(\|\xi \| \wedge 1)$
if $r >1$ and $d >0$. The proof is based on an idea of~\citet{mayerhoferpers}, who showed the corresponding result for positive
semidefinite matrices.

\begin{proposition}\label{prop:jumpfinite}
Let $V$ be a simple Euclidean Jordan algebra with rank $r >1$ and Peirce invariant $d >0$. Suppose that $\mu$ is a $K$-valued $\sigma$-finite
Borel measure on $K$ satisfying $\mu(\{0\}) = 0$ and
\[
\int_K(\|\xi\|^2 \wedge 1)\langle x, \mu (d\xi)\rangle < \infty, \quad \textrm{for all } x \in K.
\]
Then Condition~(ii)
of Proposition~\ref{prop:necessaryadm}, that is,
\begin{align}\label{eq:jumpfinitevar}
 \int_K\langle \chi(\xi), u \rangle \langle x, \mu(d\xi)\rangle < \infty \textrm { for all } u, x \in K \textrm{ with }  \langle u, x\rangle=0,
\end{align}
implies
\[
 \int_K (\|\xi \| \wedge 1)\langle x, \mu (d\xi)\rangle < \infty, \quad \textrm{for all } x \in K.
\]
\end{proposition}

\begin{remark}
It follows from the above proposition that only in the case of $\mathbb{R}_+$ and the two-dimensional Lorentz cone, jumps of infinite total variation are possible. In all other cases we could now set the truncation function $\chi$ to be $0$ and adjust the linear drift accordingly.
However, in order to cover all irreducible cones, we shall keep the truncation function in the sequel.
\end{remark}

\begin{proof}
Let $p_1, \ldots, p_r$ be a fixed Jordan frame of $V$.
Corresponding to the Peirce decomposition~\eqref{Peirce decomposition 2}, we can write for every $z \in V$
\[
z=\sum_{i=1}^r z_ip_i + \sum_{i < j} z_{ij},
\]
where $z_i \in \mathbb{R}$ and $z_{ij} \in V_{ij}$.
Hence, for the $K$-valued measure $\mu$, we define positive measures $\mu_i,\, i \in \{1, \ldots, r\}$, and for $i \neq j$,
$V_{ij}$-valued measures $\mu_{ij}$.
Consider now, for some $i \neq j$, elements of the form
\[
 x=p_i+p_j+w, \quad u= p_i+p_j-w,
\]
with $w \in V_{ij}$ such that $\|w\|^2=2$. Here the
assumption $r > 1$ and $d >0$ enters, as we require $p_i \neq p_j$ and $w \neq 0$. Due
to~\cite[Proposition IV.1.4 and Theorem IV.2.1]{Faraut1994}, $x, u
\in  K$ and we have additionally $\langle u, x \rangle =0$. Assume
without loss of generality that $\chi(\xi)=1_{\{\|\xi\|\leq
1\}}\xi$. Since for every $y \in K$, $\langle y, \mu(\cdot)\rangle$
is a positive measure supported on $K$, we have
by~\eqref{eq:jumpfinitevar}
\begin{align*}
 0 &\leq \int_{\{\|\xi\|\leq 1\}}\langle \xi, u \rangle \langle x, \mu(d\xi) \rangle < \infty,\\
 0 &\leq \int_{\{\|\xi\|\leq 1\}}\langle \xi, x \rangle \langle u, \mu(d\xi)\rangle< \infty.
\end{align*}
Thus there exists a positive constant $C$ such that for all $\delta > 0$
\begin{align}
 0 &\leq \int_{\{\delta \leq \|\xi\|\leq 1\}}\langle \xi, u \rangle \langle x, \mu(d\xi) \rangle < C,\label{eq:intfin1}\\
 0 &\leq \int_{\{\delta  \leq\|\xi\|\leq 1\}}\langle \xi, x \rangle \langle u, \mu(d\xi)\rangle< C\label{eq:intfin2}.
\end{align}
Summing up~\eqref{eq:intfin1} and~\eqref{eq:intfin2} and using the orthogonality of the Peirce decomposition, then yields
\begin{multline}\label{eq:summeasure}
0 \leq\int_{\{\delta \leq \|\xi\|\leq 1\}} \left(\xi_i \mu_{i}(d\xi)-\frac{1}{2}\langle \xi_{ij}, w \rangle \langle w, \mu_{ij}(d\xi)\rangle +\xi_j \mu_{j}(d\xi)\rangle \right)\\
+\int_{\{\delta \leq \|\xi\|\leq 1\}}\left(\xi_i \mu_{j}(d\xi)-
\frac{1}{2}\langle \xi_{ij}, w \rangle \langle w, \mu_{ij}(d\xi)\rangle+ \xi_j \mu_{i}(d\xi) \right)< 2C.
\end{multline}
Since $\mu$ is a $K$-valued measure on $K$, we have by~\citet[Exercise IV.7 (b)]{Faraut1994} and the assumption $\|w\|^2=2$
\[
\frac{1}{2} \langle \xi_{ij}, w\rangle \langle w, \mu_{ij}(E)\rangle \leq \| \xi_{ij}\| \|\mu_{ij}(E)\|\leq 2 \sqrt{ \xi_i \xi_j \mu_i(E) \mu_j(E)},
\quad E \in \mathcal{B}(K),
\]
which implies that both integrals in~\eqref{eq:summeasure} are nonnegative. We can therefore conclude that both of them are finite:
\begin{align}
0 \leq &\int_{\{\delta \leq \|\xi\|\leq 1\}} \left(\xi_i \mu_{i}(d\xi)-\frac{1}{2}\langle \xi_{ij}, w \rangle \langle w, \mu_{ij}(d\xi)\rangle +\xi_j \mu_{j}(d\xi)\rangle \right)< 2C,\label{eq:firintfinite}\\
0 \leq &\int_{\{\delta \leq \|\xi\|\leq 1\}}\left(\xi_i \mu_{j}(d\xi)-
\frac{1}{2}\langle \xi_{ij}, w \rangle \langle w, \mu_{ij}(d\xi)\rangle+ \xi_j \mu_{i}(d\xi) \right)< 2C.\label{eq:secintfinite}
\end{align}
Moreover, as $\langle p_i,p_j\rangle =0$ for $i\neq j$, we have as a direct consequence of~\eqref{eq:jumpfinitevar}
\begin{align}\label{eq:measureonecomp}
 0 \leq \int_{\{\|\xi\|\leq 1\}}  \langle \xi, p_i\rangle \langle p_j, \mu(d\xi)\rangle = \int_{\{\|\xi\|\leq 1\}} \xi_i \mu_j(d\xi)\rangle < \infty, \quad i\neq j.
\end{align}
As above, there thus exists a positive constant $C_1$ such that for all $\delta > 0$
\begin{align}\label{eq:ijfinite}
\int_{\{\delta \leq \|\xi\|\leq 1\}} \xi_i \mu_j(d\xi) < C_1, \quad i\neq j.
\end{align}
Subtracting~\eqref{eq:ijfinite} from~\eqref{eq:secintfinite}, then yields for all $\delta > 0$
\[
-2C <\int_{\{\delta \leq \|\xi\|\leq 1\}} \frac{1}{2}\langle \xi_{ij}, w \rangle \langle w, \mu_{ij}(d\xi)\rangle < 2C_1.
\]
By~\eqref{eq:firintfinite}, we therefore have for all $\delta > 0$
\[
0 \leq \int_{\{\delta \leq \|\xi\|\leq 1\}} \left(\xi_i \mu_{i}(d\xi)+\xi_j \mu_{j}(d\xi)\right)< 2(C+C_1).
\]
Together with~\eqref{eq:measureonecomp}, this implies for all $i \in \{1, \ldots,r\}$
\[
0 \leq \int_{\{\|\xi\|\leq 1\}} \xi_i \mu_{i}(d\xi)< \infty.
\]
This then yields
\[
 \int_{K}(\|\xi\| \wedge 1) \langle x, \mu(d\xi)\rangle < \infty, \quad \textrm{for all } x \in K,
\]
and proves the assertion.
\end{proof}

\subsection{The Role of the Constant Drift}

This section is devoted to show that the constant drift term $b$, as defined in Proposition~\ref{prop:LKform}~\ref{eq:const_drift}, of any affine process $X$ on an irreducible symmetric cone $K$ necessarily satisfies
\[
 b \succeq d(r-1) \alpha,
\]
where $\alpha$ is defined in Proposition~\ref{prop:alpha}. Recall that
$d$ denotes the Peirce invariant and $r$ the rank of $V$.

Before we actually prove this result, let us introduce some notation.
We shall consider the tensor product $V \otimes V^{\ast}$, which we identify via the canonical isomorphism
\[
(u \otimes v)x=\langle x,v\rangle u, \quad x \in V,
\]
with the vector space of linear maps on $V$ denoted by $\mathcal{L}(V)$.
Moreover, for an element $A \in \mathcal{L}(V)$, we denote its trace by $\Tr(A)$.\footnote{
In order to distinguish between elements of $V$ and linear maps on $V$,
we use the notations $\Tr(A)$ and $\Det(A)$ for $A \in \mathcal{L}(V)$ and
$\tr(x)$ and $\det(x)$ for elements in $V$ (compare Remark~\ref{rem:dettrace}).}
 Observe that $\Tr(A (u\otimes u))=\langle u, A u\rangle$. Indeed, by choosing a basis $\{e_{\beta}\}$ of $V$, we have
\begin{align*}
\Tr(A (u\otimes u))&=\sum_{\beta} \langle A^{\top}e_{\beta}, (u \otimes u)e_{\beta}\rangle\\
&=\sum_{\beta} \langle A^{\top}e_{\beta}, \langle u, e_{\beta}\rangle u\rangle\\
&=\sum_{\beta} \langle u, e_{\beta}\rangle \langle A^{\top}e_{\beta},  u\rangle\\
&=\langle u, A u\rangle.
\end{align*}
Let now $A: K \to S_+(V) \subset \mathcal{L}(V)$ be the linear part of the diffusion characteristic, as introduced in~\eqref{eq:characteristicfkt}.
Recall that the symmetric bilinear function $Q$
was defined via~\eqref{eq:AQ}, that is,
\[
 \Tr(A(x)(u \otimes u))=\langle u, A(x) u \rangle =\langle x, Q(u,u)\rangle.
\]
As shown in Proposition~\ref{prop:alpha}, we have $Q(u,u)=4P(u)\alpha$ for some $\alpha \in K$. Hence
\begin{align}\label{Eq:Trace}
\Tr(A(x)(u \otimes u))=\langle u, A(x) u \rangle =4\langle x, P(u)\alpha\rangle.
\end{align}
Following~\citet[Section XIV.1]{Faraut1994}, we now define a second order differential operator $D$ on $C^2(V)$ for this expression, that is,
\[
D= \Tr\left(A(x)\left(\frac{\partial}{\partial x} \otimes \frac{\partial}{\partial x}\right)\right)=4\left\langle x, P\left(\frac{\partial}{\partial x}\right )\alpha\right\rangle.
\]
As usual, the polynomial $u \mapsto \sigma_D(x,u)=4\langle x, P(u)\alpha\rangle$, whose coefficients are linear functions in $x$, is called
\emph{symbol} $\sigma_D$ of the differential operator $D$ and we have
\[
De^{\langle u,x\rangle}=\sigma_D(x,u)e^{\langle u,x\rangle}.
\]
Let us finally introduce the following integro-differential operator for (complex-valued) $C_b^2(K)$-functions.
\begin{equation}
\begin{split}\label{eq: generator}
\mathcal{A}f(x)&=\frac{1}{2}\Tr\left(A(x) \left(\frac{\partial}{\partial x}\otimes \frac{\partial}{\partial x}\right)\right)f|_x+\left\langle b+B(x), \nabla f(x)\right\rangle
-(c+\langle\gamma,x\rangle)f(x)\\
&\quad+\int_{K}
\left(f(x+\xi)-f(x)\right)m(d\xi)\\
&\quad+\int_{K}\left(f(x+\xi)-f(x)-\langle
\chi(\xi), \nabla f(x)\rangle\right)\langle x, \mu(d\xi)\rangle,
\end{split}
\end{equation}
where $A(x)$ satisfies~\eqref{Eq:Trace}.
The other parameters are specified in Proposition~\ref{prop:LKform} and are supposed to satisfy
the conditions of Proposition~\ref{prop:necessaryadm} and Propsition~\ref{prop:jumpfinite}.
Note that for the family of functions $\{e^{-\langle u,x\rangle} \,|\, u \in K\}$ this expression corresponds to the pointwise $t$-derivative of $P_te^{-\langle u,x\rangle}$ at $t=0$. This is simply a consequence of the form of $F$ and $R$, since
\[
\lim_{t\downarrow 0}\frac{\left(P_te^{-\langle u,x\rangle}-e^{-{\langle u,x\rangle}}\right)}{t}=(-F(u)-\langle R(u),x\rangle)e^{-\langle u,x\rangle}=\mathcal{A}e^{-\langle u,x\rangle}
\]
for every $x \in K$.

The following lemma is proved by means of the L\'evy--Khintchine formula on $\re_+$ similarly as in~\citet[Lemma 4.15]{cfmt} and is
related to the positive maximum principle for the operator $\mathcal{A}$.

\begin{lemma}\label{lemprop: characteristicsdet}
Let $X$ be an affine process on $K$ with constant drift parameter $b$ and linear diffusion part $Q$, as defined in Proposition~\ref{prop:LKform} (i) and (iv). Moreover, suppose that $Q$ satisfies $Q(u,u)=4P(u)\alpha$ for some $\alpha \in K$. Then, for any $y \in \partial K$, we have
\begin{equation}\label{Eq: Det positive}
\begin{split}
&\langle b, \nabla \det (y)\rangle +2\left\langle y, P\left(\frac{\partial}{\partial x}\right)\alpha \right\rangle\det|_{y}\\
&=\langle b, \nabla \det (y)\rangle +\frac{1}{2}\Tr\left(A(y)\left(\frac{\partial}{\partial x}\otimes \frac{\partial}{\partial x}\right) \right)\det|_{y}\geq 0.
\end{split}
\end{equation}
Here, $\det(y)$ denotes the determinant of an element $y \in V$, as defined in
Appendix~\ref{sec:det}, and $A(x)$ is the linear part of the diffusion characteristic, which satisfies
\begin{align}\label{eq:rel_Aalpha}
 \Tr(A(x)(u \otimes u))=\langle u, A(x) u \rangle =\langle x, Q(u,u)\rangle=4\langle x, P(u)\alpha\rangle
\end{align}
for all $u, x \in K$.
\end{lemma}

\begin{proof}
We follow the proof of~\citet[Lemma 4.15]{cfmt}.
Let $y \in \partial K$ and let $f \in C_c^{\infty}(V)$ be a
function with $f \geq 0$ on $K$ and $f(x)=\det(x)$ for all $x$ in a
neighborhood of $y$. Then, for any $v \in \re_+$, the function
$x\mapsto (e^{-vf(x)}-1)$ lies in $C_c^{\infty}(V)$, hence in particular in
$\mathcal{S}_n$, the space of rapidly decreasing $C^{\infty}$-functions on $V$.
As the Fourier transform is a linear isomorphism on $\mathcal{S}_n$, we can write
\[
e^{-vf(y)}-1=\int_V e^{\im \langle q,y \rangle} g(q) dq
\]
for some $g \in \mathcal{S}_n$.
As a consequence of~\citet[Theorem 6.4]{cuchteich} or~\citet[Theorem 3.10]{kst1} and Remark~\ref{rem:eqvdef}, we obtain by dominated convergence
\begin{align*}
\lim_{t\downarrow 0}\frac{P_t(e^{-vf(y)}-1)}{t}&=\partial_t|_{t=0}P_t(e^{-vf(y)}-1)=\int_V \partial_t|_{t=0}P_te^{\im \langle q,y \rangle} g(q) dq \\
&=\int_V  \left(-F(-\im q)-\langle R(-\im q),y\rangle\right)e^{\im \langle q,y \rangle} g(q) dq\\
&=\int_V  \mathcal{A}e^{\im \langle q,y \rangle} g(q) dq=\mathcal{A}(e^{-vf(y)}-1),
\end{align*}
where $\mathcal{A}$ is defined in \eqref{eq: generator} and thus satisfies $(-F(-\im q)-\langle R(-\im q),x\rangle)e^{\im \langle q,x \rangle} = \mathcal{A}e^{\im \langle q,x \rangle}$. Hence the limit
\begin{equation}\label{Eq:Laplace Poisson}
\begin{split}
\mathcal{A}(e^{-vf(y)}-1) &=\lim_{t\downarrow
0}\frac{1}{t}\int_{K}(e^{-vf(\xi)}-1)p_t(y,d\xi)\\
&=\lim_{t\downarrow
0}\frac{1}{t}\int_{\re_+}(e^{-vz}-1)p^f_t(y,dz),
\end{split}
\end{equation}
exists for any $v\in \re_+$, where $p_t^f(y,dz)=f_{\ast}p_t(y,dz)$ is the pushforward of
$p_t(y,\cdot)$ under $f$, which is a probability measure supported on $\re_+$. Using the same arguments as in
the proof of Proposition~\ref{prop:LKform}, we see that, for every fixed $t> 0$, the right hand side of \eqref{Eq:Laplace Poisson} is the logarithm of the Laplace transform of a compound Poisson distribution supported on $\re_+$ with intensity $1/t$ and compounding distribution $p^f_t(y,dz)$. The pointwise convergence of \eqref{Eq:Laplace Poisson}  for $t \rightarrow 0$ to some function being continuous at $0$ implies weak convergence of the compound Poisson distributions to some infinitely divisible probability distribution supported on $\re_+$. Its Laplace transform is then given as the exponential of the left hand side of \eqref{Eq:Laplace Poisson}.

Using now $f(y)=0$ and the form of $\mathcal{A}$ given by \eqref{eq: generator}, we have
\begin{equation}\label{Atildelk}
\begin{aligned}
v&\mapsto\mathcal{A}(e^{-vf(y)}-1)\\
&=2v^2\left\langle y, P(\nabla f(y))\alpha \right\rangle-2v\left\langle y, P\left(\frac{\partial}{\partial x}\right)\alpha \right\rangle f|_{y} \\
&\quad-v \langle b+B(y), \nabla f(y)\rangle +\int_{K}\left(e^{-vf(y+\xi)}-1\right)m(d\xi)\\
&\quad+\int_{K}\left(e^{-vf(y+\xi)}-1+v\langle \chi(\xi), \nabla f(y)\rangle \right)\langle y, \mu (d\xi)\rangle\\
\end{aligned}
\end{equation}
Note now that $\nabla f(y)=\nabla \det(y)$ and that $\langle \nabla\det(y),y\rangle=0$ such that the Admissibility Conditions~\ref{eq:jump_int} and~\ref{eq:inwardpoint_drift} of Proposition~\ref{prop:necessaryadm} imply \footnote{
By Proposition~\ref{prop:jumpfinite}, we have $\int_{K} \|\chi(\xi) \| \langle y, \mu(d\xi)\rangle< \infty$ if $r >1$ and $d>0$. This means that the above argument using
$\langle \nabla\det(y),y\rangle=0$ is only relevant in the two-dimensional Lorentz cone. Observe that for $K=\re_+$, $y=0$ anyway.}
\[
\int_{K} \langle \chi(\xi), \nabla \det(y)\rangle \langle y, \mu\rangle(d\xi)< \infty
\]
and
\begin{align*}\label{Eq:Bdet}
B_0(y)=\langle y,B^{\top}(\nabla \det(y))\rangle-\int_{K} \langle \chi(\xi), \nabla \det(y)\rangle \langle y, \mu\rangle(d\xi)\geq 0.
\end{align*}
By the L\'evy--Khintchine formula on $\re_+$, $\left\langle y, P(\nabla f(y))\alpha \right\rangle$ has to vanish,
which is the case due to Proposition~\ref{prop:alpha} and the fact that $\langle \nabla\det(y),y\rangle=0$.
Moreover, the coefficient of $v$ in~\eqref{Atildelk} has to be non-positive, that is,
\[
2\left\langle y, P\left(\frac{\partial}{\partial x}\right)\alpha  \right\rangle f|_{y}+ \langle b+B(y), \nabla f(y)\rangle-\int_{K\setminus\{0\}} \langle \chi(\xi), \nabla f(y)\rangle \langle y, \mu\rangle(d\xi)\geq 0.
\]
Observing that $y \mapsto B_0(y)$ is a polynomial of degree $r$, being positive for every $y \in \partial K$, and that the polynomial
\[
y \mapsto \langle b, \nabla \det (y)\rangle +2\left\langle y, P\left(\frac{\partial}{\partial x}\right)\alpha\right\rangle  \det|_{y}
\]
is of degree $r-1$, we obtain equation~\eqref{Eq: Det positive}.
\end{proof}

\begin{proposition}\label{Prop:drift}
Let $X$ be an affine process on $K$ with constant drift parameter $b\in K$ and diffusion parameter $\alpha\in K$, which defines
$Q(u,u)$ through $Q(u,u)=4P(u)\alpha$.  Then
\[
 b \succeq d(r-1) \alpha,
\]
where $d$ denotes the Peirce invariant and $r$ the rank of $V$.
\end{proposition}

\begin{proof}
From Lemma \ref{lemprop: characteristicsdet} we have the necessary condition
\[
\langle b, \nabla \det (y)\rangle +\frac{1}{2}\Tr\left(A(y)\left(\frac{\partial}{\partial x}\otimes \frac{\partial}{\partial x}\right) \right)\det|_{y}\geq 0
\]
for any $y \in \partial K$.
For $x \in \mathring K$ we can calculate the left hand side.
Since $\nabla \det(x)=\det(x)x^{-1}$ and $\frac{d}{dt}(x+tu)^{-1}|_{t=0}=-P(x^{-1})u$
(see Proposition~\ref{prop:EJA}~\ref{item:EJA3} and~\ref{item:EJA2}), we have
\begin{multline*}
 \langle b, \nabla \det (x)\rangle +\frac{1}{2}\Tr\left(A(x)\left(\frac{\partial}{\partial x}\otimes \frac{\partial}{\partial x}\right)\right) \det|_{x}\\
=\det (x)\left(\left\langle x^{-1},b \right\rangle +\frac{1}{2}\Tr\left(A(x)\left(x^{-1}\otimes x^{-1}\right)\right)-\frac{1}{2}\Tr\left(A(x)P\left(x^{-1}\right)\right)\right).
\end{multline*}
Using \eqref{eq:rel_Aalpha}, Proposition~\ref{prop:EJA}~\ref{item:EJA1} and Lemma~\ref{lem:Trace} below, we thus obtain
\begin{align*}
&\det (x)\left(\left\langle  x^{-1},b \right\rangle +2\left\langle x, P\left(x^{-1}\right)\alpha\right\rangle -\frac{1}{2}\Tr\left(A(x)P\left(x^{-1}\right)\right)\right)\\
&\qquad=\det (x)\left(\left\langle x^{-1},b \right\rangle +2\left\langle x^{-1}, \alpha\right\rangle -2\frac{n}{r}\left\langle x^{-1}, \alpha\right\rangle\right)\\
&\qquad=\det(x)\left(\left\langle x^{-1},b \right\rangle-d(r-1)\left\langle x^{-1}, \alpha\right\rangle\right)\\
&\qquad=\det(x)\left\langle x^{-1},b-d(r-1)\alpha \right\rangle.
\end{align*}
As $\det(y)y^{-1}$ is also well-defined on $\partial K$ as derivative of $\det(y)$, Condition~\eqref{Eq: Det positive} implies
\[
 b \succeq d(r-1)\alpha.
\]
\end{proof}

The following lemma is needed in the proof of the above proposition and allows us to express
$
 \Tr\left(A(x) P\left(x^{-1}\right)\right)
$
in terms of $\alpha$.
\begin{lemma}\label{lem:Trace}
Let $V$ be a simple Euclidean Jordan algebra of rank $r$ and with scalar product $\langle x, y\rangle=\tr(x \circ y)$ and let $A(x)$ be defined by~\eqref{eq:rel_Aalpha}. Then
\begin{align}\label{Eq: trace equality}
 \Tr\left(A(x) P\left(x^{-1}\right)\right)=4\frac{n}{r} \left\langle x^{-1},\alpha\right\rangle
\end{align}
for any invertible $x \in V$.
\end{lemma}

\begin{proof}
 Let $p_1,\ldots, p_r$ be a Jordan frame of $V$. Then the spectral decomposition of an arbitrary element $x$ is given by $x=\sum_{i=1}^r \lambda_i p_i$, and $P(x^{-1})$ can be written as
\begin{align}\label{Eq:Px}
 P\left(x^{-1}\right)=\sum_{i=1}^r \lambda_i^{-2}P(p_i)+\sum_{i <j}4\lambda_i^{-1}\lambda_j^{-1}L(p_i)L(p_j).
\end{align}
Let now $\{e_{\beta}\}$ be an orthonormal basis of $V$,
where the basis elements are chosen to lie in the subspaces corresponding to the Peirce decomposition,
as described in Section~\ref{sec:Peirce}. More precisely, for each $i\in \{1, \ldots, r\}$,
we choose one basis element in $V_{ii}$, which is in fact $p_i$, and for each $i < j$, we choose $d$ basis elements in
$V_{ij}$, since the dimension of $V_{ij}$ is $d$.

By the definition of the trace $\Tr$, we have
\[
 \Tr\left(A(x) P(x^{-1})\right)=\sum_{\beta}\left\langle A(x)e_{\beta}, P(x^{-1}) e_{\beta}\right\rangle.
\]
In order to evaluate $P(x^{-1})e_{\beta}$, we shall use
\begin{align*}
 V_{ii}&=\{x \in V \, |\, L(p_k)x=\delta_{ik}x\},\\
 V_{ij}&=\left\{x \in V \, |\, L(p_k)x=\frac{1}{2}(\delta_{ik}+\delta_{jk})x\right\},
\end{align*}
as derived in the proof of~\citet[Theorem IV.2.1]{Faraut1994}.
This implies for $e_{\beta} \in V_{ij},\, i\leq j$,
\begin{align}\label{Eq:Lpkebeta}
L(p_k)e_{\beta} =\frac{1}{2}(\delta_{ik}e_{\beta}+\delta_{jk}e_{\beta}),
\end{align}
and hence for $k < l$,
\begin{align*}
L(p_l)L(p_k)e_{\beta}= \left\{
\begin{array}{rl}
 \frac{1}{4}e_{\beta}, &  \textrm{if } e_{\beta} \in V_{kl} ,\\
0, &\textrm{otherwise},\end{array}\right. \qquad
P(p_k)e_{\beta}= \left\{
\begin{array}{rl}
 e_{\beta}, & \textrm{if } e_{\beta} \in V_{kk},\\
0, &\textrm{otherwise}.\end{array}\right.
\end{align*}
Note that this is obvious, since $P(p_k)$ and $4L(p_l)L(p_k)$ are the orthogonal projections on $V_{kk}$ and $V_{kl}$
respectively (see Section~\ref{sec:Peirce}).

Let now $e_{\beta} \in V_{ij}$ for some $i\leq j$ be fixed. Then, using~\eqref{Eq:Px}, the linearity of $A$ and~\eqref{eq:rel_Aalpha}, we obtain
\begin{align}
\left\langle A(x)e_{\beta}, P\left(x^{-1}\right) e_{\beta}\right\rangle&=\left\langle \sum_{k=1}^r\lambda_{k}A(p_k)e_{\beta},\lambda_i^{-1}\lambda_j^{-1}e_{\beta}\right\rangle\notag\\
&=\sum_{k=1}^r\lambda_k\lambda_i^{-1}\lambda_j^{-1}4\langle P(e_{\beta})p_k, \alpha\rangle\notag\\
&=2\left(\left\langle\lambda_i^{-1}p_i,\alpha\right\rangle+\left \langle\lambda_j^{-1}p_j,\alpha\right\rangle\right).\label{Eq:eval}
\end{align}
Here, the last equality follows from
\begin{align}\label{Eq:Pebetaek}
 P(e_{\beta})p_k=\frac{1}{2}(\delta_{ik}p_j+\delta_{jk}p_i),
\end{align}
for $e_{\beta} \in V_{ij},\,i\leq j$. For $e_{\beta} \in V_{ii}$, \eqref{Eq:Pebetaek} is simply a consequence of~\eqref{Eq:Lpkebeta} and for $e_{\beta} \in V_{ij}, i< j$, we have by~\citet[Proposition IV.1.4 (i)]{Faraut1994}
\[
P(e_{\beta})p_k=e_{\beta}^2\circ(\delta_{ik}e+\delta_{jk}e-p_k)=\frac{1}{2}(p_i+p_j)\circ(\delta_{ik}e+\delta_{jk}e-p_k),
\]
which then yields~\eqref{Eq:Pebetaek}. By summing over all basis elements, we deduce from~\eqref{Eq:eval}
\[
\Tr\left(A(x) P\left(x^{-1}\right)\right)=\sum_{i=1}^r \left(1+\frac{d}{2}(r-1)\right)4\left\langle \lambda_i^{-1}p_i,\alpha\right\rangle=\frac{n}{r}4\langle x^{-1}, \alpha\rangle,
\]
where the last equality follows from the fact that $n=r+\frac{d}{2}r(r-1)$.
\end{proof}

\section{Construction of Affine Processes on Symmetric Cones}\label{sec:existence}

Throughout this section we use the same setting as in the previous one, that is,
we suppose that $V$ is a simple Euclidean Jordan algebra of dimension $n$ and rank $r$ and $K$
is the associated irreducible symmetric cone. As before, we assume that the scalar product on $V$ is defined by
$\langle x, y\rangle=\tr(x\circ y)$ and the Peirce invariant $d$ corresponds to the dimension of $V_{ij},\, i< j$,
as defined in Appendix~\eqref{Peirce decomposition 2}.
Again we refer to~\citet{Faraut1994} and Appendix~\ref{appendix:EJA} for results on Euclidean Jordan algebras.

\subsection{Construction of Affine Diffusion Processes}\label{sec:existencediff}

The aim of this section is to establish existence of affine diffusion processes
for the following admissible parameter set $(\alpha, \delta\alpha,0,0,0,0,0)$ with $\delta \geq d(r-1)$.
To this end we consider the Riccati equations for $\phi$ and $\psi$ associated to these parameters and
show that, for every $(t,x) \in \re_+ \times K$, $e^{-\phi(t,u)-\langle \psi(t,u), x\rangle}$ is the Laplace transform
of a probability distribution supported on $K$. It will turn out that this probability distribution corresponds to
the \emph{non-central Wishart distribution} in the case of positive semidefinite matrices.
The existence of such affine diffusion processes then follows from the semi-flow property of $\phi$ and $\psi$, which
yields the Kolmogorov-Chapman equation for the transition probabilities and thus the Markov property.

\subsubsection{Explicit solutions for the Riccati equations}

We start by establishing explicit solutions for the Riccati equations associated
to the parameter set $(\alpha,\delta \alpha, 0,0,0,0,0)$ with $\delta \in \re_+$.

\begin{lemma}\label{lem:solRiccati}
Let $\alpha \in K$ and $\delta \in \re_+$. Consider the following system of Riccati differential equations for $u \in \mathring{K}$ and $t\in \re_+$
\begin{align}
\frac{\partial \psi(t,u)}{\partial t}&=-2P(\psi(t,u))\alpha,&\quad &\psi(0,u)=u \in \mathring{K}, \label{eq_simplestpsi}\\
\frac{\partial \phi(t,u)}{\partial t}&=\langle \delta \alpha, \psi(t,u)\rangle, &\quad& \phi(0,u)=0. \label{eq_simplestphi}
\end{align}
Then the solution is given by
\begin{align}
 \psi(t,u)&=(u^{-1}+2t\alpha)^{-1},  \label{eq_simplestpsisol}\\
 \phi(t,u)&=\frac{\delta}{2}\ln\det\left(e+2tP(\sqrt{\alpha})u\right). \label{eq_simplestphisol}
\end{align}
Moreover, $u \mapsto \psi(t,u)$ and $u \mapsto \phi(t,u)$ are continuous at $u=0$ for all $t \in \re_+$ with
$\phi(t,0)=0$ and $\psi(t,0)=0$.
\end{lemma}

\begin{proof}
Using
\[
 \frac{d}{dt}(x+tv)^{-1}=-P((x+tv)^{-1})v,
\]
which follows from Proposition~\ref{prop:EJA}~\ref{item:EJA2}, one easily verifies that $\psi(t,u)$ given by~\eqref{eq_simplestpsisol} satisfies~\eqref{eq_simplestpsi}. Concerning $\phi(t,u)$, let us first show that
\begin{align}\label{eq:detmult}
 \det(u)\det(u^{-1}+2t\alpha)=\det\left(e+2tP(\sqrt{\alpha})u\right).
\end{align}
Indeed, we have by Proposition~\ref{prop:EJA}~\ref{item:EJA1}, \ref{item:EJA4} and Proposition~\ref{prop:detmult}
\begin{align*}
 \det(u)\det(u^{-1}+2t\alpha)&=\det(u)\det\left(P(\sqrt{u^{-1}})\left(e+2tP(\sqrt{u})\alpha\right)\right)\\
&=\det(u)\det(u^{-1})\det\left(e+2tP(\sqrt{u})\alpha\right)\\
&=\det(uu^{-1})\det\left(e+2tP(\sqrt{u})\alpha\right)\\
&=\det\left(e+2tP(\sqrt{u})\alpha\right)\\
&=\det\left(e+2tP(\sqrt{\alpha})u\right).
\end{align*}
The last equality follows again from Proposition~\ref{prop:EJA}~\ref{item:EJA4}, which implies
\[
 \det(P(\sqrt{u})\alpha)=\det(u)\det(\alpha)=\det(P(\sqrt{\alpha})u).
\]
Hence $\phi(t,u)$ can be written as
\begin{align}\label{eq_diffphisol}
\phi(t,u)=\frac{\delta}{2}\ln\det(u)+ \frac{\delta}{2}\ln\det(u^{-1}+2t\alpha).
\end{align}
Using expression~\eqref{eq_diffphisol} for $\phi(t,u)$ and $\nabla\ln\det x=x^{-1}$ yields
\[
\frac{\partial \phi(t,u)}{\partial t} =\left\langle\frac{\delta}{2}(u^{-1}+2t\alpha)^{-1}, 2\alpha\right\rangle=\langle \delta \alpha, \psi(t,u)\rangle,
\]
and shows that~\eqref{eq_simplestphisol} solves~\eqref{eq_simplestphi}.
\end{proof}

Let now $\psi(t,u)$ and $\phi(t,u)$ be given by~\eqref{eq_simplestpsi} and~\eqref{eq_simplestphi} and consider
\[
\mathcal{L}^{\delta, \alpha,x}_t(u):=e^{-\phi(t,u)-\langle \psi(t,u),x \rangle}=\det\left(e+2tP(\sqrt{\alpha})u\right)^{-\frac{\delta}{2}}e^{-\langle (u^{-1}+2t \alpha)^{-1}, x\rangle}
\]
for $u \in \mathring{K}$. We shall now prove that for every $t \in \re_+$ and $\alpha, x \in K$
\begin{align}\label{eq:laplacetransform}
 u \mapsto \mathcal{L}^{\delta, \alpha, x}_t(u)
\end{align}
is the Laplace transform of a probability measure on $K$ if $\delta \geq d(r-1)$. In the case $K=S_r^+$, this is implied by~\citet[Proposition 3.2]{letac},
which asserts that $\mathcal{L}^{\delta, \alpha, x}_t$ corresponds to the Laplace transform of the non-central
Wishart distribution. The proof is based on the density function of this distribution, which exists for $\delta > (r-1)$
and $\alpha \in S_r^{++}$. As such a result is not available for general symmetric cones, we here establish the form of the density corresponding to $\mathcal{L}^{\delta, \alpha, x}_t$. This then yields a generalization of the non-central Wishart distribution on symmetric cones and an explicit form of the Markov kernels corresponding to the affine diffusion processes, associated with the parameter set $(\alpha, \delta \alpha, 0, 0, 0, 0, 0)$,
for $\delta>d(r-1)$ and $\alpha \in \mathring{K}$.

\subsubsection{Central Wishart Distribution}

We start by analyzing the case $x=0$, which corresponds to the \emph{central Wishart distribution} (see, e.g.,~\citet{letac} or~\citet{massam}).
Indeed, the following proposition states that for particular values of $\delta$, the expression
\[
e^{-\phi(t,u)}=\det(e+2t P(\sqrt{\alpha})u)^{-\frac{\delta}{2}}
\]
can be recognized as the Laplace transform of the central Wishart distribution with shape parameter $\frac{\delta}{2}$ and scale parameter $2t\alpha$.
In the case $\delta>d(r-1)$ and $\alpha \in \mathring{K}$, this distribution admits a density, whose form is explicitly given.

\begin{proposition}\label{centralWishart}
Let $\phi(t,u)$ be given by~\eqref{eq_simplestphi} and consider
\[
\mathcal{L}^{\delta, \alpha}_t(u):=e^{-\phi(t,u)}=\det\left(e+2tP(\sqrt{\alpha})u\right)^{-\frac{\delta}{2}}.
\]
If $\delta$ belongs to the set
\[
G=\left \{0,d,\ldots,d(r-1)\right\} \cup \left] d(r-1),\infty\right[,
\]
then, for every $t \in \re_+$ and $\alpha \in K$, $u \mapsto \mathcal{L}^{\delta, \alpha}_t(u)$ is the Laplace transform of a probability measure $W_t^{\frac{\delta}{2},\alpha}$ on $K$.
Moreover, if $\delta>d(r-1)$ and if $\alpha \in \mathring{K}$, then $W_t^{\frac{\delta}{2},\alpha}$ admits a density, which is given by
\begin{align}\label{flt central}
W_t^{\frac{\delta}{2},\alpha}(\xi)=\frac{1}{\Gamma_{K}\left(\frac{\delta}{2}\right)}\det
\left(\frac{\alpha^{-1}}{2t}\right)^{\frac{\delta}{2}}e^{-\left\langle \frac{\alpha^{-1}}{2t},
\xi\right\rangle}\det(\xi)^{\frac{\delta}{2}-\frac{n}{r}},
\end{align}
where $\Gamma_K$ denotes the Gamma function of $K$ (see~\citet[Section VII.1]{Faraut1994}).
\end{proposition}

\begin{proof}
By~\citet[Theorem VII.3.1 and Proposition VII.2.3]{Faraut1994}, $\det(u)^{-\frac{\delta}{2}}$
is the Laplace transform of a positive measure if and only if $\delta \in G$.
This is equivalent to the fact that $\det(u)^{-\frac{\delta}{2}}$ is a function of positive type (see, e.g.,~\citet[page 136]{Faraut1994}),
that is,
\[
 \sum_{i,j=1}^{N}\det(u_i+u_j)^{-\frac{\delta}{2}}c_i\bar{c}_j\geq 0
\]
for all choices of $u_1,\ldots,u_N \in K$ and complex numbers $c_1,\ldots,c_N$. For every $t \in \re_+$,
$\mathcal{L}^{\delta,\alpha}_t(u)=\det\left(e+2tP(\sqrt{\alpha})u\right)^{-\frac{\delta}{2}}$
is therefore also a function of positive type and hence the Laplace transform of a positive measure if $\delta \in G$. Since $\mathcal{L}^{\delta,\alpha}_t(u+v)\leq \mathcal{L}^{\delta,\alpha}_t(u)$ for all $u,v \in K$, the measure is supported on $K$. As $\mathcal{L}^{\delta,\alpha}_t(0)=1$, the measure is actually a probability measure.

Concerning the second assertion, we have by~\citet[Corollary VII.1.3]{Faraut1994} and Proposition~\ref{prop:EJA}~\ref{item:EJA4}
\begin{align*}
\int_K e^{-\left\langle u +\frac{\alpha^{-1}}{2t}, \xi\right\rangle} \det (\xi)^{\frac{\delta}{2}-\frac{n}{r}} d\xi&=\Gamma_K\left(\frac{\delta}{2}\right) \det\left(u+\frac{\alpha^{-1}}{2t}\right)^{-\frac{\delta}{2}}\\
&=\Gamma_K\left(\frac{\delta}{2}\right) \det\left(\frac{P(\sqrt{\alpha^{-1}})}{2t}\left(2t P(\sqrt{\alpha})u+e\right)\right)^{-\frac{\delta}{2}}\\
&=\Gamma_K\left(\frac{\delta}{2}\right) \det\left({\frac{\alpha^{-1}}{2t}}\right)^{-\frac{\delta}{2}} \det\left(e+2t P(\sqrt{\alpha})u\right)^{-\frac{\delta}{2}}.
\end{align*}
The definition of the density of $W_t^{\frac{\delta}{2},\alpha}$ then yields the assertion.
\end{proof}

\begin{remark}\label{rem:Gindinkinset}
\begin{enumerate}
\item
Analogous to the cone of positive semidefinite matrices, one can define
the central Wishart distribution $W^{p, \sigma}$ with shape parameter
\[
p \in \widetilde{G}=\left \{0,\frac{d}{2},\ldots,\frac{d(r-1)}{2}\right\} \cup \left] \frac{d(r-1)}{2},\infty\right[
\]
and scale parameter $\sigma \in K$ on a symmetric cone by its Laplace transform which takes the form stated
in Proposition~\ref{centralWishart}, that is,
\[
\int_K e^{-\langle u, \xi\rangle} W^{p,\sigma}(d\xi)=\det\left(e+P(\sqrt{\sigma})u\right)^{-p}.
\]
(see, e.g.,~\citet[Corollary 3]{massam}).
\item If $p=k\frac{d}{2}, k \in \{1, \ldots, r-1\}$, then the Wishart distribution is supported on the set of elements in $K$ which are precisely of rank $k$.
This is a consequence of~\citet[Proposition VII.2.3]{Faraut1994}.
\end{enumerate}
\end{remark}

\subsubsection{Non-central Wishart distribution}

In order to formulate Proposition~\ref{lem:noncentral} below, where we establish the form of the density function of the non-central Wishart distribution on a symmetric cone, let us introduce the so-called zonal polynomials (see~\citet[Section XI.3, p. 234]{Faraut1994}).

\begin{definition}[Zonal Polynomials]\label{def:zonal}
For each multi-index $\mathbf{m}=(m_1,\dots,m_r)\in\mathbb N^r$ with length $\vert\mathbf{m}\vert:=m_1+\dots+m_r$, we consider the generalized \emph{power function} defined by
\[
\Delta_{\mathbf{m}}(\xi)=\Delta_1(\xi)^{m_1-m_2}\Delta_2(\xi)^{m_2-m_3}\ldots\Delta_r(\xi)^{m_r},
\]
where $\Delta_i$ denotes the principal minors corresponding to the Jordan subalgebras $V^{(j)}=V(p_1+\ldots+p_j,1)$ with $p_1, \ldots, p_r$ some fixed Jordan frame.
The $\mathbf{m}^{th}$ zonal polynomial $Z_{\mathbf{m}}$ is now defined by
\begin{align*}
Z_{\mathbf{m}}(\xi)=\omega_{\mathbf{m}}\int_{O\in \mathcal{O}}\Delta_{\mathbf{m}}(O\xi)dO,
\end{align*}
where $dO$ is the normalized Haar measure on $\mathcal{O}$ and $\mathcal{O}= G \cap O(V)$, where $G$ is the connected component of the identity in the automorphism group of $K$ and $O(V)$ the orthogonal group of $V$. Moreover, $\omega_{\mathbf{m}}$ denotes some positive normalizing constant such that
\[
(\tr(\xi))^k=\sum_{|\mathbf{m}|=k} Z_{\mathbf{m}}(\xi).
\]
\end{definition}

We are now prepared to show that $\mathcal{L}^{\delta, \alpha,x}_t$ is the Laplace transform of a probability distribution if $\delta \geq d(r-1)$. As before this distribution is absolutely continuous with respect to the Lebesgue measure if  $\alpha \in \mathring{K}$ and $\delta > d(r-1)$.
In order to prove that, we generalize a result by~\citet{letac} on the density function of the non-central Wishart distribution on positive semidefinite matrices to symmetric cones.

\begin{proposition}\label{lem:noncentral}
Let $\psi(t,u)$ and $\phi(t,u)$ be given by~\eqref{eq_simplestpsi} and~\eqref{eq_simplestphi} and consider
\[
\mathcal{L}^{\delta, \alpha,x}_t(u):=e^{-\phi(t,u)+\langle \psi(t,u),x \rangle}=\det\left(e+2tP(\sqrt{\alpha})u\right)^{-\frac{\delta}{2}}e^{-\langle (u^{-1}+2t \alpha)^{-1}, x\rangle}
\]
for $u \in \mathring{K}$. Then we have:
\begin{enumerate}
\item \label{item:noncentral1}If $\delta \geq d(r-1)$, then , for every $t \in \Rplus$ and $\alpha, x \in K$, $u \mapsto \mathcal{L}^{\delta, \alpha, x}_t(u)$ is the Laplace transform of a probability measure $W_t^{\frac{\delta}{2},\alpha,x}$ on $K$.
\item \label{item:noncentral2} If $\delta>d(r-1)$ and if $\alpha \in \mathring{K}$,
then $W_t^{\frac{\delta}{2},\alpha,x}$ admits a density, which is given by
\begin{equation}\label{flt noncentral}
\begin{split}
W_t^{\frac{\delta}{2},\alpha,x}(\xi)&=
\det\left(\frac{\alpha^{-1}}{2t}\right)^{\frac{\delta}{2}}e^{-\left\langle \frac{\alpha^{-1}}{2t},\xi+x\right\rangle}\det(\xi)^{\frac{\delta}{2}-\frac{n}{r}}\\
&\quad\times\left(\sum_{\mathbf{m}\geq 0}\frac{Z_{\mathbf{m}}\left(\frac{1}{4t^2}P\left(\sqrt x\right)P\left(\alpha^{-1}\right)\xi\right)}{\vert \mathbf{m}\vert!\Gamma_K\left(\mathbf{m}+\frac{\delta}{2}\right)}\right)1_K(\xi),
\end{split}
\end{equation}
where $\Gamma_K$ denotes the Gamma function of $K$ (see~\citet[Section VII.1]{Faraut1994}) and $Z_{\mathbf{m}}$ the zonal polynomials introduced in Definition~\ref{def:zonal}.
\item \label{item:noncentral3}Let $\delta \geq d(r-1)$. If $\alpha \in \partial K$, then $W_t^{\frac{\delta}{2},\alpha,x}$ is not absolutely continuous with respect to the Lebesgue measure.
\end{enumerate}
\end{proposition}

\begin{proof}
Let us first prove~\ref{item:noncentral2} by applying similar arguments as in~\citet{letac}.
We start by showing that $W_t^{\frac{\delta}{2},\alpha,x}$, as given in~\eqref{flt noncentral}, is a well-defined positive measure. Concerning the convergence of
\begin{align}\label{eq:sum}
\sum_{\mathbf{m}\geq 0}\frac{Z_{\mathbf{m}}\left(\frac{1}{4t^2}P\left(\sqrt x\right)P\left(\alpha^{-1}\right)\xi\right)}{\vert \mathbf{m}\vert!\Gamma_K\left(\mathbf{m}+\frac{\delta}{2}\right)},
\end{align}
we can estimate $\Gamma_K(\mathbf{m}+\frac{\delta}{2})$ due to~\citet[Theorem VII.1.1]{Faraut1994} by
\[
\Gamma_K\left(\mathbf{m}+\frac{\delta}{2}\right)\geq (2\pi)^{\frac{n-r}{2}}(\min_{z \geq 0}\Gamma(z))^r=:M,
\]
where $\Gamma$ denotes the Gamma function on $\mathbb{R}$.
This implies convergence of~\eqref{eq:sum}, since we have by~\citet[Proposition XII.1.3(i)]{Faraut1994}
\begin{align}\label{eq:traceformula}
\det(e^{\xi})=e^{\tr(\xi)}=\sum_{\mathbf{m}\geq 0}\frac{Z_{\mathbf{m}}(\xi)}{\vert \mathbf{m}\vert!}
\end{align}
for every $\xi \in K$. Due to the definition of the zonal polynomials, in particular since $\Delta_{\mathbf{m}}(\xi)> 0$ for all $\xi \in K\setminus\{0\}$, $W_t^{\frac{\delta}{2}, \alpha,x}$ is therefore a well-defined positive measure.
Let us now prove that $u \mapsto L_t^{\delta,\alpha,x}(u)$ is the Laplace transform of $W_t^{\frac{\delta}{2},\alpha,x}$.
For each $\mathbf{m} \in \mathbb{N}^r$ and each automorphism $g$, we have by~\citet[Lemma XI.2.3]{Faraut1994} and
Proposition~\ref{prop:EJA}~\ref{item:EJA4}
\begin{equation}\label{eq:formula1}
\begin{split}
&\int_K e^{-\left\langle u+\frac{\alpha^{-1}}{2t},\xi\right\rangle}\det(\xi)^{\frac{\delta}{2}-\frac{n}{r}}Z_{\mathbf{m}}(g\xi) d\xi\\
&\quad =\Gamma_K\left(\mathbf{m}+\frac{\delta}{2}\right)\det\left(u+\frac{\alpha^{-1}}{2t}\right)^{-\frac{\delta}{2}}Z_{\mathbf{m}}\left(g\left(u+\frac{\alpha^{-1}}{2t}\right)^{-1}\right)\\
&\quad=\Gamma_K\left(\mathbf{m}+\frac{\delta}{2}\right)\det\left(\frac{\alpha^{-1}}{2t}\right)^{-\frac{\delta}{2}}\det\left(e+2tP\left(\sqrt{\alpha}\right)u\right)^{-\frac{\delta}{2}}\\
&\quad \quad \times Z_{\mathbf{m}}\left(g\left(u+\frac{\alpha^{-1}}{2t}\right)^{-1}\right).
\end{split}
\end{equation}
If $x$ is non-degenerate, $P(\sqrt{x})P(\alpha^{-1})$ is an automorphism and plays the role of $g$ in our case.
However, the above formula also holds true if $x$ is degenerate. Indeed, let us approximate $x$ by $x_n=x+\frac{1}{n}e$. Since
$
(\tr(\xi))^k=\sum_{\vert\mathbf{m}\vert=k} Z_{\mathbf{m}}(\xi)
$
for every $k$ (see p. 235 in~\citet{Faraut1994}), we have for $\vert\mathbf{m}\vert=k$
\[
Z_{\mathbf{m}}\left(P\left(\sqrt{x_n}\right)P\left(\alpha^{-1}\right)\xi\right)\leq \left(\tr\left(P\left(\alpha^{-1}\right)\xi x_n\right)\right)^k\leq\left(\tr\left(P\left(\alpha^{-1}\right)\xi x_1\right)\right)^k.
\]
Dominated convergence then yields~\eqref{eq:formula1} also for degenerate $x$.
By~\eqref{eq:traceformula} we obtain
\begin{equation}\label{eq:formula2}
\begin{split}
&\det\left(e+2tP\left(\sqrt{\alpha}\right)u\right)^{\frac{\delta}{2}}\int_K e^{-\langle u,\xi \rangle}W_t^{\frac{\delta}{2},\alpha,x}d\xi\\
&\quad =e^{-\left\langle \frac{\alpha^{-1}}{2t}, x\right\rangle}\left(\sum_{\mathbf{m}\geq 0}\frac{Z_{\mathbf{m}}\left(\frac{1}{4t^2}P\left(\sqrt x\right)P\left(\alpha^{-1}\right)\left(u+\frac{\alpha^{-1}}{2t}\right)^{-1}\right)}{\vert \mathbf{m}\vert!}\right)\\
&\quad =e^{-\left\langle \frac{\alpha^{-1}}{2t}, x\right\rangle+\left\langle \frac{1}{4t^2}P\left(\sqrt x\right)P\left(\alpha^{-1}\right)\left(u+\frac{\alpha^{-1}}{2t}\right)^{-1},e\right\rangle}.
\end{split}
\end{equation}
Using $P(z)=P(\sqrt{z})P(\sqrt{z})$, Proposition~\ref{prop:EJA}~\ref{item:EJA5} and~\citet[Exercise II.5 (c)]{Faraut1994},
which asserts $(z+e)^{-1}-e=-(z^{-1}+e)^{-1}$ for invertible elements $z$, $z+e$ and $z^{-1}+e$, we get
\begin{align*}
&-\left\langle \frac{\alpha^{-1}}{2t}, x\right\rangle+\left\langle \frac{1}{4t^2}P\left(\sqrt x\right)P\left(\alpha^{-1}\right)\left(u+\frac{\alpha^{-1}}{2t}\right)^{-1},e\right\rangle\\
&\quad=\left \langle -e+\frac{1}{2t} P\left(\sqrt{\alpha^{-1}}\right)\left(u+\frac{\alpha^{-1}}{2t}\right)^{-1},  \frac{1}{2t} P\left(\sqrt{\alpha^{-1}}\right)x\right\rangle\\
&\quad=\left \langle -e+\left(2tP\left(\sqrt{\alpha}\right)u+e\right)^{-1},  \frac{1}{2t} P\left(\sqrt{\alpha^{-1}}\right)x\right\rangle\\
&\quad=-\left\langle \left(\left(2tP\left(\sqrt{\alpha}\right)u\right)^{-1}+e\right)^{-1}, \frac{1}{2t} P\left(\sqrt{\alpha^{-1}}\right)x\right\rangle\\
&\quad=-\left\langle 2tP\left(\sqrt{\alpha}\right)\left(u^{-1}+2tP\left(\sqrt{\alpha}\right)e\right)^{-1}, \frac{1}{2t} P\left(\sqrt{\alpha^{-1}}\right)x\right\rangle\\
&\quad= -\left\langle(u^{-1}+2t\alpha)^{-1},x\right \rangle\\
&\quad=-\left\langle \psi(t,u),x\right \rangle.
\end{align*}
This proves that $u\mapsto \mathcal{L}_t^{\delta,\alpha,x}(u)$ is the Laplace transform of the density
given in~\eqref{flt noncentral}.
Moreover, $W_t^{\frac{\delta}{2},\alpha,x}$ qualifies as probability distribution, since $\mathcal{L}_t^{\delta, \alpha,x}(0)=1$, which can be seen by plugging $u=0$ in equation~\eqref{eq:formula2}.

Concerning the first statement~\ref{item:noncentral1}, we have locally uniform convergence on $\mathring{K}$ of $\mathcal{L}_t^{d(r-1)+\frac{1}{n}, \alpha+\frac{e}{n}, x}$ to $\mathcal{L}_t^{d(r-1), \alpha, x}$,
which is a continuous function at $0$. Hence invoking L\'evy's continuity theorem yields the assertion.

Finally, let us consider assertion~\ref{item:noncentral3}. Since affine transformations do not affect the property of having a density, we can assume
$\alpha$ to be of the form
\[
\alpha=e_m:=\sum_{i=1}^m p_i,\quad m< r,
\]
for some orthogonal idempotents $p_1, \ldots, p_m$. Let now $X$ be a $K$-valued random variable with distribution $W_t^{\frac{\delta}{2},e_m,x}$ and denote
by $\Pi_{r-m}$ the projection on $V^{(r-m)}=V(e_m,0)=\{x \in V \,| \, x\circ e_m=0\}$. Then some algebraic manipulations yield that
\[
\mathbb{E}\left[e^{-\langle u, \Pi_{r-m} (X)\rangle}\right]=e^{-\langle \Pi_{r-m}(u), \Pi_{r-m}(x)\rangle},
\]
which is the Laplace transform of the unit mass at $\Pi_{r-m}(x)$. This implies that $X$ does not admit a density, because the pushforward of a measure with a density under some linear map would again admit a density.
\end{proof}

\begin{remark}
The explicit form of the Markov kernels corresponding to the
affine diffusion processes associated with the parameter set $(\alpha, \delta \alpha, 0, 0, 0, 0, 0)$,
for $\delta>d(r-1)$ and $\alpha \in \mathring{K}$, is thus given by $W_t^{\frac{\delta}{2},\alpha,x}$.
\end{remark}

\subsubsection{Existence of Bru Processes}

Using the knowledge that $\mathcal{L}^{\delta, \alpha, x}_t$ is the Laplace transform of a probability distribution
on $K$, we can finally prove existence of affine diffusion processes associated with the particular parameter set $(\alpha, \delta \alpha, 0, 0, 0, 0, 0)$,
where $\alpha \in K$ and $\delta \geq d(r-1)$. We call these affine processes \emph{Bru processes}, since on $S_r^+$ they correspond to the following diffusion process, 
which was first studied by~\citet{bru}
\[
dX_t=\delta \alpha+\sqrt{X_t}dW_t\sqrt{\alpha}+\sqrt{\alpha}dW^{\top}\sqrt{X_t}, \quad X_0=x,
\]
where $W$ is a $r \times r $ matrix of Brownian motions.

\begin{proposition}\label{th: existence Markov}
Let $(\alpha, \delta \alpha, 0, 0, 0, 0, 0)$ be an admissible
parameter set, that is, $\alpha \in K$ and $\delta \geq d(r-1)$. Then there exists a unique affine process on $K$
such that~\eqref{eq:affineprocessK} holds for all $(t,u) \in \re_+ \times K$, where
$\phi(t,u)$ and $\psi(t,u)$ are given in Lemma~\ref{lem:solRiccati}.
\end{proposition}

\begin{proof}
By Proposition~\ref{lem:noncentral} we have for every $(t,x) \in \re_+ \times K$ the existence of a probability measure on $K$
with Laplace-transform $e^{-\phi(t,u)-\langle\psi(t,u),x\rangle}$, where $\phi(t,u)$ and $\psi(t,u)$ are specified in Lemma~\ref{lem:solRiccati}.
The Chapman-Kolmogorov equations hold in view of the flow property of $\phi$ and $\psi$,
whence the assertion follows.
\end{proof}

\begin{remark}
In the case of positive semidefinite matrices,~\citet{bru} has shown existence and
uniqueness for the process
\[
dX_t=\delta I_r+\sqrt{X_t}dW_t+dW^{\top}\sqrt{X_t}, \quad X_0=x,
\]
if $\delta > r-1$ and $x$ with distinct
eigenvalues (see~\citet[Theorem 2 and Section 3]{bru}). This process corresponds to the parameter set $(I_r,\delta I_r, 0,0,0,0,0)$ on the cone $S_r^+$.
Note that, since the Peirce invariant $d$ equals $1$ in this case, $\delta>r-1$ is a stronger assumption than what we require on $\delta$.
Actually,~\citet{bru} establishes existence
and uniqueness of solutions also for $\delta=1,\dots,r-1$. But these are
degenerate solutions, as they are only defined on lower dimensional
subsets of the boundary of $S_r^+$ (see~\citet[Corollary 1]{bru} and compare Remark~\ref{rem:Gindinkinset} (ii)).
\end{remark}

\subsubsection{Existence and Transition Densities of Wishart Processes}

By allowing a non-zero linear drift, we can enlarge the class of Bru processes to so-called \emph{Wishart processes}. These are affine processes corresponding to the parameter set\\ $(\alpha, \delta \alpha,B,0,0,0,0)$, where $\delta \geq d (r-1)$ and $B \in \mathfrak{g}(K)$. Here, $\mathfrak{g}(K)$ denotes the Lie algebra of the automorphism group $G(K)$. These linear maps satisfy
\[
2P(B(x),x)=BP(x)+P(x)B^{\top}.
\]
For such parameter sets the corresponding Riccati equations can still be solved explicitly, which is stated in the following lemma:

\begin{lemma}\label{lem:solRiccatiWish}
Let $\alpha \in K$, $\delta \in \re_+$ and $B \in \mathfrak{g}(K)$.
Consider the following system of Riccati differential equations for $u \in \mathring{K}$ and $t\in \re_+$
\begin{align}
\frac{\partial \psi(t,u)}{\partial t}&=-2P(\psi(t,u))\alpha+B^{\top}(\psi(t,u)),&\quad &\psi(0,u)=u \in \mathring{K}, \label{eq_simplestpsiB}\\
\frac{\partial \phi(t,u)}{\partial t}&=\langle \delta \alpha, \psi(t,u)\rangle, &\quad& \phi(0,u)=0. \label{eq_simplestphiB}
\end{align}
Then the solution is given by
\begin{align}
 \psi(t,u)&=e^{B^{\top}t}(u^{-1}+\sigma_t^B(\alpha))^{-1},  \label{eq_simplestpsisolB}\\
 \phi(t,u)&=\frac{\delta}{2}\ln\det\left(e+P\left(\sqrt{\sigma_t^B(\alpha)}\right)u\right), \label{eq_simplestphisolB}
\end{align}
where $\sigma_t^B(y)=2\int_0^t e^{B s}y ds$.
\end{lemma}

\begin{proof}
Differentiation of~\eqref{eq_simplestpsisolB} yields
\begin{align*}
\frac{\partial \psi(t,u)}{\partial t}&=-2e^{B^{\top}t}P((u^{-1}+\sigma_t^B(\alpha))^{-1})e^{B t}\alpha+B^{\top}(e^{B^{\top}t}(u^{-1}+\sigma_t^B(\alpha))^{-1})\\
&=-2P(\psi(t,u))+B^{\top}(\psi(t,u)),
\end{align*}
where we use Proposition~\ref{prop:EJA} \ref{item:EJA2} and the fact that for all $t \in \mathbb{R}$
\[
e^{B^{\top}t}P(u)e^{Bt}=P(e^{B^{\top}t}u),
\]
which follows from~\citet[Proposition III.5.2]{Faraut1994}, since $e^{B^{\top}t} \in G(K)$ for all $t$.
Using the same arguments as in Lemma~\ref{lem:solRiccati}, we can write~\eqref{eq_simplestphisolB}
as
\[
\phi(t,u)=\frac{\delta}{2}\ln \det(u)+\frac{\delta}{2}\ln \det(e+\sigma_t^B(\alpha)),
\]
whence
\[
\frac{\partial \phi(t,u)}{\partial t}=\left\langle \frac{\delta}{2}(u^{-1}+\sigma_t^B(\alpha))^{-1}, 2 e^{Bt} \alpha \right\rangle=\langle \delta \psi(t,u), \alpha \rangle.
\]
\end{proof}

Comparing now the solutions of Lemma~\ref{lem:solRiccatiWish} with those of Lemma~\ref{lem:solRiccati} and combining this with Proposition~\ref{lem:noncentral}, we obtain the following corollary.

\begin{corollary}
Let $(\alpha, \delta \alpha, B, 0, 0, 0, 0)$ be an admissible
parameter set with $B \in \mathfrak{g}(K)$. Then there exists a unique affine process on $K$
such that~\eqref{eq:affineprocessK} holds for all $(t,u) \in \re_+ \times K$, where
$\phi(t,u)$ and $\psi(t,u)$ are given in Lemma~\ref{lem:solRiccatiWish}.
\end{corollary}

\begin{proof}
Consider $\phi(t,u)$ and $\psi(t,u)$ as given in Lemma~\ref{lem:solRiccatiWish}. Then
\begin{align}\label{eq:WishLaplace}
e^{-\phi(t,u)-\langle\psi(t,u),x\rangle}=\mathcal{L}_{\frac{1}{2}}^{\frac{\delta}{2}, \sigma_t^B(\alpha), e^{Bt}x},
\end{align}
where $\mathcal{L}_t^{\frac{\delta}{2},\alpha, x}$ is specified in Proposition~\ref{lem:noncentral}. Statement~\ref{item:noncentral1} of this proposition thus implies
for every $(t,x) \in \re_+ \times K$ the existence of a probability measure $W_{\frac{1}{2}}^{\frac{\delta}{2}, \sigma_t^B(\alpha), e^{Bt}x}$ on $K$, whose Laplace transform is given by the above expression.
The Chapman-Kolmogorov equations hold again due to the flow property of $\phi$ and $\psi$,
whence the assertion follows.
\end{proof}

Let us finally characterize the existence of a density for the affine process corresponding to the parameter set
$(\alpha, \delta \alpha, B, 0, 0, 0, 0)$ with $\alpha \in K$, $\delta > d(r-1)$ and $B \in \mathfrak{g}(K)$.

\begin{proposition}
Let $X$ be an affine process corresponding to the parameter set
$(\alpha, \delta \alpha, B, 0, 0, 0, 0)$ with $\alpha \in K$, $\delta > d(r-1)$ and $B \in \mathfrak{g}(K)$. Then, $X_t$ has a density
$W_{\frac{1}{2}}^{\frac{\delta}{2}, \sigma_t^B(\alpha), e^{Bt}x}$ for one (and hence all) $t > 0$ if and only if
\begin{align}\label{eq:condfullrank}
\rk(L(\alpha),L(B(\alpha)),\ldots, L(B^{n-1}(\alpha)))=n,
\end{align}
where $L$ is the left product operator defined in~\eqref{eq:leftprod}.
\end{proposition}

\begin{proof}
Under the assumption $\delta > d(r-1)$, Proposition~\ref{lem:noncentral} implies that the probability measure $W_{\frac{1}{2}}^{\frac{\delta}{2}, \alpha,x}$ admits a density if and only if $\alpha \in \mathring{K}$. Moreover, according to Lemma~\ref{lem:condfullrank} below, $\sigma_t^B(\alpha)\in \mathring{K}$ is equivalent to Condition~\eqref{eq:condfullrank}. Due to relation~\eqref{eq:WishLaplace}, Proposition~\ref{lem:noncentral} thus yields the existence of a density given by $W_{\frac{1}{2}}^{\frac{\delta}{2}, \sigma_t^B(\alpha), e^{Bt}x}$ if and only if~\eqref{eq:condfullrank} is satisfied.
\end{proof}

The following lemma states the above used equivalence between Condition~\eqref{eq:condfullrank} and the fact that $\sigma_t^B(\alpha) \in \mathring{K}$.

\begin{lemma}\label{lem:condfullrank}
Let $\alpha \in K$ and let $B$ be a linear map satisfying $\langle B(x),u \rangle\geq 0$ for all $u,x \in K$ with $\langle u,x \rangle =0$.
Then
\[
\sigma_t^B(\alpha)=2\int_0^te^{Bs}\alpha \in \mathring{K},
\]
 for one (and hence for all) $t > 0$, if and only if~\eqref{eq:condfullrank} holds true.
\end{lemma}

\begin{proof}
Note that the condition on $B$ implies $e^{Bt}\alpha \in K$ for all $t\geq 0$. Let us therefore prove the equivalent statement, that is,
$\sigma_t^B(\alpha) \in \partial K $ if and only if $\rk(L(\alpha),L(B(\alpha)),\ldots, L(B^{n-1}(\alpha)))<n$.
By~\citet[Theorem III.2.1]{Faraut1994}, $\sigma_t^B(\alpha) \in \partial K $  is equivalent to
\[
\int_0^t L(e^{Bs}\alpha)ds \in \partial S_+(V),
\]
which in turn holds true if and only if
\[
L(e^{Bt}\alpha)u =0,
\]
for some $u \neq 0 \in V$ and all $t > 0$. By the definition of the exponential function this is however equivalent to
\begin{align}\label{eq:eqiv}
L(B^k(\alpha))u=0, \quad k=0,1,2, \ldots.
\end{align}
Since $B^k, \, k \geq n$ can be expressed as a linear combination of $I, B, \ldots, B^{n-1}$, which is a consequence of the Cayley-Hamilton Theorem,~\eqref{eq:eqiv}
is equivalent to
\[
(L(\alpha),L(B(\alpha)),\ldots, L(B^{n-1}(\alpha)))^{\top}u=0.
\]
This proves the assertion.
\end{proof}

\subsection{Existence of Affine Processes on Symmetric Cones}\label{sec:existencesymcone}

In Proposition~\ref{th: existence Markov} and Proposition~\ref{th: existence Markov2} we have proved existence of affine diffusion processes with a particular constant drift parameter and
existence of pure affine jump processes.
In order to establish existence of affine processes on symmetric cones
for any admissible parameter set, we now combine the respective Riccati equations to show that
\[
e^{-\phi(t,u)-\langle \psi(t,u),x \rangle}=\lim_{N \rightarrow \infty}\left[P^2_{\frac{t}{N}}P^1_{\frac{t}{N}}\right]^Ne^{-\langle u,x \rangle},
\]
is the Laplace transform of a probability distribution on $K$ for any admissible parameter set. Here, $P^i, i=1,2$, denote the respective semigroups of the diffusion process and the pure jump process.

Given an admissible parameter set $(\alpha, b, B, c, \gamma, m, \mu)$, let us therefore consider the following two systems of Riccati ODEs:
\begin{equation}\label{eq:phipsi_i}
\begin{split}
\partial_t \psi_1(t,u)&=R_1(\psi_1(t,u))=-2P(\psi_1(t,u))\alpha, \\
\partial_t \phi_1(t,u)&=F_1(\psi_1(t,u))=\langle \delta \alpha, \psi_1(t,u)\rangle, \\
\partial_t \psi_2(t,u)&=R_2(\psi_2(t,u))=B^{\top}(\psi_2(t,u))+\gamma \\
&\qquad \qquad \qquad \quad-\int_{K} \left(e^{-\langle\xi, \psi_2(t,u)\rangle}-1+ \langle \chi(\xi),\psi_2(t,u) \rangle \right) \mu(d\xi), \\
\partial_t \phi_2(t,u)&=F_2(\psi_2(t,u))=\langle b-\delta \alpha,\psi_2(t,u) \rangle + c -\int_{K} \left(e^{-\langle\xi, \psi_2(t,u)\rangle}-1\right) m(d\xi),
\end{split}
\end{equation}
where we set $\delta=d(r-1)$. The original Riccati equations corresponding to the parameter set $(\alpha, b, B, c, \gamma, m, \mu)$ are then given by
\begin{align}
\partial_t \psi(t,u)&=R(\psi(t,u))=R_1(\psi(t,u))+R_2(\psi(t,u)), &\quad& \psi(0,u)=u, \label{eq:Riccatisplit1}\\
\partial_t \phi(t,u)&=F(\psi(t,u))=F_1(\psi(t,u))+F_2(\psi(t,u)), &\quad& \phi(0,u)=0.\label{eq:Riccatisplit2}
\end{align}

Let us remark that, due to Proposition~\ref{prop_ricc_sol}, there exists a global unique solution to~\eqref{eq:Riccatisplit1}-\eqref{eq:Riccatisplit2} for every $u \in \mathring{K}$, which remains in $\mathring{K}$ for all $t \in \re_+$.

\begin{lemma}\label{lem: splitting}
Let $\phi_i, \psi_i$, $i=1,2$, be defined by~\eqref{eq:phipsi_i} and let
$u \in \mathring{K}$ and $t \geq 0$ be fixed. Define recursively for each $N \in \mathbb{N}$ and $n \in \{0, \ldots,N\}$
\begin{align*}
y_0(u)&:=u, &\quad w_0(u)&:=0,\\
y_n(u)&:=\psi_2(\tau, \psi_1(\tau, y_{n-1})), &\quad w_n(u)&:=\phi_1(\tau, y_{n-1})+\phi_2(\tau, \psi_1(\tau, y_{n-1})+w_{n-1},
\end{align*}
where $\tau=\frac{t}{N}$.
Then
\[
\psi(t,u)=\lim_{N \rightarrow \infty} y_N(u) \quad \textrm{and} \quad \phi(t,u)=\lim_{N \rightarrow \infty}w_N(u),
\]
where $\phi$ and $\psi$ are given by~\eqref{eq:Riccatisplit1}-\eqref{eq:Riccatisplit2}.
\end{lemma}

\begin{proof}
Let us first remark that the limits are well-defined, since we have existence of global solutions of~\eqref{eq:Riccatisplit1}-\eqref{eq:Riccatisplit2} by Proposition~\ref{prop_ricc_sol}. In order to prove convergence of this splitting scheme, let us first calculate the local errors of the approximations for $\phi$ and $\psi$ for a given step size $\tau=\frac{t}{N}$ with $N \in \mathbb{N}$ fixed. An estimate of the global error is then obtained by transporting the local errors to the final point $t$ and adding them up, as it is done in~\citet[Theorem 3.6]{hairer}. Following~\citet[Chapter II.3]{hairer}, let us define the increment functions $\Phi_{\psi}$ and $\Phi_{\phi}$ by
\begin{align*}
y_n(u)&=y_{n-1}(u)+\tau \Phi_{\psi}(y_{n-1}(u),\tau),\\
w_n(u)&=w_{n-1}(u)+\tau \Phi_{\psi}(y_{n-1}(u),\tau).
\end{align*}
Using Taylor expansions at $\tau=0$, we obtain due to the
real-analyticity of $R_1, R_2$ and $F_1, F_2$ on $\mathring{K}$ for
$y \in \mathring{K}$
\begin{align*}
\Phi_{\psi}(y,\tau)&=R_2(y)+R_1(y)\\
&\quad+\frac{1}{2}\tau (DR_2(y)R_2(y) + 2DR_2(y)R_1(y)+DR_1(y)R_1(y))\\
&\quad+\mathcal{O}(\tau^2),\\
\Phi_{\phi}(y,\tau)&=F_1(y)+F_2(y)\\
&\quad+\frac{1}{2}\tau (\langle DF_2(y),R_2(y)\rangle+ 2\langle DF_2(y),R_1(y)\rangle+\langle DF_1(y),R_1(y)\rangle)\\
&\quad+\mathcal{O}(\tau^2).
\end{align*}
Hence the local errors satisfy by another Taylor expansion of $\psi$ and $\phi$
\begin{align*}
&\|\psi(t+\tau,u)-\psi(t,u)-\tau\Phi_{\psi}(\psi(t,u), \tau)\|\\
&\quad = \frac{1}{2}\tau^2\|DR_1(\psi(t,u))R_2(\psi(t,u))-DR_2(\psi(t,u))R_1(\psi(t,u))\|+\mathcal{O}(\tau^3)\\
&\quad\leq C_{\psi}\tau^2,\\
&|\phi(t+\tau,u)-\phi(t,u)-\tau\Phi_{\phi}(\psi(t,u), \tau)|\\
&\quad=\frac{1}{2}\tau^2|\langle DF_1(\psi(t,u)),R_2(\psi(t,u))\rangle-\langle DF_2(\psi(t,u)),R_1(\psi(t,u))\rangle|+\mathcal{O}(\tau^3)\\
&\quad \leq C_{\phi}\tau^2.
\end{align*}
Since $R_1, R_2$ and $F_1, F_2$ are real-analytic on $\mathring{K}$,
the following Lipschitz conditions for some constants $\Lambda_{\psi}, \Lambda_{\phi}$ are satisfied in a neighborhood of the solution
\[
\| \Phi_{\psi}(z,\tau)-\Phi_{\psi}(y,\tau)\|\leq \Lambda_{\psi}\| z-y\|, \quad | \Phi_{\phi}(z,\tau)-\Phi_{\phi}(y,\tau)|\leq \Lambda_{\phi}\|z-y\|.
\]
Moreover, by Proposition~\ref{prop_ricc_sol}, $\psi(t,u) \in \mathring{K}$ for all $(t,u)\in \re_+ \times \mathring{K}$
such that we have by~\citet[Theorem II.3.6]{hairer}
\begin{align*}
\|\psi(t,u)-y_N(u)\|&\leq \tau\frac{C_{\psi}}{\Lambda_{\psi}} e^{\Lambda_{\psi}t-1}, \\
\|\phi(t,u)-w_N(u)\|&\leq \tau\frac{C_{\phi}}{\Lambda_{\phi}}
e^{\Lambda_{\phi}t-1}.
\end{align*}
Since $\tau = \frac{t}{N}$ both terms converge to
$0$ as $N \to \infty$.
\end{proof}

We are now prepared to prove the main result of this section, which establishes existence of affine processes on irreducible symmetric cones for any given admissible
parameter set.

\begin{proposition}\label{th:existence}
Let $(\alpha, b, B, c, \gamma, m, \mu)$ be an admissible
parameter set. Then there exists a unique affine process on $K$,
such that~\eqref{eq:affineprocessK} holds for all $(t,u) \in \re_+ \times K$, where
$\phi(t,u)$ and $\psi(t,u)$ are given by~\eqref{eq:RiccatiF} and~\eqref{eq:RiccatiR}.
\end{proposition}

\begin{proof}
By Lemma~\ref{lem: splitting}, we have for each fixed $t$
\[
e^{-\phi(t,u)-\langle \psi(t,u),x \rangle}=\lim_{N \rightarrow \infty}e^{-w_N(u) -\langle y_N(u),x \rangle}, \quad u \in \mathring{K}.
\]
For each $N \in \mathbb{N}$, $n \in \{0, \ldots,N\}$ and $x \in K$, $u \mapsto e^{-w_n(u) -\langle y_n(u),x \rangle}$
is the Laplace transform of a probability distribution on $K$. Indeed, let us proceed by induction.
For $n=0$, $e^{-\langle u, x\rangle}$ is the Laplace transform of $\delta_x(d\xi)$. We now suppose that for every
$x\in K$, $e^{-w_{n-1} -\langle y_{n-1},x \rangle}$ is the Laplace transform of a probability distribution $\mu_{n-1}(x, \cdot)$ on $K$.
Due to Proposition~\ref{lem:noncentral} and Proposition~\ref{prop: C CS semiflow},
$u \mapsto e^{-\phi_i(\tau,u)-\langle \psi_i(\tau,u),x \rangle},$ $i=1,2$, are Laplace transforms of probability measures
supported on $K$, which we denote by $p_{\tau}^i(x, d\xi), \, i=1,2$. Since we have
\begin{align*}
e^{-w_n -\langle y_n,x \rangle}&=e^{-w_{n-1}}e^{-\phi_1(\tau, y_{n-1})-\phi_2(\tau, \psi_1(\tau, y_{n-1}))-\langle \psi_2(\tau, \psi_1(\tau, y_{n-1})) ,x \rangle}\\
&=\int_K \int_K e^{-w_{n-1}-\langle y_{n-1}, \widetilde{\xi}\rangle}p_{\tau}^1(\xi, d\widetilde{\xi})p_{\tau}^2(x, d\xi),\\
&=\int_K \int_K \int_K e^{-\langle u,z\rangle}\mu_{n-1}(\widetilde{\xi},dz) p_{\tau}^1(\xi, d\widetilde{\xi})p_{\tau}^2(x, d\xi),\\
\end{align*}
$e^{-w_n -\langle y_n,x \rangle}$ is the Laplace transform of the probability distribution given by
\[
\mu_n(x,\cdot)=\int_K\int_K \mu(\widetilde{\xi},\cdot) p_{\tau}^1(\xi, d\widetilde{\xi})p_{\tau}^2(x, d\xi).
\]
As $u \mapsto e^{-\phi(t,u)-\langle \psi(t,u),x \rangle}$ is continuous at $0$
and the limit of a sequence of Laplace transforms of probability distributions supported on $K$,
L\'evy's continuity theorem implies that $u \mapsto e^{-\phi(t,u)-\langle \psi(t,u),x \rangle}$
is also the Laplace transform of a probability distribution on $K$.

Moreover, the Chapman-Kolmogorov equation holds in view of the semi-flow property of $\phi$ and $\psi$, which implies the assertion.
\end{proof}

\section{Boundary Non-Attainment on Symmetric Cones}\label{sec:nonboundary}

In this section, $X$ is a conservative affine process with admissible parameters $(\alpha,b,B,m,\mu,c,\gamma)$
relative to a truncation function $\chi$ such that
\begin{equation}\label{eq: intcond}
 \int_K (\|\xi\|\wedge 1)\langle \mu(d\xi),x \rangle <\infty.
\end{equation}
for all $x \in K$. The conservativeness implies that $c=0,\gamma=0$, and for each $x$, $X_t\circ \mathbb P_x$ allows a modification
to a a c\`adl\`ag semimartingale with decomposition
\[
X_t=X_0+X^c+\sum_{s\leq t}\Delta X_s,
\]
with $X^c=M^c+B^c$, where $M^c$ is the continuous martingale part of $X$ and $B^c=
\int_0^{\cdot} (b+B_0(X_s))ds$. Here $B_0$ denotes the modified linear drift
$B_0(\cdot):=B(\cdot)-\int_K \chi(\xi) \langle \mu(d\xi), \cdot\rangle$. The admissibility condition~\eqref{eq: betaij}
implies that $B_0$ is inward pointing, that is,
\begin{equation}\label{eq: inwpointing0}
\langle B_0(x),u\rangle\geq 0\textrm{  for all  } u,x\textrm{  with  } \langle u, x\rangle=0.
\end{equation}

\begin{proposition}\label{prop:boundary}
Let $X$ be a conservative affine process on an irreducible symmetric cone $K$ such that~\eqref{eq: intcond} is satisfied.
If $b\succeq(d(r-1)+2)\alpha$, then for each $x\in \mathring{K}$, we have $\mathbb{P}_x$-a.s.
\[
T_x:=\inf\{t>0\mid X_{t-}\notin
\mathring{K}\}=\infty.
\]
\end{proposition}

\begin{proof}
Since $\Delta X_t\succeq 0$ for $t<T_x$, it follows that $T_x=\inf\{t>0\mid \det(X^c_{t_-})=0\}$. In other words,
if $X_t$ touches the boundary in finite time, it must diffuse thereto.

For $x\in \mathring{K}$, we have by Proposition~\ref{prop:EJA} \ref{item:EJA3} and \ref{item:EJA2}
\begin{align}
D_u \ln(\det(x))&=\frac{1}{\det(x)}\langle \det(x) x^{-1},u\rangle=\langle x^{-1},u\rangle,\\
D_u D_v \ln(\det(x))&=-\langle  P(x^{-1})(v),u\rangle.
\end{align}
Hence, by an application of It\^o's formula, we have for $t<T_x$
\begin{align*}
 d\ln(\det(X_t))&=\left\langle X_t^{-1},(b+B_0(X_t))dt+dM_t^c)\right\rangle-\frac{1}{2}\Tr\left(A(X_t)P(X_t^{-1})\right)dt\\
 &\quad +\ln\left(\frac{\det (X_t)}{\det (X_{t-})}\right).
\end{align*}
By Lemma \ref{lem:Trace}, we further obtain
\[
\left\langle X_t^{-1},(b+B_0(X_t))\right\rangle  +\frac{1}{2}\Tr\left(A(X_t)
 P(X_t^{-1})\right)=\left\langle b-2\frac{n}{r}\alpha, X_t^{-1}\right\rangle+\left\langle B_0(X_t), X_t^{-1}\right\rangle.
\]
Now $\widetilde M^c_t:=\int_0^t\langle X_s^{-1}, dM_s^c\rangle$ is a one dimensional continuous local martingale on the
stochastic interval $[0,T_x)$. Also, by~\eqref{eq: inwpointing0} and since $\langle \nabla \det(x), x\rangle=0$ at
$x\in\partial K$,
\[
\left\langle B_0(X_t),X_t^{-1}\right\rangle=\frac{\langle B_0(X_t),\nabla\det(X_t)\rangle}{\det(X_t)}
\]
must be bounded from below, along any path on the closed interval $[0, T_x]$. On the other hand,
since
\[
\frac{\det(X_t)}{\det(X_{t-})}=\det\left(e+P(\sqrt{\Delta X_t})X_{t-}^{-1}\right) \geq 1
\]
and by the assumption of the proposition (recall that $n=r+\frac{d}{2}r(r-1)$), we have that
\[
\left\langle b-2\frac{n}{r}\alpha, X_t^{-1}\right\rangle+\ln\left(\frac{\det (X_t)}{\det (X_{t-})}\right)\geq 0
\]
for $t<T_x$. All in all, we have that
\[
\widetilde P_t:=\int_0^t \left(\left\langle b-2\frac{n}{r}\alpha, X_s^{-1}\right\rangle+\left\langle B_0(X_s),
X_s^{-1}\right\rangle\right)ds+\sum_{s\leq t}\ln\left(\frac{\det(X_s)}{\det(X_{s-})}\right)
\]
is (pathwise) bounded from below on $[0, T_x]$ and
\[
\ln(\det(X_t))=\ln(\det(x))+\widetilde M^c_t+\widetilde P_t.
\]
But $\ln(\det(X_t))\rightarrow-\infty$ for $t\rightarrow T_x$. By McKean's argument (see~\citet[Section 4.1]{mps}) we therefore
must have $\mathbb{P}_x$-almost surely $T_x=\infty$.
\end{proof}

\appendix

\section{An Introduction to Euclidean Jordan Algebras}\label{appendix:EJA}

In this section $(V, \langle \cdot, \cdot \rangle)$ denotes an Euclidean Jordan algebra and $K$
the corresponding symmetric cone of squares, as introduced in
Definition~\ref{def:EJA}, Definition~\ref{def:symcone} and Theorem~\ref{th:symconeEJA}.
The aim of this appendix is to review some important notions and results related to Euclidean Jordan algebras.

\subsection{Determinant, Trace and Inverse}\label{sec:det}

We denote by $\mathbb{R}[\lambda]$ the polynomial ring over $\mathbb{R}$ in a single variable $\lambda$ and, for $x \in V$,
we define $\re[x] := \{p(x)\, |\, p \in \re[\lambda]\}$. Then we have
\[
\re[x] = \re[\lambda]/\mathcal{J} (x)
\]
with the ideal $\mathcal{J} (x) := \{p \in \re[\lambda] \,|\, p(x) = 0\}$. Since $\re[\lambda]$ is a principal ring, $\mathcal{J}(x)$ is
generated by a polynomial, which is referred to as the \emph{minimal polynomial} if the leading coefficient is $1$.
Its degree is denoted by $m(x)$. We have
\[
m(x) = \min\left\{k > 0 \,|\, e, x, x^2, \ldots, x^k \textrm{ are linearly dependent} \right\}.
\]
Furthermore, the \emph{rank} of $V$ is the number
\begin{equation}\label{Eq:rank}
r := \max_{x\in V} m(x),
\end{equation}
which is bounded by $n = \dim(V)$. An element $x$ is said to be regular if $m(x)=r$.
By~\citet[Proposition II.2.1]{Faraut1994}, there exist unique polynomials $a_1, \ldots, a_r$ on $V$ such that the minimal polynomial of every regular element $x$ is given by
\[
f(\lambda; x) = \lambda^r - a_1(x)\lambda ^{r-1} + a_2(x)\lambda^{r-2} + \cdots + (-1)^r a_r(x).
\]
Using this fact, one introduces, for any $x \in V$, the \emph{determinant} $\det(x)$ and \emph{trace} $\tr(x)$ as
\[
\det(x) := a_r(x) \quad \textrm{and} \quad  \tr(x) := a_1(x).
\]

\begin{remark}\label{rem:dettrace}
In order to distinguish between elements of $V$ and linear maps on $V$,
we use the notations $\Tr(A)$ and $\Det(A)$ for $A \in \mathcal{L}(V)$.
\end{remark}

An element $x$ is said to be \emph{invertible} if there exists an element $u \in \re[x]$ such that $x \circ y=e$. Since $\re[x]$
is associative, $y$ is unique. It is called the \emph{inverse} of $x$ and is denoted by $y=x^{-1}$.

\begin{remark}
We remark that the notions ``rank'', ``trace'' and ``determinant'' are motivated by the fact that
for the Euclidean Jordan algebra of real
$r \times r$ symmetric matrices, the rank is equal to $r$, and $a_r(x)$ and $a_1(x)$
are the usual determinant and trace, respectively.
The inverse is also the usual one.
\end{remark}

Note that the determinant is not multiplicative in general, but we have the following.

\begin{proposition}\label{prop:detmult}
For all $z \in V$ and $x,y \in \re[z]$,
\[
 \det(x \circ y)=\det(x)\det(y).
\]
Moreover, $\det(e)=1$ and $\tr(e)=r$. In particular,
\[
 \det(x)\det(x^{-1})=\det(x\circ x^{-1})=1.
\]
\end{proposition}

\subsection{Idempotents, Spectral-- and Peirce Decomposition}\label{sec:Peirce}

An element $p$ of $V$ is called \emph{idempotent} if $p^2 = p$, and two idempotents $p,q$ are called orthogonal if $p \circ q = 0$.
Note that for idempotents this notion of orthogonality coincides with the notion of orthogonality
with respect to the scalar product $\langle \cdot, \cdot \rangle$ on $V$ (see Lemma~\ref{Lem:orthogonality}). Finally, a non-zero idempotent element is called \emph{primitive} if it cannot be expressed as a sum of non-zero orthogonal idempotents.

The following results are cornerstones of the theory of Jordan algebras:

\begin{description}
\item[Spectral Decomposition]
A set of mutually orthogonal primitive idempotents $p_1, \ldots, p_r$ such that $p_1 + \dotsm + p_r = e$
is called a \emph{Jordan frame} of rank $r$ corresponding to the rank of
the Euclidean Jordan algebra, as defined in~\eqref{Eq:rank}.
The spectral decomposition theorem (see~\citet[Theorem III.1.2]{Faraut1994})
states that, for every $x \in V$, there exists a Jordan frame $p_1, \ldots, p_r$ and real numbers $\lambda_1,\ldots, \lambda_r$,
such that
\begin{align*}
x = \sum_{k=1}^r{\lambda_k p_k},
\end{align*}
where the numbers $\lambda_k$ are uniquely determined by $x$.
Moreover, $x$ is an element of the symmetric cone $K$ if and only if $\lambda_k\geq 0$ for all $k\in \{1, \ldots, r\}$.
Again motivated by symmetric matrices, $\lambda_1, \ldots, \lambda_r$ are referred to as \emph{eigenvalues}.

\item[Peirce decomposition 1]
Let $c$ be an idempotent in $V$. Define, for $k = 0,1/2,1$ the subspaces $V(c,k) := \{x \in V: c \circ x = kx\}$.
Then, by~\citet[Proposition IV.1.1]{Faraut1994}, $V$ can be written as the direct orthogonal sum
\[
V = V(c,1) \oplus V(c,1/2) \oplus V(c,0).
\]

Moreover, any $x \in V$ decomposes with respect to this decomposition into
\begin{align}\label{eq:Peircex}
 x=x_1+x_{\frac{1}{2}}+x_{0},
\end{align}
where $x_k \in V(c,k)$ for  $k = 0,1/2,1$.

By~\citet[Proposition IV.1.1]{Faraut1994}, the Peirce spaces $V(c,k)$ satisfy certain so called \emph{Peirce multiplication} rules:
\begin{align}
 V(c,1) \circ V(c,0)&=\{0\},\\
(V(c,1) +V(c,0)) \circ V(c,1/2) &\subset V\left(c,1/2\right),\\
 V\left(c,1/2\right) \circ V(c,1/2) &\subset V(c,1)+V(c,0),\\
P\left(V\left(c, 1/2\right)\right) V(c,1) &\subset V(c,0). \label{eq:Peircemult}
\end{align}

\item[Peirce decomposition 2]
Let  $p_1, \ldots, p_r$ be a Jordan frame. Then $V$ can be written as the direct orthogonal sum
\begin{equation}\label{Peirce decomposition 2}
V = \bigoplus_{i \le j} V_{ij},
\end{equation}
where $V_{ii} = V(p_i,1)=\re p_i$ and $V_{ij} = V(p_i,1/2) \cap V(p_j,1/2)$ (see~\citet[Theorem IV.2.1]{Faraut1994}).

Moreover, the projection onto $V_{ii}$ is given by the quadratic representation
$P(p_i)$, and the projection onto $V_{ij}$ by
$4 L(p_i) L(p_j)$.

If $V$ is simple, then the dimension of $V_{ij}, \, i<j$, denoted by
\begin{align}\label{eq:Peirce}
d= \dim V_{ij}
\end{align}
is independent of $i,j$ and the Jordan frame. It is called \emph{Peirce invariant}.
If $V$ is simple of dimension $n$ and rank $r$, then we have by~\citet[Corollary IV.2.6]{Faraut1994}
\[
 n=r+\frac{d}{2}r(r-1).
\]

Corresponding to the decomposition~\eqref{Peirce decomposition 2}, we can write for all $x \in V$
\[
 x=\sum_{i=1}^r x_i + \sum_{i<j} x_{ij},
\]
with $x_i \in \re p_i$ and $x_{ij} \in V_{ij}$. One can think of $x$ as a symmetric $r \times r$  matrix, whose
diagonal elements are the $x_i$ and whose off-diagonal elements are the $x_{ij}$.
\end{description}

\begin{example}

\begin{enumerate}
\item Consider $V =\re^n$ equipped with the Euclidean scalar\\ product. Together with the Jordan product
\[
x \circ y = (x_1 y_1, \dotsc, x_n y_n)
\]
(element-wise multiplication in $\re$), $V$ is a Jordan algebra and the corresponding reducible
cone is $\re_+^n$.
The idempotents are vectors consisting only of zeros and ones.
The non-zero primitive idempotents are the unit vectors.
The spectral decomposition $(\lambda_1, \dots, \lambda_n)$ are simply the coordinates of $x$ in the Cartesian coordinate system.

\item Consider $V = S_r$, the space of real symmetric $r \times r$-matrices.
Here, the idempotents correspond to the orthogonal projections and the non-zero primitive idempotents
are the orthogonal projections on one-di\-mens\-ional subspaces.
In the spectral decomposition,  $\lambda_1, \ldots, \lambda_r$ are the usual eigenvalues of $x$ and $p_1, \ldots, p_r$ are the orthogonal projections on the corresponding eigenvectors.

Concerning Peirce decomposition 1, the matrix of block form
\[
\begin{pmatrix} I_k & 0 \\ 0 & 0 \end{pmatrix}
\]
is an idempotent of $V$. The associated Peirce decomposition of a symmetric matrix is then given by
\[
\begin{pmatrix}
x_1 & x_{12}^\top \\ x_{12} & x_0 \end{pmatrix} =
\begin{pmatrix} x_1 & 0 \\ 0 & 0 \end{pmatrix}
+ \begin{pmatrix} 0 & x_{12}^\top \\ x_{12} & 0 \end{pmatrix} + \begin{pmatrix} 0 & 0 \\ 0 & x_0 \end{pmatrix} .
\]

As already established above, Peirce decomposition 2 corresponds to
\[
 x=\sum_{i=1}^r \begin{pmatrix}
\ddots&&&\\
&\ddots &&&\\
 && x_i && \\
&& &\ddots&\\
&&&&\ddots
\end{pmatrix} + \sum_{i<j} \begin{pmatrix}
&&&&  \text{\reflectbox{$\ddots$}}\\
 & & & x_{ij} &\\
 & &\text{\reflectbox{$\ddots$}}& & \\
 &x_{ij}&&& \\
\text{\reflectbox{$\ddots$}} &&&
\end{pmatrix}.
\]
Note that the dimension $d$ of $V_{ij}, \, i<j$, is $1$ here.
\end{enumerate}
\end{example}

\subsection{Classification of Simple Euclidean Jordan Algebras}

We here state the classification of all simple Euclidean Jordan algebras and their corresponding irreducible cones.
This classification is summarized in the following table.
Indeed, every simple Euclidean Jordan algebra is isomorph to one of these cases.

Here, $\mathbb{H}$ and $\mathbb{O}$ denote the algebra of quaternions and octonions, respectively (see~\citet[page 84]{Faraut1994}).
We further denote by $\Herm(r,\mathbb{A})$ the real vector space of Hermitian matrices with entries in $\mathbb{A}$, where $\mathbb{A}$
corresponds either to $\mathbb{C}, \mathbb{H}$ or $\mathbb{O}$.

\begin{center}
\begin{tabular}{ccccc}
\hline
$K$ & $V$ & $\dim V$ & $\rk V$ & $d$ \\
\hline
\\
 $S_r^+$ &  $S_r$ & $\frac{1}{2}r(r+1)$ & $r$  & $1$  \\
\\
$ \Herm_+(r,\mathbb{C})$ & $\Herm(r,\mathbb{C})$ & $r^2$ & $r$  & $2$  \\
\\
$ \Herm_+(r,\mathbb{H})$ & $\Herm(r,\mathbb{H})$  & $r(2r-1)$ & $r$  & $4$  \\
\\
Lorentz cone & $\mathbb{R}$ $\times$ $\mathbb{R}^{n-1}$  & $n$ & $2$  & $n-2$  \\
\\
Exceptional cone & $\Herm(3,\mathbb{O})$ & $27$ & $3$  & $8$  \\
\hline
\end{tabular}
\end{center}

\subsection{{Additional Results}}

In this section we collect a number of lemmas and propositions which are used in the proofs of Section~\ref{sec:symcone}.
In most cases, we only cite the assertions without proofs, as they can be found in~\citet{Faraut1994}.
For the sake of notational convenience we always assume that $V$ is a simple Euclidean Jordan algebra of dimension $n$ and rank $r$,
equipped with the natural scalar product
\[
\langle \cdot, \cdot \rangle : V \times V \to \re, \quad \langle x, y\rangle := \tr(x \circ y).
\]
However, the particular form of the scalar product and the assumption that $V$ is simple is not always needed.

\begin{lemma}\label{Lem:orthogonality}
\begin{enumerate}
\item Let $a,b$ be idempotents in $V$. Then $\langle a,b\rangle \geq 0$. Moreover, $\langle a, b\rangle = 0$ if and only if $a\circ b = 0$.
\item Let $a,b \in K$ and suppose that $\langle a, b\rangle = 0$. Then $a \circ b = 0$.
\end{enumerate}
\end{lemma}

\begin{proof}
For (i) see~\citet{Nomura1993} and for (ii)~\citet{Hertneck1962}.
\end{proof}

\begin{proposition}\label{prop:EJA}
The following assertions hold true:
\begin{enumerate}
\item\label{item:EJA1} An element $x$ is invertible if and only if  $P(x)$ is invertible. Then we have
\begin{align*}
P(x)x^{-1}&=x,\\
P(x)^{-1}&=P(x^{-1}).
\end{align*}
\item\label{item:EJA5} If $x$ and $y$ are invertible, then $P(x)y$ is invertible and
\[
 (P(x)y)^{-1}=P(x^{-1})y^{-1}.
\]
\item\label{item:EJA2} The differential map $x \mapsto x^{-1}$ is $-P(x)^{-1}$, that is,
\begin{align*}
\frac{d}{dt}(x+tu)^{-1}|_{t=0}=-P(x^{-1})u.
\end{align*}
\item\label{item:EJA4}
$\det(P(x)y)=(\det x)^2\det y$.
\item\label{item:EJA3}
$\nabla \ln \det x=x^{-1}$.
\end{enumerate}
\end{proposition}

\begin{proof}
See~\citet[Proposition II.3.1, II.3.3 (i), II.3.3 (ii), III.4.2 (i), III.4.2 (ii)]{Faraut1994}
\end{proof}

\backmatter

\begin{thebibliography}{29}
\providecommand{\natexlab}[1]{#1}
\providecommand{\url}[1]{\texttt{#1}}
\expandafter\ifx\csname urlstyle\endcsname\relax
  \providecommand{\doi}[1]{doi: #1}\else
  \providecommand{\doi}{doi: \begingroup \urlstyle{rm}\Url}\fi

\bibitem[Aliprantis and Tourky(2007)]{aliprantis_07}
C.~D. Aliprantis and R.~Tourky.
\newblock \emph{Cones and duality}, volume~84 of \emph{Graduate Studies in
  Mathematics}.
\newblock American Mathematical Society, Providence, RI, 2007.

\bibitem[Bru(1991)]{bru}
M.-F. Bru.
\newblock Wishart processes.
\newblock \emph{J. Theoret. Probab.}, 4\penalty0 (4):\penalty0 725--751, 1991.

\bibitem[Cuchiero and Teichmann(2011)]{cuchteich}
C.~Cuchiero and J.~Teichmann.
\newblock Path properties and regularity of affine processes on general state
  spaces.
\newblock Preprint, 2011.

\bibitem[Cuchiero et~al.(2011)Cuchiero, Filipovi{\'c}, Mayerhofer, and
  Teichmann]{cfmt}
C.~Cuchiero, D.~Filipovi{\'c}, E.~Mayerhofer, and J.~Teichmann.
\newblock Affine processes on positive semidefinite matrices.
\newblock \emph{Ann. Appl. Probab., Forthcoming}, 2011.

\bibitem[Dieudonn{\'e}(1969)]{dieu_69}
J.~Dieudonn{\'e}.
\newblock \emph{Foundations of modern analysis}.
\newblock Academic Press, New York, 1969.
\newblock Enlarged and corrected printing, Pure and Applied Mathematics, Vol.
  10-I.

\bibitem[Duffie et~al.(2003)Duffie, Filipovi{\'c}, and Schachermayer]{dfs}
D.~Duffie, D.~Filipovi{\'c}, and W.~Schachermayer.
\newblock Affine processes and applications in finance.
\newblock \emph{Ann. Appl. Probab.}, 13\penalty0 (3):\penalty0 984--1053, 2003.

\bibitem[Faraut and Kor{\'a}nyi(1994)]{Faraut1994}
J.~Faraut and A.~Kor{\'a}nyi.
\newblock \emph{Analysis on symmetric cones}.
\newblock Oxford Mathematical Monographs. The Clarendon Press Oxford University
  Press, New York, 1994.
\newblock Oxford Science Publications.

\bibitem[Grasselli and Tebaldi(2008)]{grasselli}
M.~Grasselli and C.~Tebaldi.
\newblock Solvable affine term structure models.
\newblock \emph{Math. Finance}, 18\penalty0 (1):\penalty0 135--153, 2008.

\bibitem[Hairer et~al.(1993)Hairer, N{\o}rsett, and Wanner]{hairer}
E.~Hairer, S.~P. N{\o}rsett, and G.~Wanner.
\newblock \emph{Solving ordinary differential equations. {I}}, volume~8 of
  \emph{Springer Series in Computational Mathematics}.
\newblock Springer-Verlag, Berlin, second edition, 1993.
\newblock Nonstiff problems.

\bibitem[Hertneck(1962)]{Hertneck1962}
C.~Hertneck.
\newblock Positivit{\"a}tsbereiche und {J}ordan-{S}trukturen.
\newblock \emph{Math. Ann.}, 146:\penalty0 433--455, 1962.

\bibitem[Hiriart-Urruty and Lemar{\'e}chal(1993)]{hiriart}
J.-B. Hiriart-Urruty and C.~Lemar{\'e}chal.
\newblock \emph{Convex analysis and minimization algorithms. {I}}, volume 305
  of \emph{Grundlehren der Mathematischen Wissenschaften [Fundamental
  Principles of Mathematical Sciences]}.
\newblock Springer-Verlag, Berlin, 1993.
\newblock Fundamentals.

\bibitem[Ishi(2005)]{Ishi_gradient}
H.~Ishi.
\newblock The gradient maps associated to certain non-homogeneous cones.
\newblock \emph{Proc. Japan Acad. Ser. A Math. Sci.}, 81\penalty0 (3):\penalty0
  44--46, 2005.

\bibitem[Keller-Ressel(2009)]{keller}
M.~Keller-Ressel.
\newblock Affine processes - theory and applications in mathematical finance.
\newblock PhD thesis, {V}ienna { U}niversity of {T}echnology, 2009.

\bibitem[Keller-Ressel et~al.(2010)Keller-Ressel, Schachermayer, and
  Teichmann]{kst}
M.~Keller-Ressel, W.~Schachermayer, and J.~Teichmann.
\newblock Affine processes are regular.
\newblock \emph{Probab. Theory Related Fields, Forthcoming}, 2010.

\bibitem[Keller-Ressel et~al.(2011)Keller-Ressel, Schachermayer, and
  Teichmann]{kst1}
M.~Keller-Ressel, W.~Schachermayer, and J.~Teichmann.
\newblock Regularity of affine processes on general state spaces.
\newblock Working paper, 2011.

\bibitem[Letac and Massam(2004)]{letac}
G.~Letac and H.~Massam.
\newblock A tutorial on non-central {W}ishart distributions.
\newblock 2004.

\bibitem[L{\'e}vy(1948)]{levy}
P.~L{\'e}vy.
\newblock The arithmetic character of the {W}ishart distribution.
\newblock \emph{Proc. Cambridge Philos. Soc.}, 44:\penalty0 295--297, 1948.

\bibitem[Massam and Neher(1997)]{massam}
H.~Massam and E.~Neher.
\newblock On transformations and determinants of {W}ishart variables on
  symmetric cones.
\newblock \emph{J. Theoret. Probab.}, 10\penalty0 (4):\penalty0 867--902, 1997.

\bibitem[Mayerhofer(2011)]{mayerhoferpers}
E.~Mayerhofer.
\newblock Positive semidefinite affine processes have jumps of finite
  variation.
\newblock Preprint, 2011.

\bibitem[Mayerhofer et~al.(2011{\natexlab{a}})Mayerhofer, Muhle-Karbe, and
  Smirnov]{mayerhofer}
E.~Mayerhofer, J.~Muhle-Karbe, and A.~G. Smirnov.
\newblock A characterization of the martingale property of exponentially affine
  processes.
\newblock \emph{Stochastic Process. Appl.}, 121\penalty0 (3):\penalty0
  568--582, 2011{\natexlab{a}}.

\bibitem[Mayerhofer et~al.(2011{\natexlab{b}})Mayerhofer, Pfaffel, and
  Stelzer]{mps}
E.~Mayerhofer, O.~Pfaffel, and R.~Stelzer.
\newblock On strong solutions for positive definite jump-diffusions.
\newblock \emph{Stochastic Process. Appl.}, 121\penalty0 (9):\penalty0
  2072--2086, 2011{\natexlab{b}}.

\bibitem[Nomura(1993)]{Nomura1993}
T.~Nomura.
\newblock Manifold of primitive idempotents in a {J}ordan-{H}ilbert algebra.
\newblock \emph{J. Math. Soc. Japan}, 45\penalty0 (1):\penalty0 37--58, 1993.

\bibitem[Sato(1999)]{sato}
K.~Sato.
\newblock \emph{L{\'e}vy processes and infinitely divisible distributions},
  volume~68 of \emph{Cambridge Studies in Advanced Mathematics}.
\newblock Cambridge University Press, Cambridge, 1999.
\newblock Translated from the 1990 Japanese original, Revised by the author.

\bibitem[Skorohod(1991)]{skorohod}
A.~V. Skorohod.
\newblock \emph{Random processes with independent increments}, volume~47 of
  \emph{Mathematics and its Applications (Soviet Series)}.
\newblock Kluwer Academic Publishers Group, Dordrecht, 1991.
\newblock Translated from the second Russian edition by P. V. Malyshev.

\bibitem[Spreij and Veerman(2010)]{spreij1}
P.~Spreij and E.~Veerman.
\newblock Affine diffusions with non-canonical state space.
\newblock Preprint, 2010.

\bibitem[Veerman(2011)]{veerman}
E.~Veerman.
\newblock Affine {M}arkov processes on a general {E}uclidean state space.
\newblock PhD Thesis, 2011.

\bibitem[Vinberg(1960)]{vinberg}
{\`E}.~B. Vinberg.
\newblock Homogeneous cones.
\newblock \emph{Soviet Math. Dokl.}, 1:\penalty0 787--790, 1960.

\bibitem[Volkmann(1973)]{Volkmann1973}
P.~Volkmann.
\newblock {\"U}ber die {I}nvarianz konvexer {M}engen und
  {D}ifferentialungleichungen in einem normierten {R}aume.
\newblock \emph{Math. Ann.}, 203:\penalty0 201--210, 1973.

\bibitem[Walter(1993)]{walter}
W.~Walter.
\newblock \emph{Gew{\"o}hnliche {D}ifferentialgleichungen}.
\newblock Springer-Lehrbuch. [Springer Textbook]. Springer-Verlag, Berlin,
  fifth edition, 1993.
\newblock Eine Einf{\"u}hrung. [An introduction].

\end{thebibliography}

\end{document}